\documentclass[11pt]{article}
\usepackage{amsmath,amssymb,graphicx}
\usepackage{subcaption}
\usepackage{amsfonts, amsmath, amssymb,latexsym,amsthm, mathrsfs, mathtools}
\usepackage[margin=0.9in]{geometry}
\usepackage{url}
\usepackage[english]{babel}
\usepackage{epsfig}
\usepackage{pifont}
\usepackage{tikz-cd}
\usepackage{authblk}

\newtheorem{thm}{Theorem}[section]
\newtheorem*{thm*}{Theorem}
\newtheorem*{prop*}{Proposition}
\newtheorem*{corol*}{Corollary}
\newtheorem*{lemma*}{Lemma}
\newtheorem{prop}[thm]{Proposition}
\newtheorem{corol}[thm]{Corollary}
\newtheorem{lemma}[thm]{Lemma}
\usepackage[all]{xy}
\newtheorem*{claim*}{Claim}
\newtheorem*{question*}{Question}

\usepackage{chngcntr}
\counterwithin{figure}{section}

\usepackage{enumitem}
\theoremstyle{definition}
\newtheorem{defi}[thm]{Definition}
\newtheorem*{defi*}{Definition}

\newtheorem{remark}[thm]{Remark}

\newtheorem*{remark*}{Remark}
\newtheorem*{remarks*}{Remarks}

\theoremstyle{plain}
\newtheorem*{namedthm}{\namedthmname}
\newcounter{namedthm}
\makeatletter
\newenvironment{named}[1]
{\def\namedthmname{#1}%
	\refstepcounter{namedthm}%
	\namedthm\def\@currentlabel{#1}}
{\endnamedthm}
\makeatother

\usepackage[maxbibnames=99, backend=bibtex, style=alphabetic]{biblatex}
\usepackage{caption}
\usepackage{subcaption}

\usepackage{fancyhdr}

\usepackage[hidelinks]{hyperref}

\title{Periodic boundary points for simply connected Fatou components of transcendental maps}
\author[1]{Anna Jov\'e\thanks{This work is supported by the Spanish government grant FPI PRE2021-097372 and PID2020-118281GB-C32. Corresponding author. \url{ajovecam7@alumnes.ub.edu}.}}
\affil[1]{\small Departament de Matemàtiques i Informàtica, Universitat de Barcelona, Barcelona, Spain}

\begin{document}
\maketitle
\begin{abstract}
Let $ f\colon\mathbb{C}\to\widehat{\mathbb{C}} $ be a transcendental map, and let $ U $ be an attracting or parabolic basin, or a doubly parabolic Baker domain. Assume $ U $ is simply connected. Then, we prove that periodic points are dense in $ \partial U $, under certain hypothesis on the postsingular set. This generalizes a result by F. Przytycki and A. Zdunik for rational maps \cite{przytycki_zdunik}. Our proof uses techniques from measure theory, ergodic theory, conformal analysis, and inner functions. In particular, a result on the distortion of inner functions near the unit circle is provided, which is of independent interest.
\end{abstract}

\section{Introduction}

Chaos is a key concept related to the complexity of a topological dynamical system, which is in turn, and according to any definition of chaos, strongly connected with the density of periodic points. In this paper, we investigate density of periodic points on invariant subset of the phase space, namely the boundary of stable regions like, for example basins of attraction.

Throughout this paper, we consider the discrete  dynamical system generated by  a holomorphic map $ f \colon\mathbb{C}\to\widehat{\mathbb{C}} $, i.e. we study the sequences of iterates $ \left\lbrace f^n(z)\right\rbrace _n $, where $ z\in{\mathbb{C}} $. We assume $ f $ is neither constant nor Möbius. If infinity is an essential singularity of $ f $, then we say that $ f $ is a \textit{transcendental meromorphic map}; otherwise $ f $ extends holomorphically to $ \widehat{\mathbb{C}} $ as a rational map. These dynamical systems arise naturally, for example, from the popular Newton’s root-finding method applied to entire functions. For general background on the iteration of meromorphic maps, we refer to \cite{bergweiler}. 

In this situation, the complex plane, regarded as the phase space of the dynamical system, is divided into two totally invariant sets: the {\em Fatou set} $ \mathcal{F}(f) $,  the set of points $ z\in\mathbb{C} $ such that $ \left\lbrace f^n\right\rbrace _{n\in\mathbb{N}} $ is well-defined and forms a normal family in some neighbourhood of $ z$; and the {\em Julia set} $ \mathcal{J}(f) $, its complement, where the dynamics is chaotic.  In particular, periodic points are dense in $ \mathcal{J}(f) $.  Indeed, for rational maps, this was already established by Fatou and Julia in the beginning of the twentieth century. Baker \cite[Thm. 1]{Baker_FixedPoints} proved that this holds for transcendental entire functions as well, relying on a deep theorem of Ahlfors, a proof that was later generalized to transcendental meromorphic functions \cite[Thm. 1]{BakerKotusLu_I}.  The goal of this paper is to study density of periodic points in appropriate invariant subsets of the Julia set, which will be determined by means of the Fatou set.

More precisely, the Fatou set is open and consists in general of infinitely many components, which are called \textit{Fatou components}. Due to the invariance of the Fatou and Julia sets, Fatou components are either periodic, preperiodic or wandering. Periodic Fatou components are classified into \textit{rotation domains} (\textit{Siegel disks} or \textit{Herman rings}), \textit{attracting} and \textit{parabolic basins}, and \textit{Baker domains}, being the latter exclusive of transcendental functions.

Observe that, if $ U $ is a $ p $-periodic Fatou component, then $ \left\lbrace f^n(\partial U)\right\rbrace _{n=0}^{p-1} $ is a closed invariant subset of the Julia set. Hence, we ask the following question.

\begin{question*}
	Let $ U $ be a periodic Fatou component. Are periodic points dense in $ \partial U $?
\end{question*}

Note that although periodic points are dense in the Julia set, {\em a priori} they could  accumulate on $ \partial U $ only from the complement of $ \overline{U} $, without being in $ \partial U $. For instance, if $ U $ is a rotation domain with locally connected boundary, then there are no periodic points in $ \partial U $ at all. Nevertheless, F. Przytycki and A. Zdunik gave a positive answer to the question for rational maps, excluding Fatou components which are rotation domains.

\begin{thm*}{\bf (Przytycki-Zdunik, {\normalfont\cite{przytycki_zdunik}})}
	Let $ f\colon\widehat{\mathbb{C}}\to\widehat{\mathbb{C}} $ be a rational map, and let $ U $ be an attracting or parabolic basin for $ f $. Then, periodic points are dense in $ \partial U $.
\end{thm*}

In the seminal paper \cite{przytycki_zdunik}, two different proofs are provided: one for simply connected attracting basins, and a general one, which works in the non-simply connected or parabolic situations. The latter relies on a  technique, known as geometric coding trees, which has been shown not to work well in the infinite degree case, even for hyperbolic maps (see \cite[p. 405]{BaranskiKarpinska_GeometricCodingTree}).

Their proof relies on three specific features of rational maps: $ f $ having finitely many {\em singular values} (points on which some branch of $ f^{-1} $ is not locally well-defined, see Sect. \ref{subsect-regular-singular}), $ f $ extending analytically to the boundary of  $ U $ (taken in $ \widehat{\mathbb{C}} $),  and $ f $ having finite degree. Note that these three assumptions are no longer satisfied for transcendental meromorphic maps. Indeed, a transcendental  map $ f $ can have infinitely many singular values, and it may have essential singularities on the boundary of  $ U $, implying that $ f|_{\partial U} $ is no longer analytic. In addition, $ f|_{ U} $ may have infinite degree. 
Moreover, when dealing with transcendental meromorphic functions, one encounters other new challenges, namely a new type of Fatou components (Baker domains, on which iterates accumulate on the essential singularity), and  the presence of asymptotic values.

The goal of this paper is to show that, by imposing some mild assumptions on the 
\textit{postsingular set} of $ f $\[ P(f)\coloneqq \overline {\bigcup\limits_{s\in SV(f)}\bigcup\limits_{n\geq 0} f^n(s)},\] where $ SV(f) $ denotes the set of singular values of $ f $, we are able to overcome these challenges, and prove the density of periodic boundary points.
Indeed, we show the following. (Recall that, given a simply connected domain $ U $, $ C\subset U $ is a {\em crosscut}  if $ C $ is a Jordan arc such that $ \overline{C}= C\cup \left\lbrace a,b\right\rbrace  $, with $ a,b\in\partial U $, $ a\neq b $; any of the two connected components of $ U\smallsetminus C $ is a {\em crosscut neighbourhood}.)

\begin{named}{Theorem A}\label{teo:A} {\bf (Periodic boundary points are dense)} 
	Let $ f\colon\mathbb{C}\to\widehat{\mathbb{C}} $ be a  meromorphic function, and let $ U $ be a periodic simply connected Fatou component for $ f $.  
	
	\noindent Assume the following conditions are satisfied.
	\begin{enumerate}[label={\em (\roman*)}]
		\item\label{maini} $ U $ is either an attracting basin, a parabolic basin, or a doubly parabolic Baker domain, with $ f|_{\partial U} $ recurrent.
		\item\label{mainii} There exists $ x\in\partial U $ and $ r>0 $ such that $ P(f)\cap D(x,r)=\emptyset $.
		\item\label{mainiii} There exists a crosscut neighbourhood $ N \subset U$ with $ P(f)\cap N =\emptyset$. 
	\end{enumerate}
	Then, periodic points are dense in $ \partial U $.
\end{named}

Next, we shall outline the main steps on the proof of \cite{przytycki_zdunik}, and explain how the new difficulties that appear for transcendental maps are overcome, showing at the same time the need for the hypotheses of the theorem.  For simplicity, we shall assume that the Fatou component $ U $ is invariant.

In the case of an attracting basin of a rational function, 
	Pesin's theory can be applied to prove that for $ \omega_U $-almost every $ x\in\partial U $, there exist inverse branches which are locally well-defined and contracting with respect to the Euclidean metric \cite[Lemma 1]{przytycki_zdunik} (see also \cite[Lemma 1]{PUZ91}, and \cite[Thm. 11.2.3]{PrzytyckiUrbanski}).  Crucial ingredients in this proof are the ergodic properties of $ f|_{\partial U} $, studied in \cite{przytycki_attracting, przytycki_attracting2}, together with  $ f|_{\partial U} $ being analytic and the finitude of critical values. None of the previous conditions is satisfied for a general transcendental meromorphic map, so \textit{a priori} Pesin's theory cannot be applied in our situation. We solve this by assuming that inverse branches are well-defined $ \omega_U $-almost everywhere (this is a straightforward consequence of {\em\ref{mainii}}), and we prove contraction of inverse branches with respect to the hyperbolic metric in a suitable domain.

Next, we  extend the proof of \cite{przytycki_zdunik} to other Fatou components, apart from attracting basins. Indeed, our proof relies only on the ergodic properties of the map $ f|_{\partial U} $, not on the precise type of Fatou component we are considering. More precisely, we only ask  $ f|_{\partial U} $ to be ergodic and recurrent with respect to the harmonic measure $ \omega_U  $, which implies that $ \omega_U $-almost every orbit in $ \partial U $ is dense in $ \partial U $.  Hence, all Fatou components for which the boundary map is ergodic and recurrent may be considered, and  these include attracting and parabolic basins, rotation domains and certain types of Baker domains (for instance, doubly parabolic Baker domains with singular values compactly contained in $ U $, Thm. \ref{thm-ergodic-boundary-map}). However, note that rotation domains never satisfy the hypothesis of our theorem, since $ P(f) $ is always dense in their boundary, and {\em\ref{mainii}} is never fulfilled.

Finally, as it is common in constructions of this kind, the proof relies strongly on the inferred dynamics in the unit disk $ \mathbb{D} $ via the Riemann map $ \varphi\colon\mathbb{D}\to U $. More precisely, let $ U $ be  a $ p $-periodic (simply connected) attracting basin of a rational map $ f $, and consider a Riemann map $ \varphi\colon\mathbb{D}\to U $. Then, the function \[g\colon\mathbb{D}\to\mathbb{D}, \hspace{0.5cm} g\coloneqq\varphi^{-1}\circ f^p\circ \varphi\] is the inner function associated to $ (f^p, U) $ by $ \varphi $ (see Sect. \ref{subsect-associated-inner-function}). 

A careful study of the associated inner function is required. 
In the case of a rational attracting basin considered in \cite{przytycki_zdunik}, $ g $ is a finite Blaschke product, which can be chosen to satisfy $ g(0)=0 $.  We shall view $ g $ as a rational map $ g\colon\widehat{\mathbb{C}}\to\widehat{\mathbb{C}} $, extended by Schwarz reflection. Then, its critical values (which are finitely many) are compactly contained in $ \mathbb{D} $ (and, by reflection, in $ \widehat{\mathbb{C}}\smallsetminus\overline{\mathbb{D}} $) and their orbits converge uniformly to 0 (or to $ \infty $), which are attracting fixed points. Hence, inverse branches of $ g $ are  well-defined for all points in $\partial \mathbb{D} $.   Moreover, precise estimates on the behaviour of such inverse branches are given in \cite[Lemma 2]{przytycki_zdunik}.

In contrast to this setting,  in the general situation we consider, $ g $ is no longer a finite Blaschke product, and may not have an attracting fixed point in $ \mathbb{D} $. However, having singular values of $ f|_U $  compactly contained in $ U $ allows us to control inverse branches for the associated inner function $ g $ at $ \lambda $-almost every point in $ \mathbb{D} $, even if $ g $ has infinite degree. Indeed, we consider the maximal meromorphic extension of $ g $:
\[g\colon\widehat{\mathbb{C}}\smallsetminus E(g)\to \widehat{\mathbb{C}},\] where $ E(g)\subset \partial\mathbb{D} $ denotes the set of singularities of $ g $ (points at which $ g $ cannot be extended analytically), and denote by $ g^*\colon\partial\mathbb{D} \to\partial\mathbb{D}$ its radial extension (see Sect. \ref{sect-iteration-inner-functions}).  In this situation, we prove the following result concerning inner functions (not necessarily associated to Fatou components), which is of independent interest. \begin{named}{Theorem B}\label{thm-inverse-inner}{\bf (Inverse branches at boundary points)}
	Let $ g\colon\mathbb{D}\to\mathbb{D} $ be an inner function, such that $ g^*|_{\partial\mathbb{D}}  $ is recurrent. Assume that there exists a crosscut neighbourhood $ N\subset\mathbb{D} $ with $ P(g)\cap N=\emptyset $. Then,
	for $ \lambda $-almost every $ \xi\in\partial\mathbb{D} $, there exists $ \rho_0\coloneqq \rho_0(\xi)>0 $ such that all  branches $ G_n$ of $ g^{-n }$ are well-defined in $ D(\xi, \rho_0) $. In particular, the set $ E(g) $ of singularities of $ g $ has $ \lambda $-measure zero.
	
		\noindent  In addition, for all $0< \alpha<\frac{\pi}{2} $, there exists $ \rho_1<\rho_0 $ such that, for all $ n\geq0 $, all branches $ G_n $ of $ g^{-n} $ are well-defined in $ D(\xi, \rho_1) $ and, if $ \gamma $ denotes the radial segment at $ \xi $, then the curve $ G_n(\gamma) $ tends   to $ G_n(\xi )$ non-tangetially with angle at most $ \alpha $.
		
\end{named}

Note that $ \alpha $ does not depend on $ n $, nor on the chosen inverse branch. Apart from giving a precise characterization of inverse branches, \ref{thm-inverse-inner} also describes measure-theoretically the set of singularities, improving the results in \cite{efjs, FatousAssociates, JoveFagella2}. Compare also with the situation for one component inner functions (a more restrictive class of inner functions) described in \cite[Part III]{ivrii2023inner}.

At this point, we shall make some additional remarks, in order to clarify and  contextualize our results. 
 
 First, one should observe that the class of (transcendental) meromorphic functions is not closed under composition: iterates $ f^n $ of a transcendental meromorphic function $ f $ have, in general, countably  many analytic singularities, so they are no longer meromorphic functions of the plane. Hence, we consider functions in class $ \mathbb{K} $, the smallest class of functions which  includes  transcendental meromorphic functions and which is closed under composition. Formally, $ f\in\mathbb{K} $ if there exists a compact countable set $ E(f) \subset\widehat{\mathbb{C}}$ such that \[f\colon \widehat{\mathbb{C}}\smallsetminus E(f)\to \widehat{\mathbb{C}}\] is meromorphic in $ \widehat{\mathbb{C}}\smallsetminus E(f) $ but in no larger set. 
	The theory of Fatou and Julia of iteration of rational maps was extended to class $ \mathbb{K} $ by Bolsch, and Baker, Domínguez and Herring \cite{ Bolsch_repulsivepp, Bolsch-thesis,Bolsch-Fatoucomponents, BakerDominguezHerring,BakerDominguezHerring2,  Dominguez4,Dominguez3}. Although more sophisticated tools are needed, the main features of iteration theory extend successfully to class $ \mathbb{K} $ (see Sect. \ref{sect-classeK}). In particular, if $ f\in \mathbb{K} $, then for any $ k\geq 1 $, $ f^k\in \mathbb{K} $ and $ \mathcal{F}(f)=\mathcal{F}(f^k) $. This allows us to reduce the study of $ k $-periodic Fatou components to the study of the invariant ones, just replacing $ f $ by $ f^k $.
	
Second, \ref{teo:A} concerns periodic Fatou components which are simply connected. It is well-known that Fatou components of meromorphic functions (and for functions in class $ \mathbb{K} $) are either simply connected, doubly connected or infinitely connected. There are plenty of examples of functions and classes of functions whose Fatou components are simply connected. For instance, periodic Fatou components of entire maps  are always simply connected \cite{Baker_SimplyConnected}. Moreover, if $ f $ is an entire function, then its Newton's method \[N_f\colon\mathbb{C}\to\widehat{\mathbb{C}},\hspace{0.5cm} N_f(z)\coloneqq z-\dfrac{f(z)}{f'(z)}\] is a meromorphic function, whose Fatou components are simply connected \cite{Taixés1, Taixés2,bfjk_connectivity,bfjk_Newton}.
	
Third, for a transcendental meromorphic function, it is not  even known whether there exists a periodic point in the boundary of every periodic Fatou component.   In particular, our result also answers this question for a wide class of Fatou components.
	
Finally, \ref{teo:A} and \ref{thm-inverse-inner} are presented in the introduction in a simplified form. 
For the stronger version of \ref{teo:A}, see Theorem \ref{thm-periodic-points-are-dense}; while for \ref{thm-inverse-inner} in its complete form, look at Theorem \ref{thm-postsingular}.

\begin{remark*}{\bf (Doubly connected Fatou components)} 
	It is well-known that periodic Fatou components of functions in class $ \mathbb{K} $ are either simply connected, doubly connected or infinitely connected. However, in contrast with the rational case, doubly connected Fatou components are not necessarily Herman rings.
Indeed,	an explicit example of an invariant doubly connected attracting basin of a holomorphic self-map of the punctured plane $ \mathbb{C}^* $ is provided in \cite[p. 545]{Bolsch-Fatoucomponents} (see also \cite[Sect. 3 and 6]{Keen}). Moreover, in 	\cite{EvdoridouMartiPeteSixsmith} an example of a doubly connected Baker domain of a holomorphic self-map of the punctured plane $ \mathbb{C}^* $
is constructed using approximation theory. More examples are provided in \cite{huang2022connectivity}. 

It is our belief that doubly connected Fatou components can be studied in a similar maner as simply connected ones. Indeed, in this case, there always exists a universal covering $ \pi\colon\mathbb{D}\to U $, which behaves locally near $ \partial\mathbb{D} $ as a conformal map.  Hence, it seems plausible that the same arguments apply to prove that periodic points are dense in the boundary of such Fatou components, under the same conditions of \ref{teo:A} (or its general version \ref{thm-periodic-points-are-dense}).
\end{remark*}

\vspace{0.2cm}
{\bf Structure of the paper.} In Section \ref{sect-prelim}, some preliminary definitions and results are collected, which are used through the article. This includes distortion estimates for univalent maps, abstract ergodic theory, boundary extension of holomorphic maps, and harmonic measure. Section \ref{sect-iteration-inner-functions} is devoted to the iteration of inner functions, including the classical Denjoy-Wolff Theorem, Cowen's classification of self-maps of the unit disk, and Doering-Mañé's results on the ergodic properties of inner functions. Next, in Section \ref{sect-teoA}, we deal with inverse branches of inner functions near the unit circle, proving \ref{thm-inverse-inner}. In Section \ref{sect-classeK} we include the Fatou and Julia theory of functions in class $ \mathbb{K} $, and the results needed for the statement and the proof of Theorem \ref{thm-periodic-points-are-dense}. Finally, Section \ref{sect-demostracio} is devoted to prove  \ref{teo:A} using all the tools  previously developed.

\vspace{0.2cm}
{\bf Notation.} Throughout this article, $ \mathbb{C} $ and $ \widehat{\mathbb{C}} $ denote the complex plane and the Riemann sphere, respectively. Given a domain $ U\subset\widehat{\mathbb{C}} $, we denote its boundary (in $ \widehat{\mathbb{C}} $) by $ \widehat{\partial} U $. We save the notation of $ {\partial} U $ to denote the boundary of $ U $ with respect to the domain of definition of the function we are considering (see Sec. \ref{sect-classeK}).

 If $ U $ is a  simply connected domain, $ \omega_U $ stands for the class of harmonic measures in $ \widehat{\partial} U $,  while  the notation $ \omega_U (z_0, \cdot)$ stands for the harmonic measure in $ \widehat{\partial} U $ with respect to $ z_0\in U $ (see Sect. \ref{subsection-harmonic-measure}). We denote by $ \mathbb{D} $, the unit disk; by $ \partial\mathbb{D} $, the unit circle; and by $ \lambda $, the Lebesgue measure on $ \partial\mathbb{D} $, normalized so that $ \lambda ( \partial\mathbb{D} )=1 $.

\vspace{0.2cm}
{\bf Acknowledgments.} I am indebted to my supervisor, Núria Fagella, for all her support and encouragement about this project. I also wish to thank Oleg Ivrii, for interesting discussions and carefully reading de manuscript, Lasse Rempe, and specially Anna Zdunik for interesting discussions and comments. Finally, I want to express my gratitude to the Complex Dynamics Group of the University of Barcelona for a series of lectures on ergodic theory, which inspired this paper.

\section{Preliminaries}\label{sect-prelim}
In this section we gather the tools we use throughout the article, including distortion estimates for univalent maps, abstract ergodic theory, boundary extension of holomorphic maps, and harmonic measure. Although all the results in this section seem to be well-known, we include the proof of those for which we could not find a written reference.

\subsection{Distortion estimates for univalent maps}

We need the following results concerning the distortion for univalent maps.

\begin{thm}{\bf (De Branges, {\normalfont\cite{bieberbach}})}\label{thm-bieberbach}
	Let $ \varphi\colon\mathbb{D}\to \mathbb{C} $ be univalent, with $ \varphi(0)=0 $ and $ \varphi'(0)=1 $. Then, \[\varphi(z)=z+\sum_{n\geq 2} a_n z^n,\] with $ \left| a_n\right| \leq n $, for $ n\geq 2 $.
\end{thm}
\begin{corol}{\bf (Distortion estimates for univalent maps)}\label{corol-distotion} Let $ \varphi\colon D(z_0, r_0) \to \mathbb{C}$ be univalent, and let $ r\in(0,r_0) $. Then, there exists $ C\coloneqq C(r, r_0)$, with $ C(r, r_0)\to 0 $ as $ \frac{r}{r_0}\to 0 $, such that, for all $ z\in D(z_0, r) $,\[ \left| \varphi(z)-L(z)\right| \leq C\left|\varphi '(z_0)\right| \left| z-z_0\right| , \] where $ L $ stands for the liner map $ L(z)\coloneqq \varphi(z_0)+\varphi'(z_0)(z-z_0) $.

\noindent In particular,  if $ \varphi $ additionally satisfies $ \varphi(0)=0 $ and $ \varphi'(0)=1 $.  Then, for all $ r\in (0,1) $,  there exists  $ C \coloneqq C(r)$, with $ C(r) \to 0$ as $ r\to 0 $, for all $ z\in D(0,r) $, \[\left| \frac{\varphi(z)}{z}-1\right|\leq C .\] 
\end{corol}
Note that, in both cases, $ C $ does not depend on the univalent map considered.
\begin{proof} Let us start by proving the particular case of $ \varphi \colon\mathbb{D}\to\mathbb{C} $, satisfying $ \varphi(0)=0 $ and $ \varphi'(0)=1 $. Then,	by Theorem \ref{thm-bieberbach}, for all $ z\in\mathbb{D} $, it holds \[\dfrac{\varphi(z)}{z}=1+\sum_{n\geq 2} a_n z^{n-1}, \]with $ \left| a_n\right| \leq n $, for $ n\geq 2 $. Hence, for $ r\in (0,1) $ and $ z\in D(0,r) $, it holds
		\[\left| \dfrac{\varphi(z)}{z}-1\right|=\left| \sum_{n\geq 2} a_n z^{n-1}\right| \leq  \sum_{n\geq 2} \left| a_n \right| r^{n-1}\leq  \sum_{n\geq 2}  n r^{n-1}\eqqcolon C(r).\] Note that the last power series converges for $ r<1 $, and $ C(r) \to 0$ as $ r\to 0 $, as desired.
		
	Now, consider any univalent map $ \varphi \colon\mathbb{D}\to\mathbb{C} $, and  let $ \psi\colon\mathbb{D}\to\mathbb{C} $ be defined as\[ \psi(w)\coloneqq \dfrac{\varphi(z_0+r_0w)-\varphi(z_0)}{r_0\varphi'(z_0)}.\]Note that $ \psi$ is univalent, and satisfies $ \psi(0)=0 $ and $ \psi'(0)=1 $. Let $ r<r_0 $ and $ \rho\coloneqq \frac{r}{r_0}<1 $. Hence, there exists $ C\coloneqq C(\rho)$, such that, for $ w\in D(0,\rho) $, \[\left| \frac{\psi(w)}{w}-1\right| \leq C.\] Letting $ z=z_0+r_0w $, we get that, for $ z\in D(z_0, r) $,\[ \dfrac{\left|\varphi(z)-(\varphi(z_0)+\varphi'(z_0)(z-z_0))\right| }{\left| z-z_0\right| \left| \varphi'(z_0)\right|} \leq C,\] as desired.
\end{proof}

\subsection{Abstract ergodic theory}
We recall some basic notions used in abstract ergodic theory (for more details, see e.g. \cite{Aaronson97,PrzytyckiUrbanski,Hawkins}).

\begin{defi}{\bf (Ergodic properties of measurable maps)} Let $ (X,\mathcal{A}, \mu) $ be a measure space, and let $ T\colon X\to X $ be measurable. Then, $ T $ is: \begin{itemize}
		\item {\em non-singular}, if for every $ A\in\mathcal{A} $, it holds $ \mu (T^{-1}(A))=0 $ if and only if  $ \mu (A)=0 $;
		\item {\em $ \mu $- preserving}, if for every $ A\in\mathcal{A} $, it holds $ \mu (T^{-1}(A))=\mu (A) $ (we also say that $ \mu $ is $ T $-invariant);
		\item {\em recurrent}, if for every  $ A\in\mathcal{A} $ and $ \mu $-almost every $ x\in A $, there exists a sequence $ n_k\to \infty $ such that $ T^{n_k}(x)\in A $;
		\item {\em ergodic}, if $ T $ is non-singular and for every $ A\in\mathcal{A} $ with $ T^{-1}(A)=A $, it holds $ \mu(A)=0$ or $ \mu (X\smallsetminus A)=0 $.
	\end{itemize}
\end{defi}

Clearly, invariance implies non-singularity. Moreover, the following holds true.

\begin{thm}{\bf (Poincaré Recurrence Theorem, {\normalfont\cite[Thm. 2.12]{Hawkins}})}\label{thm-poincare-recurrence}
		Let $ (X,\mathcal{A}, \mu) $ be a measure space, and let $ T\colon X\to X $ be a measurable transformation. Assume $ \mu(X)<\infty $, and $ T $ is $ \mu $-preserving. Then, $ T $ is recurrent with respect to $ \mu $. 
\end{thm}
\begin{thm}{\bf (Almost every orbit is dense, {\normalfont\cite[Prop. 1.2.2]{Aaronson97}})}\label{thm-almost-every-orbit-dense}
	Let $ (X,\mathcal{A}, \mu) $ be a measure space, and let $ T\colon X\to X $ be non-singular. Then, the following are equivalent.\begin{enumerate}[label={\em (\alph*)}]
		\item $ T $ is ergodic and recurrent.
		\item\label{thm-orbitesdenses-b}  For every  $ A\in\mathcal{A} $ with $ \mu (A)>0 $, we have that for $ \mu $-almost every $ x\in X $, there exists a sequence $ n_k\to \infty $ such that $ T^{n_k}(x)\in A $.
	\end{enumerate}
\end{thm}

Note that, if the space $ X $ is endowed with a topology whose open sets are measurable and have positive measure, then statement {\em \ref{thm-orbitesdenses-b}} implies that $ \mu $-almost every orbit is dense in $ X $.

In holomorphic dynamics, it is possible to replace  the  function by an iterate of it, since the dynamics remain essentially the same. Thus, we are interested in knowing which ergodic properties remain under taking iterates of the function.

\begin{lemma}{\bf (Ergodic properties for $ T^k $)}\label{lemma-ergodic-properties-invariant}
		Let $ (X,\mathcal{A}, \mu) $ be a measure space, and let $ T\colon X\to X $ be non-singular. Let $ k$ be a positive integer. Then, \begin{enumerate}[label={\em (\alph*)}]
			\item $ T $ is recurrent if and only if $ T^k $ is recurrent.
			\item If $ T^k $ is ergodic, so is $ T $. The converse is not true in general.
		\end{enumerate} 
\end{lemma}
\begin{proof} 
	\begin{enumerate}[label={(\alph*)}]
		\item It is clear that $ T^k $ recurrent implies $ T $ recurrent. We shall see the converse. To do so, consider $ A\in\mathcal{A} $ with $ \mu (A)>0 $. Since $ T $ is assumed to be recurrent, for $ \mu $-almost every $ x\in A $ there exists a sequence $ n_j \to \infty$ such that $ T^{n_j}(x)\in A $. For such $ x $, consider the following subsequences 
		\[\left\lbrace T^{kn}(x)\right\rbrace _n,\left\lbrace T^{kn+1}(x)\right\rbrace _n, \dots, \left\lbrace T^{2kn-1}(x)\right\rbrace _n. \] At least one of them, say $ \left\lbrace T^{kn+l}(x)\right\rbrace _n $, contains infinitely many $ T^{n_j}(x)$'s. Choose $ n $ and $ k $ so that \[y\coloneqq T^{kn+l}(x)\in A  . \] Then, $ y $ is a point in $ A $ whose orbit returns to $ A $ infinitely many times under $ T^k $. We claim that such points have full measure in $ A $. Assume that, on the contrary, there exists $ B\subset A $ with $ \mu (B)>0 $ such that, for all $ x\in B $, $ T^{nk}(x) \in A$,  only for finitely many $ n $'s. Applying the same procedure as before to $ B $ we can find a  point in $ B $ whose orbit returns to $ B $, and hence to $ A $, infinitely many times under $ T^k $, which is a contradiction. Hence, $ T^k $ is recurrent.
		\item Let $ A\in\mathcal{A} $ be such that $ T^{-1}(A)=A $. Then, $ T^{-k}(A)=A $, and, since $ T^k $ is assumed to be ergodic, either $ \mu(A)=0 $ or $ \mu(X\smallsetminus A)=0 $. Thus, $ T $ is ergodic.
		
		To see that the converse is not true in general, consider the space $ \mathbb{Z} $ endowed with the counting measure $ \mu $, i.e. given $ X\subset\mathbb{Z} $, $ \mu(X)$ is the number of elements of $ X $. Then, the translation $ T\coloneqq x\mapsto x+1 $ is ergodic, since there are no proper $ T $-invariant subsets of $ X $. However, $ T^2 = x\mapsto x+2$ is not ergodic, since $ 2\mathbb{Z} $ is invariant, and $ \mu(2\mathbb{Z})>0 $ and $ \mu(\mathbb{Z}\smallsetminus2\mathbb{Z})>0 $.
	\end{enumerate}
\end{proof}
Mostly, we will use the measure space $ (\partial\mathbb{D}, \mathcal{B}(\partial\mathbb{D}), \lambda) $, where $ \mathcal{B}(\partial\mathbb{D}) $ denotes the Borel $ \sigma $-algebra of $ \partial\mathbb{D} $, and $ \lambda $, its normalized Lebesgue measure. We need the concept of Lebesgue density.

\begin{defi}{\bf (Lebesgue density)}\label{def-lebesgue-density}
Given a Borel set $ A\in  \mathcal{B}(\partial\mathbb{D})$, the {\em Lebesgue density} of $ A $ at $ \xi\in \partial\mathbb{D} $ is defined as \[ d_\xi (A)\coloneqq\lim\limits_{\rho\to 0}\dfrac{\lambda(A\cap D(\xi, \rho))}{\lambda(D(\xi, \rho))}.\] A point $ \xi\in\partial \mathbb{D} $ is called a {\em Lebesgue density point} for $ A $ if $ d_\xi (A)=1 $.
\end{defi}

Given any Borel set $ A\in  \mathcal{B}(\partial\mathbb{D})$, with $ \lambda(A)>0 $, then $ \lambda $-almost every point in $A $ is a Lebesgue density point for $ A $ \cite[p. 138]{Rudin}.

\subsection{Generalized radial arcs and Stolz angles}
Throughout the article, the following concepts will be needed, which describe ways one may approach a  boundary point $ \xi\in\partial\mathbb{D} $.  

In the sequel, we denote the (Euclidean) disk of radius $ \rho >0$ centered at $ \xi\in\partial\mathbb{D} $ by $ D(\xi,\rho) $. We also consider the radial segment at $ \xi $ of length $ \rho>0 $, \[R_{\rho}(\xi)\coloneqq \left\lbrace r\xi\colon r\in(1-\rho,1)\right\rbrace .\]

\begin{defi}{\bf (Landing set)}
	Given a curve $ \gamma\colon \left[ 0,1\right) \to\widehat{\mathbb{C}} $, we consider its \textit{landing set} \[L(\gamma)\coloneqq \left\lbrace v\in\widehat{\mathbb{C}}\colon \textrm{ there exists }\left\lbrace t_n\right\rbrace _n\subset\left[ 0,1\right) , t_n\to 1\textrm{ such that }\gamma(t_n)\to v \right\rbrace .\] 
\end{defi}

By definition, $ L(\gamma) $ is a connected, compact subset of $ \widehat{\mathbb{C}} $. We say that  $ \gamma $ \textit{lands} at $ v\in\widehat{\mathbb{C}} $ if $ L(\gamma)=\left\lbrace v\right\rbrace  $, or, equivalently, if  \[\lim\limits_{t\to 1^-}\gamma(t)=v.\] Given a domain $ U\subset\widehat{\mathbb{C}} $, a point $ p\in\widehat{\partial} U $ is \textit{accessible} from $ U $ if there exists a curve $ \gamma\subset U $ landing at $ p $. 

\begin{defi}{\bf (Crosscut neighbourhoods and Stolz angles)} Let $ \xi\in\partial\mathbb{D} $.
	\begin{itemize}
		\item 	A \textit{crosscut} $ C $ is an open Jordan arc $ C\subset\mathbb{D} $ such that $ \overline{C}=C\cup \left\lbrace a,b\right\rbrace  $, with $ a,b\in\partial\mathbb{D} $. If $ a=b $, we say that $ C $ is {\em degenerate}; otherwise it is {\em non-degenerate}.
		
		\item A \textit{crosscut neighbourhood} of $ \xi \in\partial\mathbb{D}$ is an open set $ N\subset\mathbb{D} $ such that $ \xi\in\partial N$, and $C\coloneqq \partial N \cap\mathbb{D} $ is a  non-degenerate crosscut. We usually write $ N_\xi $ or  $ N_C $, to stress the  dependence on the point $ \xi $ or on the crosscut $ C $. Note that for a crosscut neighbourhood $ N $, $ \partial\mathbb{D}\cap \overline{N} $ is a non-trivial arc. 
		\item Given $ \xi\in\partial\mathbb{D} $, a \textit{Stolz angle}\footnote{Note that the usual defintion of Stolz angle is \[\Delta=\left\lbrace z\in\mathbb{D}\colon \left| \textrm{Arg }\xi- \textrm{Arg }(\xi-z)\right| <\alpha, \left| \xi-z \right| <\rho \right\rbrace .\]However,  both definitions are equivalent for our purposes, and the stated one is slightly more convenient in our setting.} at $ \xi $ is a set of the form 	\[\Delta_{\alpha, \rho}=\left\lbrace z\in\mathbb{D}\colon \left| \textrm{Arg }\xi- \textrm{Arg }(\xi-z)\right| <\alpha, \left| z \right| >1-\rho \right\rbrace .\] 
		\item We say that  $ \gamma $ \textit{lands non-tangentially} at $ \xi \in\partial\mathbb{D}$  if $ \gamma $ lands at $ \xi $, and there exists a Stolz angle $ \Delta_{\alpha, \rho} $ at $ \xi $ with $ \gamma\subset\Delta_{\alpha, \rho}  $. 
	\end{itemize}
\end{defi}

\begin{figure}[h]
	\centering
	\captionsetup[subfigure]{labelformat=empty, justification=centering}
	\hfill
	\begin{subfigure}{0.22\textwidth}
		\includegraphics[width=\textwidth]{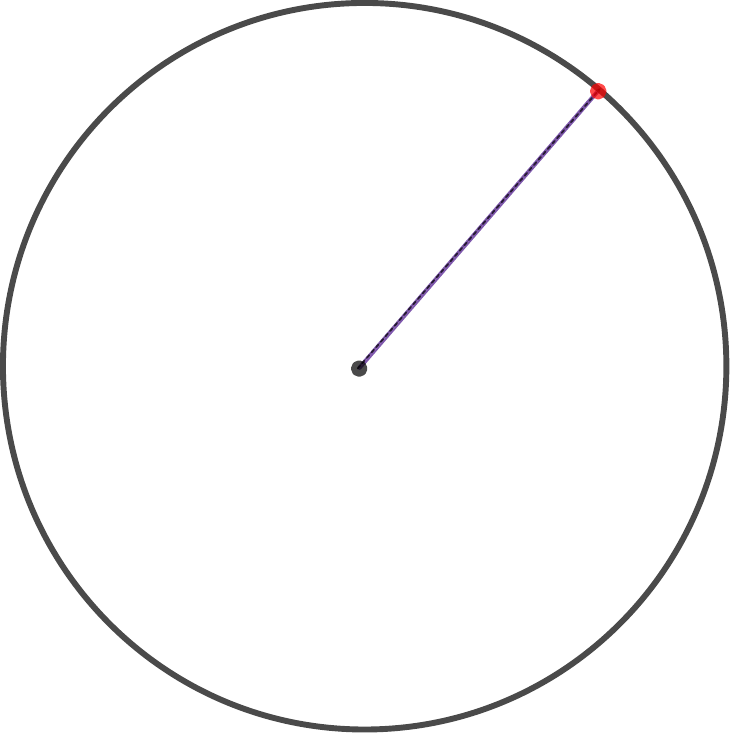}
		\setlength{\unitlength}{\textwidth}
		\put(-0.15, 0.87){$\xi$}
		\caption{\footnotesize Radial segment}
	\end{subfigure}
	\hfill
	\begin{subfigure}{0.22\textwidth}
		\includegraphics[width=\textwidth]{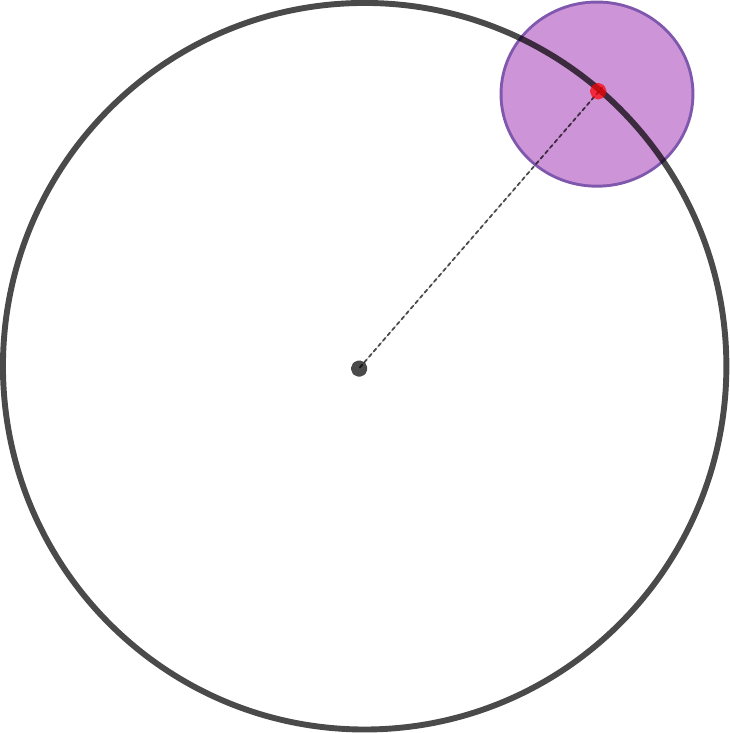}
		\setlength{\unitlength}{\textwidth}
		\put(-0.15, 0.87){$\xi$}
		\caption{\footnotesize (Euclidean) disk}
	\end{subfigure}
	\hfill
	\begin{subfigure}{0.22\textwidth}
		\includegraphics[width=\textwidth]{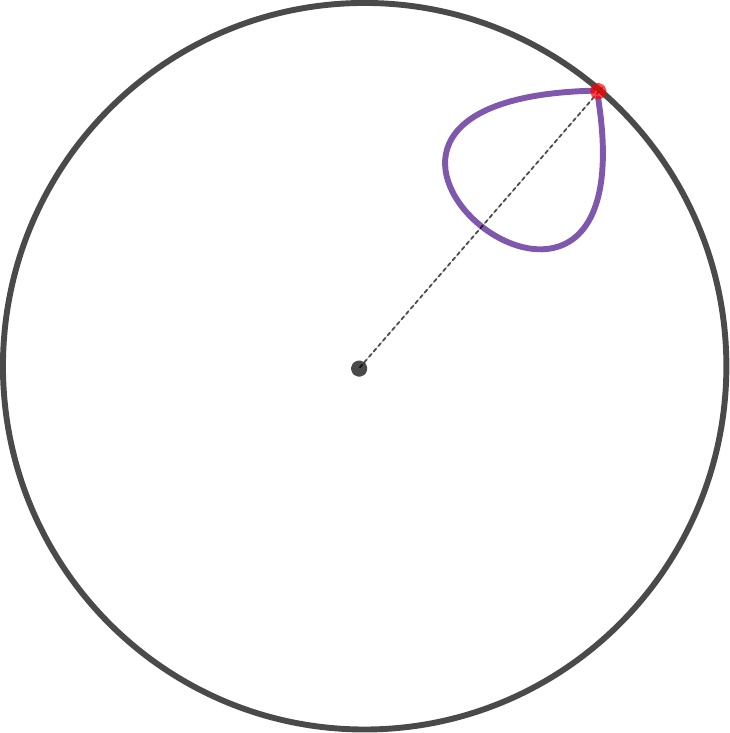}
				\setlength{\unitlength}{\textwidth}
		\put(-0.15, 0.87){$\xi$}
		\caption{\footnotesize Degenerate crosscut}
	\end{subfigure}
\hfill
		\vspace{1cm}
	
		\hfill
	\begin{subfigure}{0.22\textwidth}
		\includegraphics[width=\textwidth]{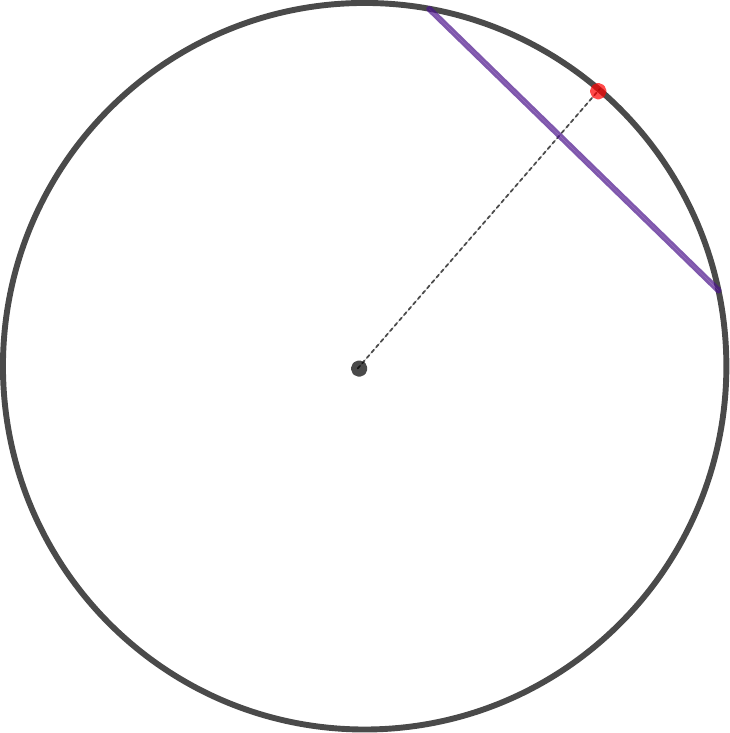}
				\setlength{\unitlength}{\textwidth}
		\put(-0.15, 0.87){$\xi$}
		\caption{\footnotesize (Non-degenerate) crosscut}
	\end{subfigure}
	\hfill
	\begin{subfigure}{0.22\textwidth}
		\includegraphics[width=\textwidth]{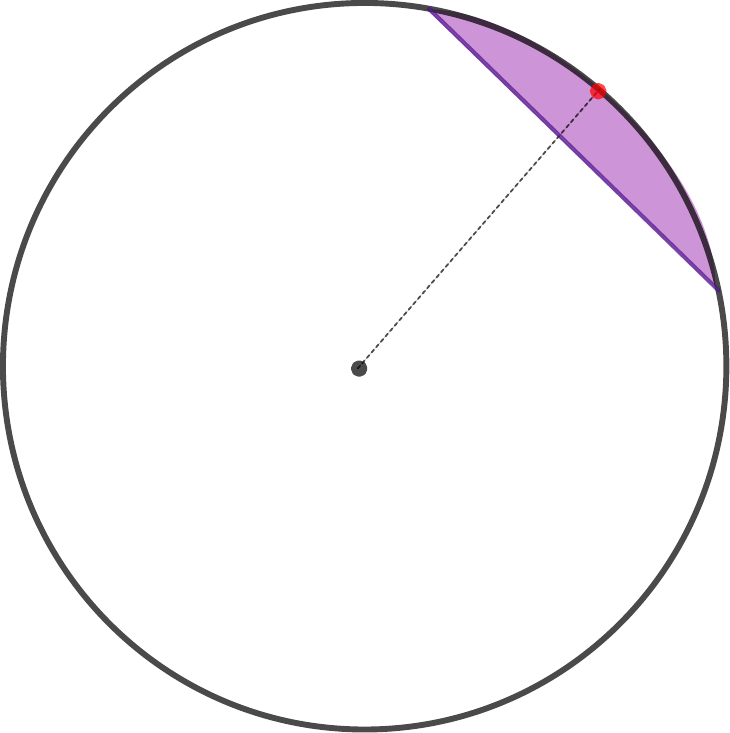}
				\setlength{\unitlength}{\textwidth}
		\put(-0.15, 0.87){$\xi$}
		\caption{\footnotesize Crosscut neighbourhood}
	\end{subfigure}
	\hfill
	\begin{subfigure}{0.22\textwidth}
		\includegraphics[width=\textwidth]{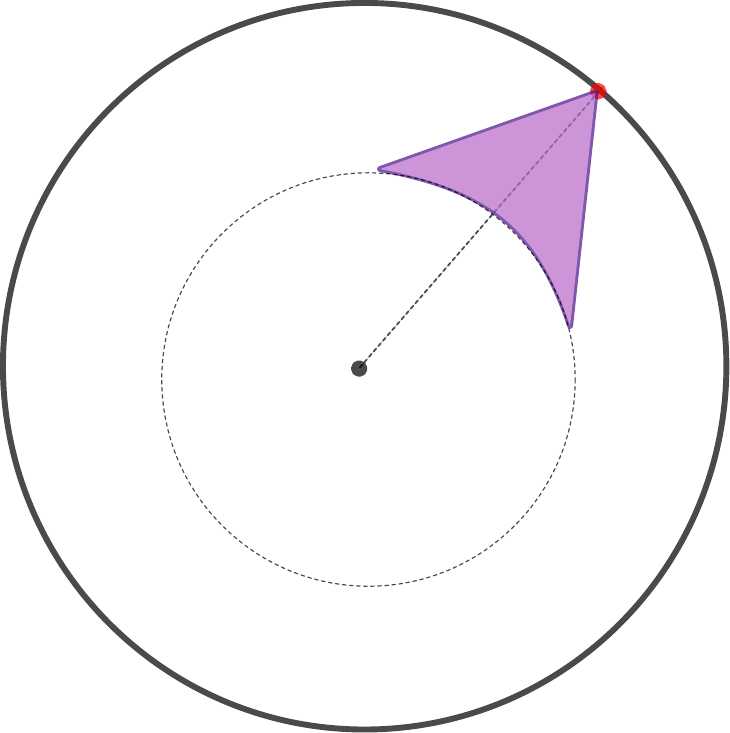}
				\setlength{\unitlength}{\textwidth}
		\put(-0.15, 0.87){$\xi$}
		\caption{\footnotesize Stolz angle \\ \textcolor{white}{text}}
	\end{subfigure}
	\hfill
	\hfill
	\caption{\footnotesize Different sets related to $ \xi\in\partial\mathbb{D} $.}\label{fig-radial-nbhds}
\end{figure}

Some times it is more convenient to work in the upper half-plane $ \mathbb{H} $ rather than in the unit disk $ \mathbb{D} $.  The previous concepts can be defined analogously for points in $ \partial\mathbb{H} $. In particular, the specific formulas for both the radial segment and Stolz angles at a point $ x\in\mathbb{R}$ are \[  R_{\rho}^\mathbb{H}(x)\coloneqq\left\lbrace z\in\mathbb{H}\colon {\textrm{Im }}w<\rho, { \textrm{Re }}w=x \right\rbrace;\]
\[\Delta^\mathbb{H}_{\alpha, \rho}(x)\coloneqq\left\lbrace z\in\mathbb{H}\colon {\textrm{Im }}w<\rho, \dfrac{ \left| \textrm{Re }w-x\right| }{ {\textrm{Im }}w}<\tan\alpha \right\rbrace .\]

A more flexible notion of radial segment and Stolz angle will be needed for our purposes.

\begin{defi}{\bf (Generalized radial arc and  Stolz angle) } \label{defi-radi-stolz-angle} Let $ p\in\overline{\mathbb{D}} $ and let $ \xi\in\partial\mathbb{D} $, $ \xi\neq p $. Let $ \rho>0 $ and $ 0<\alpha<\pi/2 $. \begin{itemize}
		\item If $ p\in \mathbb{D}$, consider the Möbius transformation $ M\colon\mathbb{D}\to\mathbb{D} $, $ M(z)=\dfrac{p-z}{1-\overline{p}z} $. Then,  the {\em (generalized) radial segment} $ R_{\rho}(\xi,p) $ of length $ \rho $ at $ \xi $ is defined as the preimage under $ M $ of the radial segment $ R_\rho(M(\xi)) $. Analogously, the {\em (generalized) Stolz angle} $ \Delta_{\alpha, \rho}(\xi, p) $ of angle $ \alpha $ and length $ \rho $  is  the preimage under $ M $ of the Stolz angle $ \Delta_{\alpha, \rho}(M(\xi)) $.
		That is, \[ R_{\rho}(\xi, p)\coloneqq M^{-1}(R_{\rho}(M(\xi))),\]\[ \Delta_{\alpha, \rho}(\xi, p)\coloneqq M^{-1}(\Delta_{\alpha,\rho}(M(\xi))).\]
		
		\item 	If $ p\in\partial\mathbb{D} $, consider the Möbius transformation $ M\colon\mathbb{D}\to\mathbb{H} $, $ M(z)=i\dfrac{p+z}{p-z} $. Then, the {\em (generalized) radial segment} and {\em Stolz angle} at $ \xi $ are defined as the preimages of the corresponding radial segment and Stolz angle at $ M(\xi) \in\mathbb{R}$. That is, \[ R_{\rho}(\xi,p)\coloneqq M^{-1}(R_{\rho}^\mathbb{H}(M(\xi)))\] \[ \Delta_{\alpha, \rho}(\xi,p)\coloneqq M^{-1}(\Delta_{\alpha,\rho}^\mathbb{H}(M(\xi))).\]
	\end{itemize}
\end{defi}
See Figures \ref{fig-StolzangleD} and \ref{fig-StolzangleH}.

Observe that $ R_{\rho}(\xi) =R_{\rho}(\xi, 0) $,  and $ \Delta_{\alpha, \rho}(\xi) =\Delta_{\alpha, \rho}(\xi, 0) $. Note also that $ R_{\rho}(\xi,p) $ is a curve landing non-tangentially at $ \xi\in\partial\mathbb{D} $, while $ \Delta_{\alpha, \rho}(\xi,p) $ is an angular neighbourhood of $ \xi $, since Möbius transformations are conformal, and hence angle-preserving. 
		\begin{figure}[htb!]\centering
	\includegraphics[width=14cm]{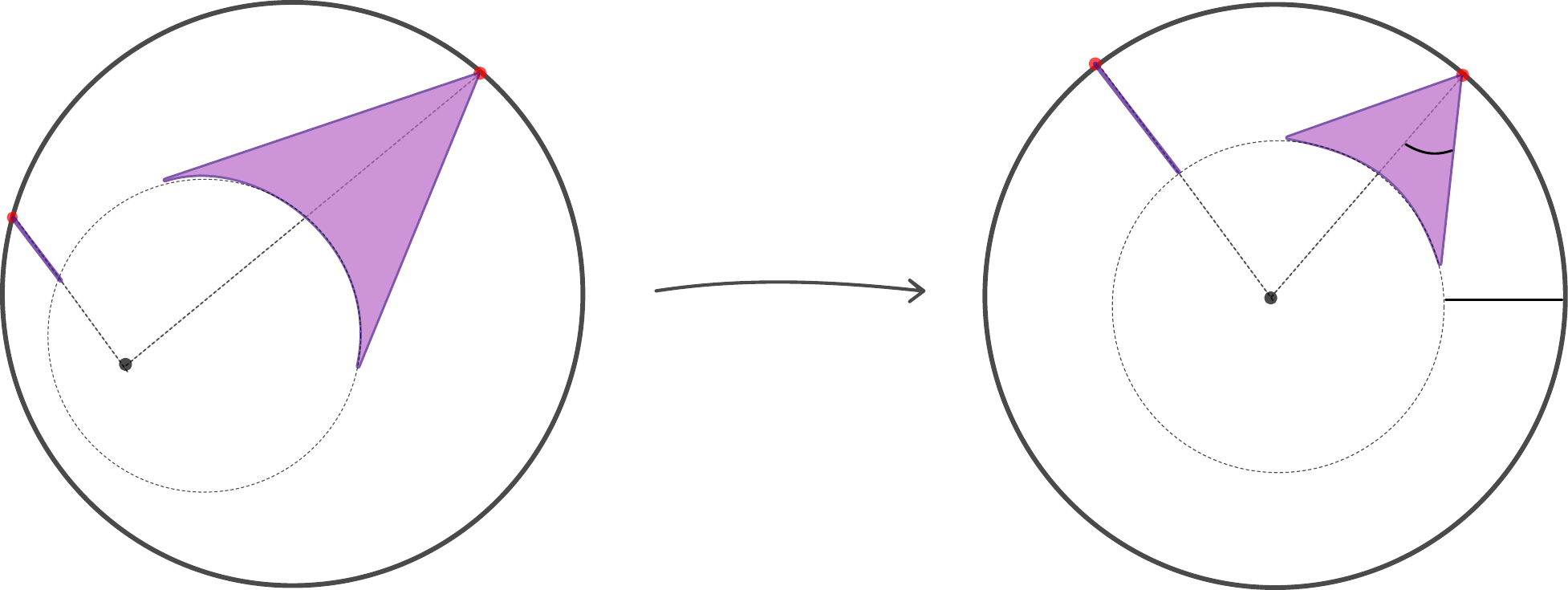}
	\setlength{\unitlength}{14cm}
	\put(-0.5, 0.2){$M$}
	\put(-0.92, 0.125){$p$}
	\put(-0.18, 0.18){$0$}
		\put(-0.685, 0.33){$\xi_1$}
	\put(-1.025, 0.23){$\xi_2$}
			\put(-0.06, 0.33){$M(\xi_1)$}
	\put(-0.39, 0.33){$M(\xi_2)$}
	\put(-0.68,0.02){$ \mathbb{D} $}
		\put(-0.05,0.02){$ \mathbb{D} $}
			\put(-0.88, 0.31){\footnotesize$\Delta_{\alpha, \rho}(\xi_1, p)$}
		\put(-0.98, 0.22){\footnotesize$R_{ \rho}(\xi_2, p)$}
			\put(-0.22, 0.32){\footnotesize$\Delta_{\alpha, \rho}(M(\xi_1))$}
		\put(-0.35, 0.25){\footnotesize$R_{ \rho}(M(\xi_2))$}
			\put(-0.05, 0.165){$ \rho $}
		\put(-0.1, 0.26){$ \alpha$}
	\caption{\footnotesize  Radial arc and angular neighbourhood with respect to $ p\in\mathbb{D} $.}\label{fig-StolzangleD}
\end{figure}
\begin{figure}[htb!]\centering
	\includegraphics[width=14cm]{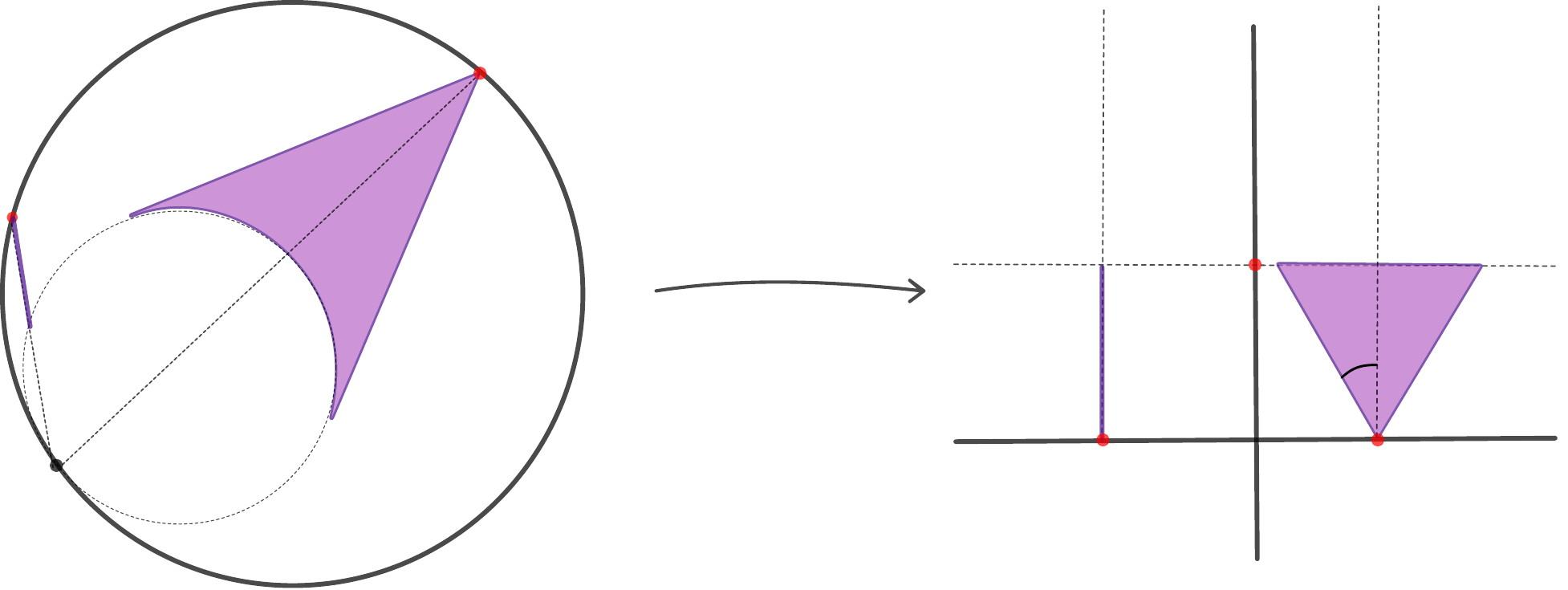}
		\setlength{\unitlength}{14cm}
	\put(-0.5, 0.2){$M$}
	\put(-0.98,0.05){$p$}
	\put(-0.685, 0.33){$\xi_1$}
		\put(-0.88, 0.3){\footnotesize$\Delta_{\alpha, \rho}(\xi_1, p)$}
			\put(-0.98, 0.2){\footnotesize$R_{ \rho}(\xi_2, p)$}
			\put(-0.08, 0.15){\footnotesize$\Delta^\mathbb{H}_{\alpha, \rho}(M(\xi_1))$}
				\put(-0.41, 0.15){\footnotesize$R^\mathbb{H}_{ \rho}(M(\xi_2))$}
	\put(-1.025, 0.23){$\xi_2$}
	\put(-0.15, 0.06){$M(\xi_1)$}
	\put(-0.34, 0.06){$M(\xi_2)$}
		\put(-0.68,0.02){$ \mathbb{D} $}
				\put(-0.06, 0.33){$ \mathbb{H} $}
								\put(-0.23, 0.21){$ \rho i $}
												\put(-0.145, 0.15){$ \alpha$}
	\caption{\footnotesize Radial arc and angular neighbourhood with respect to $ p\in\partial\mathbb{D} $.}\label{fig-StolzangleH}
\end{figure}
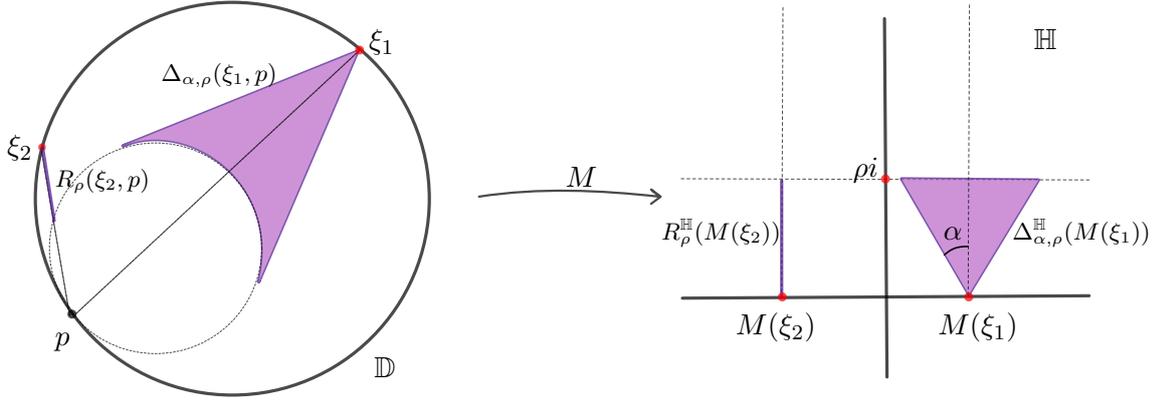

\subsection{Boundary behaviour of meromorphic maps in $ \mathbb{D} $}\label{subsection-radial-extension-meromorphic-maps}
In this section, we are interested in the boundary behaviour of meromorphic maps $ h\colon\mathbb{D}\to\widehat{\mathbb{C}} $. Since $ h $ may not extend continuously to $ \partial\mathbb{D} $, the concepts of radial and angular limit are a keystone on studying the boundary behavior of $ h $.

\begin{defi}{\bf (Radial and angular limit)}
	Let $ h\colon\mathbb{D}\to\widehat{\mathbb{C}} $ be a meromorphic map, and let $ \xi\in\partial\mathbb{D} $.
We say that  $ h $  has \textit{radial limit} at $ \xi  $ if the limit \[h^*(\xi)\coloneqq \lim\limits_{t\to 1^-}h(t\xi) \]  exists. 
We say that  $ h $  has \textit{angular limit} at $ \xi  $ if, for any Stolz angle $ \Delta $ at $ \xi $, the limit \[\lim\limits_{z\to\xi, z\in\Delta}h(z) \]  exists. 
\end{defi}
 
 Note that, whenever we write $ h^*(\xi)=v $ we are assuming implicitly that the radial limit exists, and equals $ v $. The map \[h^*\colon\partial\mathbb{D}\to\widehat{\mathbb{C}}\]  is called the \textit{radial extension} of $ h $ (defined wherever the radial limit exists). 
 
 For maps $ h\colon\mathbb{D}\to\widehat{\mathbb{C}} $ omitting three values in $ \widehat{\mathbb{C}}  $, the following theorem relates radial and angular limits.
\begin{thm}{\bf (Lehto-Virtanen, {\normalfont\cite[Sect. 4.1]{Pommerenke}})}\label{thm-lehto-virtanen}
	Let $ h\colon\mathbb{D}\to\widehat{\mathbb{C}} $ be a meromorphic map omitting three values in $ \widehat{\mathbb{C}}  $. Let $ \gamma $ be a curve in $ \mathbb{D} $ landing at $ \xi\in\partial \mathbb{D} $. If $ h(\gamma) $ lands at a point $ v\in\mathbb{C} $, then $ h $ has angular limit at $ \xi $  equal to $ v $. In particular, radial and angular limits are the same.
\end{thm} 

\begin{remark}{\bf (Limit on generalized radial arcs and Stolz angles)}
Note that, in particular, the Lehto-Virtanen Theorem justifies that, for meromorphic maps omitting three values, it is equivalent to take the limit along the radial segment, than along any generalized radial arc. Likewise, the angular limit can be computed along generalized Stolz angles. 
\end{remark}

The following theorems describe more precisely  the boundary behaviour of meromorphic maps $ h\colon\mathbb{D}\to\widehat{\mathbb{C}} $ in terms of measure.

\begin{thm}{\bf (Radial extensions are measurable, {\normalfont\cite[Prop. 6.5]{Pommerenke}})}\label{thm-mesurabilitat-funcions-disk}
Let $ h\colon\mathbb{D} \to\widehat{\mathbb{C}}$ be continuous. Then, the points $ \xi\in\partial\mathbb{D} $ where the radial limit $ h^* $ exists form a Borel set, and if $ A\subset\widehat{\mathbb{C}} $ is a Borel set, then \[(h^*)^{-1}(A)\coloneqq\left\lbrace \xi\in\partial\mathbb{D}\colon h^*(\xi)\in A \right\rbrace \subset\partial\mathbb{D}\] is also a Borel set.
\end{thm}

In the particular case where $ h= \varphi\colon\mathbb{D}\to U $ is a Riemann map, 
the following theorem, due to Fatou, Riesz and Riesz, ensures the existence of radial limits almost everywhere. 

\begin{thm}{\bf (Existence of radial limits) }\label{thm-FatouRiez}
	Let $ \varphi\colon\mathbb{D}\to U $ be a Riemann map. Then, for $ \lambda $-almost every point $ \xi \in\partial \mathbb{D}$, the radial limit $ \varphi^*(\xi ) $ exists. Moreover, if we fix $ \xi\in \partial\mathbb{D} $ for which $ \varphi^*(\xi) $ exists, then $ \varphi^*(\xi )\neq  \varphi^*(\zeta)  $, for $ \lambda $-almost every point $ \zeta\in\partial \mathbb{D}$.
\end{thm}

\subsection{Harmonic measure}\label{subsection-harmonic-measure}
Let $ U $ be a simply connected domain. Then, the Riemann map $ \varphi\colon \mathbb{D}\to U $ induces a measure in $ \widehat{\partial} U $, the harmonic measure, which is the appropriate one when dealing with the boundaries of Fatou components. We define harmonic measure in $ \widehat{\partial} U $ in terms of the push-forward under a Riemann map of the normalized measure on the unit circle $ \partial\mathbb{D} $.

\begin{defi}{\bf (Harmonic measure)}
	Let $ U\subsetneq{\mathbb{C}} $ be a simply connected domain, $ z\in U $, and let $ \varphi\colon\mathbb{D}\to U $ be a Riemann map, such that $ \varphi(0)=z\in U $. Let $ (\partial\mathbb{D}, \mathcal{B}, \lambda) $ be the measure space on $ \partial \mathbb{D} $ defined by $ \mathcal{B} $, the Borel $ \sigma $-algebra of $ \partial \mathbb{D} $, and $ \lambda $, its normalized Lebesgue measure. Consider the measurable space $ (\widehat{\mathbb{C}}, \mathcal{B}(\widehat{\mathbb{C}})) $, where  $ \mathcal{B}(\widehat{\mathbb{C}}) $ is the Borel $ \sigma $-algebra of $ \widehat{\mathbb{C}} $. Then, given $ A\in  \mathcal{B}(\widehat{\mathbb{C}}) $, the \textit{harmonic measure at $ z $ relative to $ U $} of the set $ A$ is defined as:\[\omega_U(z, A)\coloneqq\lambda ((\varphi^*)^{-1}(A)).\]
\end{defi}

Note that the harmonic measure $ \omega_U(z,\cdot) $ is well-defined. Indeed, by Theorem \ref{thm-mesurabilitat-funcions-disk}, the set \[(\varphi^*)^{-1}(A)=\left\lbrace \xi\in\partial\mathbb{D}\colon\varphi^*(\xi)\in A\right\rbrace \]is a Borel set of $\partial \mathbb{D} $, and hence measurable. 
We also note that the definition of $ \omega_U(z, \cdot) $ is independent of the choice of $ \varphi$, provided it satisfies $ \varphi(0)=z $, since $ \lambda $ is invariant under rotation.

We refer to \cite{harmonicmeasure2, Pommerenke} for equivalent definitions and further properties of the harmonic measure. We only need the following simple facts.

\begin{lemma}{\bf (Sets of zero and full harmonic measure)}
	Let $ U\subsetneq{\mathbb{C}} $ be a simply connected domain, and $ B\in\mathcal{B}(\widehat{\mathbb{C}}) $.  If there exists $ z_0\in U $ such that $ \omega_U(z_0, B)=0 $ (resp. $ \omega_U(z_0, B)=1 $), then $ \omega_U(z, B)=0 $ (resp. $ \omega_U(z, B)=1 $) for all $ z\in U $.
	In this case, we say that the set $ B $ has {\em zero} (resp. {\em full}) {\em harmonic measure relative to $ U $}, and we write $ \omega_U(B)=0 $ (resp. $ \omega_U(B)=1 $).
\end{lemma}

 Finally, we are interested in the support of $ \omega_U $. Recall that \[\textrm{supp }\omega_U\coloneqq\left\lbrace x\in\widehat{\mathbb{C}}\colon \textrm{ for all }r>0,\textrm{ } \omega_U(D(x,r))>0\right\rbrace .\]
 Note that it only depends on the sets of zero measure, hence it is well-defined without specifying the base-point of the harmonic measure.

 \begin{lemma}{\bf (Support of harmonic measure)}\label{lemma_support_harmonic_measure}	Let $ U\subsetneq{\mathbb{C}} $ be a simply connected domain. Then,
 	\[\textrm{\em supp }\omega_U =\widehat{\partial }U.\]
 \end{lemma}
 This follows easily from considering an equivalent definition of harmonic measure in terms of solutions to the Dirichlet problem, see e.g. \cite[Chap. 21]{Conway}.

\subsection{Regular and singular values for holomorphic maps}\label{subsect-regular-singular}

Throughout the paper, we will make an extensive use of the concepts of regular and singular values. Although these definitions are quite standard in the context of entire or meromorphic maps (i.e. with one single essential singularity), we believe it is useful to give definitions in the rather general context of functions of class $ \mathbb{K} $ or inner functions.

We consider the following class of meromorphic functions, denoted by $ \mathbb{M} $, consisting of functions\[f\colon \widehat{\mathbb{C}}\smallsetminus E(f)\longrightarrow \widehat{\mathbb{C}},\] where $ \Omega(f)\coloneqq\widehat{\mathbb{C}}\smallsetminus E(f)  $ is the largest set where $ f $ is meromorphic, and, for all $ z\in E(f) $, the cluster set $ Cl(f, z) $ of $ f $ at $ z $ is $ \widehat{\mathbb{C}} $, that is
\begin{align}\label{eq-cluster-set}
	Cl(f, z)=\left\lbrace w\in \widehat{\mathbb{C}}\colon\textrm{ there exists }\left\lbrace z_n\right\rbrace _n\subset \Omega (f) ,\, z_n\to z, \,f(z_n)\to w \right\rbrace =\widehat{\mathbb{C}}.
\end{align}

If $ E(f)=\emptyset $, then $ f $ is rational and we make the further assumption that $ f $ is non-constant. 
Note that $ \Omega(f)$ is open, and $ E(f) $ has empty interior. Indeed, if $ z $ is an interior point for $ E $, there does not exist any sequence in $ \Omega (f) $ converging to $ z $, and hence $ Cl(f, z) $ is empty, a contradiction.

In the case that $ f $ is an inner function (Sect. \ref{sect-iteration-inner-functions}), $ E(f) $ will be a closed subset of $ \partial\mathbb{D} $, while if $ f\in\mathbb{K} $, $ E(f) $ is a countable subset of the Riemann sphere. In both cases, the assumption on the cluster set is satisfied (see e.g.  \cite{BakerDomínguez, BakerDominguezHerring}).

In this general setting, regular and singular values, and critical and asymptotic values, are defined as follows. Note that appropriate charts have to be used when dealing with $ \infty $.
\begin{defi}{\bf (Regular and singular values)}\label{defi-regular-singular-values}
	Given a value $ v\in \widehat{\mathbb{C}}$, we say that  $ v $ is a \textit{regular} value for $ f$ if there exists $ r\coloneqq r(v)>0 $ such that all branches $ F_1 $ of $ f^{-1} $ are well-defined (and, hence, conformal) in $ D(v,r) $. Otherwise we say that $ v $ is a \textit{singular} value for $ f $. 
\end{defi}

The set of singular values of $ f $ is denoted by $ SV(f) $. Note that $ SV(f) $ is closed by definition, and it is the smallest set for which \[f\colon \widehat{\mathbb{C}}\smallsetminus(E(f)\cup f^{-1}(SV(f)))\longrightarrow\widehat{\mathbb{C}}\smallsetminus SV(f)\] is a covering map.

\begin{defi}{\bf (Critical and asymptotic values)}
	Given a value $ v\in \widehat{\mathbb{C}}$, we say  that $ v $ is a \textit{critical value}  if there exists  $ z\in\Omega $ such that $ f' (z)=0$ and $ f(z)=v $. We say that $ z $ is a {\em critical point}.
	
\noindent	We say that $ v $ is an {\em asymptotic value} if there exists a curve $ \gamma\colon \left[ 0,1\right)  \to \Omega$ such that $ \gamma(t)\to\partial\Omega $ and $ f(\gamma(t))\to v $, as $ t\to 1 $. We say that the curve $ \gamma $ is an {\em asymptotic path}.
\end{defi}

The set of critical values of $ f $ is denoted by $ CV(f) $, while $ AV(f) $ stands for the set of asymptotic values.

Note that  we do not assume, in general, that $ \gamma(t) $ lands at a definite point $ \partial\Omega$ as $ t\to 1 $. 
 However, this will be the case for both inner functions (Lemma \ref{lemma-asymptotic-paths-inner}) and functions in class $ \mathbb{K} $ (Sect. \ref{subsect-associated-inner-function}). We say that $ v $  is an asymptotic value corresponding to $ x\in \partial \Omega $ if $ v=\lim\limits_{t\to 1} f(\gamma(t))$, where $ \gamma $ is a curve such that $ \gamma(t)\to x $ as $ t\to1$. Note that an asymptotic value may correspond to more than one point in $ \partial \Omega  $.

The following lemma makes explicit the relation between regular and singular values, and critical and asymptotic values, in the sense of Iversen \cite{Iversen-thesis,BergweilerEremenko}. Although its content is well-known for meromorphic functions and for functions of finite type, we were unable to find a proof in the literature which fits into the general setting we are considering.
We follow the ideas of \cite{BergweilerEremenko}. 

\begin{lemma}{\bf (Characterization of singular values)}
	Let $ f \in \mathbb{M}$. Then, \[SV(f)=\overline{CV(f)\cup AV(f)}.\]
\end{lemma}
\begin{proof}
	
	Let $ v\in\widehat{\mathbb{C}} $. For every $ r>0 $, choose a component $ U(r) $ of $ f^{-1}(D(v,r)) $ in such a way that $ r_1<r_2 $ implies $ U(r_1)\subset U(r_2) $. Note that $ f^{-1}(D(v,r)) $ is non-empty for all $v\in\widehat{\mathbb{C}} $ and $ r>0 $, because of (\ref{eq-cluster-set}) if $ E(f)\neq \emptyset $ and trivially if $ E(f)=\emptyset $.
	
	Two possibilities can occur.
	\begin{itemize}
		\item $ \bigcap_{r>0} U(r)\neq\emptyset$. In such a case, there exists $ z\in\Omega(f) $ such that $ \bigcap_{r>0} U(r)=\left\lbrace z\right\rbrace $ and, hence, $ f(z)=v $ (indeed, note that if the previous intersection was larger than a point, this would contradict the Open Mapping Principle). If $ f'(z)\neq 0 $, then the inverse branch $ F_1 $ of $ f^{-1} $ sending $ v $ to $ z $ is well-defined and conformal in $ D(v, r_0) $, for some $ r_0>0 $. If $ f'(z)= 0 $, $ f $ acts as a branched covering around $ z $, and the corresponding inverse branch of $ f^{-1} $ is not well-defined around $ v $. Note that, this latter case occurs if and only if $ z $ is a critical point and $ v $ is a critical value.
		\item $ \bigcap_{r> 0} U(r)=\emptyset$. We show that this case correponds with the case of $ v $ being an asymptotic value. 
		First, assume $ \bigcap_{r> 0} U(r)=\emptyset$, and we shall construct an asymptotic path $ \gamma $ for $ v $. Let $ r_k $ be a sequence of positive real numbers tending to 0, and let $ z_k\in U(r_k) $. Let $ \gamma_k \subset U(r_k)$ be a curve connecting $ z_k $ to $ z_{k+1} $. Then, $ \gamma=\cup_k\gamma_k $ satisfies $ \gamma(t)\to\partial \Omega $, and hence is an asymptotic path for $ v $.
		
		On the other hand, if $ v $ is an asymptotic value, let $ \gamma  $ be an asymptotic path. Then, define $ U(r) $ to be the connected component of $ f^{-1}(D(v,r)) $ containing the \textit{tail} of $ \gamma $. Then, $ \bigcap_{r> 0} U(r)=\emptyset$, as desired.
	\end{itemize}
	Then, it is clear that $ v $ is a regular value (as in Def. \ref{defi-regular-singular-values}) if and only if $ v $ is not a critical, nor an asymptotic value, nor an accumulation thereof.
\end{proof}

\section{Iteration of inner functions}\label{sect-iteration-inner-functions}
Consider a holomorphic map $ g\colon\mathbb{D}\to\mathbb{D} $. Since $ g $ is bounded, the radial extension $ g^* $ exists $ \lambda $-almost everywhere (Thm. \ref{thm-FatouRiez}). We are interested in the case where $ g^* $ preserves $ \partial\mathbb{D} $ $ \lambda $-almost everywhere.

\begin{defi}{\bf (Inner function)}
	A holomorphic self-map of the unit disk $ g\colon\mathbb{D}\to\mathbb{D} $ is an {\em inner function} if its radial extension $ g^* $ satisfies $ g^*(\xi)\in\partial \mathbb{D} $, for $ \lambda $-almost every point  $\xi\in \partial\mathbb{D} $.
	\end{defi}

In general, inner functions present a highly discontinuous behaviour in $ \partial\mathbb{D} $.

\begin{defi}{\bf (Singularity)}
	Let $ g $ be an inner function. A point $ \xi\in\partial\mathbb{D} $ is called a {\em singularity} of $ g $ if $ g $ cannot be continued analytically to any neighbourhood of $ \xi $. Denote the set of singularities of $ g $ by $ E(g) $.
\end{defi}

Throughout the paper, given any inner function $ g $, we consider it continued to $ \widehat{\mathbb{C}}\smallsetminus\overline{\mathbb{D}} $ by the reflection principle, and to $ \partial\mathbb{D}\smallsetminus E(g) $ by analytic continuation. In other words, $ g $ is considered as its maximal meromorphic extension \[g\colon \widehat{\mathbb{C}}\smallsetminus E(g)\to \widehat{\mathbb{C}}.\]

If an inner function has finite degree, then it is a finite Blaschke product. In this case, $ g $ has no singularities, and it extends to the Riemann sphere as  a rational map. On the other hand, infinite degree inner functions must have at least one singularity. The following lemma characterizes the singularities of an inner function.

\begin{lemma}{\bf (Characterization of singularities, {\normalfont \cite[Thm. II.6.6]{garnett}})}\label{lemma-sing}
	Let $ g\colon\mathbb{D}\to\mathbb{D} $ be an inner function. Then, $ \xi\in E(g) $ if and only if, for any crosscut neighbourhood $ N_\xi $ of $ \xi $, \[\overline{g(N_\xi)}=\overline{\mathbb{D}}.\]
\end{lemma}
\subsection{Iteration of holomorphic self-maps of the unit disk}
The asymptotic behaviour of the iterates of a holomorphic self-map of the unit disk is essentially described by the Denjoy-Wolff theorem. Note that the results in this section are valid for any holomorphic self-map of $ \mathbb{D} $, not necessarily an inner function.

\begin{thm}{\bf (Denjoy-Wolff,{ \normalfont \cite[Thm. 5.2]{milnor}})}
Let $ g\colon\mathbb{D}\to\mathbb{D} $ be holomorphic, which is not the identity nor an elliptic Möbius transformation. Then, there exists a point $ p\in\overline{\mathbb{D}} $, the Denjoy-Wolff point of $ g $, such that for all $ z\in\mathbb{D} $, $ g^n(z)\to p$. 
\end{thm}

Hence, holomorphic self-maps of $ \mathbb{D} $ are classified into two types: the elliptic ones, for which $ p\in \mathbb{D} $, and the non-elliptic ones, with $ p\in \partial\mathbb{D} $. In the first case, the Schwarz lemma describes  the dynamics precisely. 
\begin{thm}{\bf (Schwarz lemma, {\normalfont \cite[Lemma 1.2]{milnor}})}\label{scharzlemma} Let $ g\colon\mathbb{D}\to\mathbb{D} $ be holomorphic, with $ g(0)=0 $. Then, for all $ z\in\mathbb{D} $, $\left| g( z )\right| \leq\left| z\right|  $, and $ \left| g'(0)\right| \leq 1 $. 
\end{thm}

An analogous result was obtained by Wolff for non-elliptic self-maps of $ \mathbb{D} $.
\begin{thm}{\bf (Wolff lemma, {\normalfont \cite{wolff}})}\label{wolfflemma} Let $ g\colon\mathbb{D}\to\mathbb{D} $ be holomorphic, with Denjoy-Wolff point $p\in\partial\mathbb{D}$. Let $ D \subset \mathbb{D}$ be an open disk tangent to $ \partial \mathbb{D} $ at $ p $. Then, $ g(D) \subset D$. In particular, $ g^*(p)=p $.  
\end{thm}
	
Another equivalent way of stating Wolff lemma is that, for	any holomorphic function $ h\colon\mathbb{H}\to\mathbb{H} $ with Denjoy-Wolff point $ \infty $ and any upper half-plane $ H $, $ h(H) \subset H$ (see also \cite[Lemma 2.33]{Bargmann}).

Note that, in the elliptic case, $ g $ is holomorphic in a neighbourhood of the Denjoy-Wolff point $ p\in\mathbb{D} $, which is fixed and it is either attracting (if $ \left| g'(p)\right| \in (0,1)$) or superattracting (if $ g'(p)=0$). In the former case, $ g $ is conjugate to $ z\mapsto \left| g'(p) \right| z $ in a neighourhood of $ p $ (by Koenigs Theorem, see e.g. \cite[Chap. 8]{milnor}). In the latter case, the dynamics are conjugate to those of $ z\mapsto z^d $, where $ d $ stands for the local degree of $ g $ at $ p $ (by Böttcher Theorem, see e.g. \cite[Chap. 9]{milnor}).

An analogous result for the non-elliptic case is given by the following result of Cowen, which leads to a classification of non-elliptic self-maps of $ \mathbb{D} $ in terms of the dynamics near the Denjoy-Wolff point. 
\begin{defi}{\bf (Absorbing domains and fundamental sets)}
	Let $ U $ be a domain in $ \mathbb{C} $ and let $ f\colon U \to U $ be a holomorphic map. 	A domain $ V\subset U $ is said to be an \textit{absorbing domain} for $ f $ in $ U $ if  $ f(V)\subset V $ and for every compact set $ K\subset U $ there exists $ n\geq 0 $ such that $ f^n(K)\subset V $.
	If, additionally,  $ V $ is  simply connected and $ f|_{V} $ is univalent, $ V$ is said to be a \textit{fundamental set} for $ f $ in $ U $.
\end{defi}
\begin{thm}{\bf (Cowen's classification of self-maps of $ \mathbb{D} $, {\normalfont \cite{cowen}})}\label{teo-cowen}
	Let $ g$ be a holomorphic self-map of  $ \mathbb{D} $ with Denjoy-Wolff point $ p\in\partial \mathbb{D} $. Then, there exists a fundamental set $ V $ for $ g $ in $ \mathbb{D} $.
	
	\noindent Moreover, given a fundamental set $ V $, there exists a domain $ \Omega $ equal to $ \mathbb{C} $ or $ \mathbb{H}=\left\lbrace \textrm{\em Im }z>0\right\rbrace  $, a holomorphic map $ \psi\colon \mathbb{D}\to\Omega  $, and a Möbius transformation $ T\colon\Omega\to\Omega $, such that: \begin{enumerate}[label={\em (\alph*)}]
		\item $ \psi(V) $ is a fundamental set for $ T $ in $ \Omega $,
		\item $ \psi\circ g=T\circ \psi $ in $ \mathbb{D} $,
		\item $ \psi $ is univalent in $ V $.
	\end{enumerate}
	
	Moreover, $ T $  and $ \Omega $ depend only on the map $ g $, not  on the fundamental set $ V $. In fact (up to a conjugacy of $ T $ by a Möbius transformation preserving $ \Omega $), one of the following cases holds:\begin{itemize}
		\item $ \Omega= \mathbb{C} $, $ T=\textrm{\em id}_\mathbb{C} +1 $ {\em (doubly parabolic type)},
		\item $ \Omega= \mathbb{H} $, $ T=\lambda\textrm{\em id}_\mathbb{H}$, for some $ \lambda>1 $ {\em (hyperbolic type)},
		\item $ \Omega= \mathbb{H} $, $ T=\textrm{\em id}_\mathbb{H} \pm1 $ {\em (simply parabolic type)}.
	\end{itemize}
\end{thm}

Finally, note that if $ g $ is a self-map of $ \mathbb{D} $, so is $ g^k $, for all $ k\geq1 $. The type in Cowen's classification is preserved by taking iterates.

\begin{lemma}{\bf (Cowen's classification for $ g^k $)}\label{lemma-cowen-invariant}
	Let $ g\colon \mathbb{D}\to\mathbb{D}$ holomorphic, and let $ k $ be a positive integer. Then, $ g $ is elliptic (resp. doubly parabolic, hyperbolic, simply parabolic) if and only if so is $ g^k $.
\end{lemma}
\begin{proof}
	It is clear that $ g $ is of elliptic type if and only if so is $ g^k $. Now, assume that $ p\in\partial \mathbb{D} $ is the Denjoy-Wolff point of $ g $, and choose a fundamental set $ V $ for $ g $ in $ \mathbb{D} $. Then, $ V $ is a fundamental set for $ g^k $ in $ \mathbb{D} $. It follows that $ g|_{V} $ is conformally conjugate to $ T_1\colon \Omega_1 \to \Omega_1$, and $ g^k|_{V} $ is conformally conjugate to $ T_2\colon \Omega_2 \to \Omega_2$. Hence, $T_1\colon \Omega_1 \to \Omega_1 $ and $ T_2\colon \Omega_2 \to \Omega_2$ are conformally conjugate. Since $ T $ and $ \Omega $ are unique up to conformal conjugacy, and do not depend on the choice of the fundamental set, it follows that $ g $ and $ g^k $ are of the same type in Cowen's classification.
\end{proof}
\subsection{Ergodic properties of the radial extension $ g^*\colon\partial\mathbb{D}\to  \partial\mathbb{D}$}
Let $ g\colon\mathbb{D}\to\mathbb{D} $ be an inner function. We consider the dynamical system induced by its the radial extension \[ g^*\colon\partial\mathbb{D}\to\partial\mathbb{D}.\] Recall that  if $ g $ is an inner function, so is $ g^k $ \cite[Lemma 4]{BakerDomínguez}, so the equality \[(g^n)^*(\xi)=(g^*)^n(\xi)\] holds $ \lambda $-almost everywhere. Moreover, the radial extension $ g^* $ is measurable (Thm. \ref{thm-mesurabilitat-funcions-disk}), and hence analyzable from the point of view of ergodic theory. The following is a recollection of ergodic properties of $ g^* $, with precise references.

\begin{thm}{\bf (Ergodic properties of $ g^* $)}\label{thm-ergodic-properties}
	Let $ g\colon \mathbb{D}\to\mathbb{D} $ be an inner function with Denjoy-Wolff point $ p\in\overline{\mathbb{D}} $. The following holds. \begin{enumerate}[label={\em (\alph*)}]
	
	\item  \label{ergodic-properties-nonsing} $ g^* $ is  non-singular. In particular, for $ \lambda $-almost every $ \xi\in\partial\mathbb{D} $, its infinite orbit under $ g^* $, $ \left\lbrace (g^n)^*(\xi)\right\rbrace _n $, is well-defined.
	
		\item   \label{ergodic-properties-a} $ g^* $ is ergodic if and only if $ g $ is elliptic or doubly parabolic.
		
		\item  \label{ergodic-properties-b} If $ g^* $ is recurrent, then it is ergodic. In this case,   for every $ A\in \mathcal{B} (\mathbb{D})$ with $ \lambda(A)>0 $, we have that for $ \lambda $-almost every $ \xi\in\partial\mathbb{D} $, there exists  a sequence $ n_k\to\infty $ such that $  (g^{n_k})^*(\xi)\in A $. In particular, $ \lambda $-almost every $ \xi\in\partial\mathbb{D} $, $ \left\lbrace (g^n)^*(\xi) \right\rbrace_n  $ is dense in $ \partial\mathbb{D} $.
		\item \label{ergodic-properties-c} If $ g $ is an elliptic inner function, then $ g^* $ is recurrent.

		\item \label{ergodic-properties-d}	The radial extension of a doubly parabolic inner function is not recurrent in general. However, if $ g $ is  doubly parabolic  and the Denjoy-Wolff point $ p $ is not a singularity for $ g $, then $ g^* $ is recurrent. Moreover, if $ g $ is doubly parabolic and there exists $ z\in\mathbb{D} $ and $ r>1 $ such that \[\textrm{\em dist}_\mathbb{D}(g^{n+1}(z), g^n(z))\leq \frac{1}{n}+O\left( \frac{1}{n^r}\right), \] as $ n\to\infty $, then $ g^* $ is recurrent.
		\item  \label{ergodic-properties-e} Let $ k $ be a positive integer. Then, $ g^k $ is an inner function. Moreover, $ g^* $ is ergodic (resp. recurrent) if and only if $ (g^k)^*$ is ergodic (resp. recurrent).
	\end{enumerate}
\end{thm}

\begin{proof}
	\begin{enumerate}[label={(\alph*)}]
		\item The proof that $ g^* $ is non-singular can be found in \cite[Prop. 6.1.1]{Aaronson97}. We claim that this already implies that, for $ \lambda $-almost every $ \xi\in\partial\mathbb{D} $,  its infinite orbit $ \left\lbrace (g^n)^*(\xi)\right\rbrace _n $ is well-defined. We shall prove it by induction. First, it is clear that the set \[\left\lbrace \xi\in\partial\mathbb{D}\colon g^*(\xi)\textrm{ is well-defined}\right\rbrace \]has full measure, and, since $ g^*$ is non-singular, $ g^*(\partial\mathbb{D}) =1$. Now, assume that the set \[\left\lbrace \xi\in\partial\mathbb{D}\colon\left\lbrace  (g^n)^*(\xi)\right\rbrace _{n=0}^{k-1}\textrm{ is well-defined}\right\rbrace \] has full measure, and $ \lambda((g^{k-1})^*(\partial\mathbb{D}))=1 $. Then, the set \[(g^{k-1})^*(\partial\mathbb{D})\cap\left\lbrace \xi\in\partial\mathbb{D}\colon (g^n)^*(\xi)\textrm{ is well-defined}\right\rbrace  \] has also full measure, proving that the orbit $ \left\lbrace  g^*(\xi)\right\rbrace _{n=0}^{k} $ is well-defined for $ \lambda $-almost every $ \xi\in\partial\mathbb{D} $, as desired.
		\item  It follows from combining \cite[Thm. G]{DoeringMané1991} with \cite[Thm. 1.4]{Bonfert}. 
		\item  See \cite[Thm. E, F]{DoeringMané1991}, as well as  \cite[Thm. 6.1.7]{Aaronson97}. 
		\item The first statement is \cite[Corol. 1]{DoeringMané1991} (see also \cite[Thm. 6.1.8]{Aaronson97}). The second statement follows from applying Theorem \ref{thm-almost-every-orbit-dense}, and applying that open sets in $ \partial\mathbb{D} $ have positive measure. See also \cite[Sect. 8.3]{befrs-dw}.
		\item An example of a doubly parabolic inner function whose boundary map is not recurrent is given in \cite[Example 1.3]{bfjk-escaping}.  Conditions which imply recurrence 
		are found in \cite[Thm. B]{bfjk-escaping} and \cite[Thm. E]{bfjk-escaping}, respectively. 
		\item 	Finally, the proof that $ g^k $ is an inner function can be found in \cite[Lemma 4]{BakerDomínguez}.  Since ergodicity depends only on the type in Cowen's classification \ref{teo-cowen}, which is invariant under talking iterates (Lemma \ref{lemma-cowen-invariant}), $ g^* $ is ergodic  if and only if $ (g^k)^*$ is ergodic. Recurrence is always preserved under taking iterates (see Lemma \ref{lemma-ergodic-properties-invariant}).
	\end{enumerate}
\end{proof}
\section{Inverse branches for inner functions at boundary points. Proof of \ref{thm-inverse-inner}}\label{sect-teoA}
Let $ g\colon\mathbb{D}\to\mathbb{D} $ be an inner function, considered as its maximal meromorphic extension \[g\colon\widehat{\mathbb{C}}\smallsetminus E(g)\to\widehat{\mathbb{C}},\]as before. We shall consider regular and singular values for $ g $ as introduced in Section \ref{subsect-regular-singular}. We observe that for any inner function $ g $ and any $ z\in E(g)$, the cluster set (as defined in Section \ref{subsect-regular-singular}) is $ \widehat{\mathbb{C}} $. Indeed, if $z\in E(g)$, then Lemma \ref{lemma-sing} and Schwarz reflection imply that $ Cl(g, z) =\widehat{\mathbb{C}}$ (compare also with \cite{BakerDomínguez}). Thus, $ g\in\mathbb{M} $, and all the previous description of singular values applies. 

Note that, since $ \mathbb{D} $ and $ \widehat{\mathbb{C}}\smallsetminus\overline{\mathbb{D}}$ are totally invariant, there cannot be critical points nor critical values in $ \partial\mathbb{D} $. 
Moreover, by symmetry, if $ v\in\mathbb{D} $ is regular (resp. singular), then so is $ 1/\overline{v} $. Hence, we consider \[SV(g, \overline{\mathbb{D}})=\left\lbrace v\in\overline{\mathbb{D}}\colon v\textrm{ is singular}\right\rbrace .\]
We start by proving that asymptotic paths actually land at points in $ E(g) $. 
\begin{lemma}{\bf (Asymptotic paths land)}\label{lemma-asymptotic-paths-inner}
	Let $ v\in\overline{\mathbb{D}} $ be an asymptotic value for $ g $, and let $ \gamma\colon \left[ 0,1\right) \to \widehat{\mathbb{C}}\smallsetminus E(g)$ be an asymptotic path for $ v $. Then, 
	there exists  a singularity $ \xi\in E(g)\subset\partial\mathbb{D} $ such that $ \gamma(t)\to \xi $, as $ t\to 1 $. 
\end{lemma}
\begin{proof}
	Assume, on the contrary, that the landing set $ L(\gamma) $ of the asymptotic path $ \gamma $ is  a continuum in $ E(g) $. Then, $ L(\gamma) $ is a closed non-degenerate interval in the unit circle. 
	
	On the one hand, for $ \lambda $-almost every point $ \xi $ in $ L(\gamma) $, the radial limit $ g^*(\xi) $ exists. Let us denote the radial segment  at $ \xi $ by $ R_\xi $.
	
On the other hand, without loss of generality, we can assume $  \gamma\colon \left[ 0,1\right) \to \overline{\mathbb{D}}\smallsetminus E(g)$. Then,  
since $ L(\gamma)\subset E(g) $ and $ \gamma\subset  \overline{\mathbb{D}}\smallsetminus E(g)$, for every point $ \xi\in L(\gamma) $ (except at the endpoints), there exists a sequence $\left\lbrace  \xi_n\right\rbrace \subset \gamma \cap R_\xi $ with $ \xi_n\to\xi $. Then, $ g (\xi_n)\to v $, implying that the radial limit $ g^* (\xi)$ equals $ v $. This contradicts the fact that radial limits are different almost everywhere (Thm. \ref{thm-FatouRiez}).
	\end{proof}

Next we prove that singular values in $ \partial\mathbb{D} $ correspond to accumulation points of singular values in $ \mathbb{D} $.
\begin{prop}{\bf (Singular values in $ \partial\mathbb{D} $)}\label{prop-inverse-branches}
	Let $ g\colon\mathbb{D}\to\mathbb{D} $ be an inner function, and let $ \xi\in\partial\mathbb{D} $. The following are equivalent.\begin{enumerate}[label={\em (\alph*)}]
		\item\label{SVa} There exists a crosscut $ C $, with crosscut neighbourhood $ N_C $ and $ \xi\in \partial N_C $ such that $ SV(g)\cap N_C =\emptyset$.
		\item\label{SVb} $ v $ is regular, i.e. there exists $ \rho\coloneqq \rho(\xi)>0 $ such that all inverse branches $ G_1 $ of $ g $ are well-defined in $ D(\xi, \rho) $.
	\end{enumerate}
\end{prop}
\begin{proof}
	The implication {\em \ref{SVb}}$ \Rightarrow ${\em \ref{SVa}} is trivial. Let us prove {\em \ref{SVa}}$ \Rightarrow ${\em \ref{SVb}}. Without loss of generality, we can assume that there are no singular values in $ \partial\mathbb{D}\cap\overline{N_C} $. Moreover, note that {\em \ref{SVa}} implies that all inverse branches $ G_1 $ are well-defined (and conformal) in $ N_C $. The assumption that there are no singular values in $ \partial\mathbb{D}\cap\overline{N_C} $ implies that $ G_1 $ is holomorphic in $ \overline{N_C}\cap \mathbb{D} $.
	We shall show that $ G_1 $ can be extended across $ \partial\mathbb{D}  $   by Schwarz reflection (see Fig. \ref{fig-SV}). 
	
	To this end, let $ \varphi\colon\mathbb{D}\to N_C $ be a Riemann map. Note that it extends homeomorphically to $ \partial N_C $; we shall denote this extension again by $ \varphi $. Then, \[ G_1\circ\varphi\colon\mathbb{D}\to  G_1(N_C)\] is a Riemann map for the simply connected domain $ G_1(N_C) $, where $ G_1 $ is any branch of $ g^{-1} $.
	
	 Consider the radial extension \[G^*_1\colon\partial N_C\to \partial G_1(N_C),\] defined as \[\ G_1^*(x)= (G_1\circ\varphi)^*(\varphi^{-1}(x)), \] for $ x\in\partial N_C $. Note that, since $ G_1|_{N_C}$ is bounded, the radial extension is well-defined almost everywhere in $ \partial N_C $. Since we assume that $ G_1 $ is holomorphic in $ \overline{N_C}\cap \mathbb{D} $, it follows that $ G_1(\overline{N_C}\cap \mathbb{D})\subset  \mathbb{D}$. Indeed, assume $ z\in\overline{N_C}\cap\mathbb{D} $ and $ G_1(z)\in\partial\mathbb{D} $. Since $ G_1 $ is conformal, it would map points in $ \mathbb{D} $ to points in $ \widehat{\mathbb{C}}\smallsetminus\overline{\mathbb{D}} $, a contradiction.
	 
	  Moreover, modifying slightly the crosscut if needed, we can assume that, for the two endpoints $ \left\lbrace \xi_1, \xi_2\right\rbrace =\partial N_C \cap\partial\mathbb{D} $, the limit \[G_1^*(\xi_i)=\lim\limits_{z\to\xi_i, z\in N_C\cap\mathbb{D}} G_1(z)\]  exists, for $ i=1,2 $. This is possible since the radial extension is well-defined almost everywhere. Note that, since $ g $ is assumed to be an inner function,  $ G_1^*(\xi_i)\in \partial\mathbb{D} $ (otherwise, there would exist a point in $ \mathbb{D} $ mapped by $ g $ to $ \partial\mathbb{D} $, a contradiction). Hence, $ G_1(\partial N_C\cap \mathbb{D})$ is a crosscut in $ \mathbb{D} $; we shall denote this crosscut by $ C' $.
	 On the other hand, note that, for $ x\in \partial N_C\cap \partial\mathbb{D} $, it holds
	 \[Cl(G_1, x)\coloneqq \left\lbrace w\in \overline{\mathbb{D}}\colon \textrm{ there exists }\left\lbrace x_n\right\rbrace _n\subset N_C \textrm{ with } x_n\to x \textrm{ and } G_1(x_n)\to w\right\rbrace \subset \partial\mathbb{D},\] since $ g(\mathbb{D})\subset\mathbb{D} $.
	 Hence, \[\partial G_1(N_C)\subset C'\cup\partial\mathbb{D}.\]
	  Therefore, $ \partial G_1(N_C) $ is locally connected (and, in fact, a Jordan curve), so $ G_1 $ extends homeomorphically to $ \partial N_C $,
	and $ G_1(\partial N_C\cap \partial\mathbb{D})\subset  \partial\mathbb{D}$. 
	
	Thus, by Schwarz reflection, we can extend holomorphically $ G_1 $ to \[N_C\cup (\partial N_C\cap \partial\mathbb{D})\cup \left\lbrace z\in\widehat{\mathbb{C}}\colon \frac{1}{\overline{z}}\in N_C\right\rbrace ,\]  (see Fig. \ref{fig-SV}). In particular, there exists $ \rho\coloneqq \rho(\xi) >0 $ such that all inverse branches of $ g $ are well-defined in $ D(\xi, \rho) $, as desired.
	
			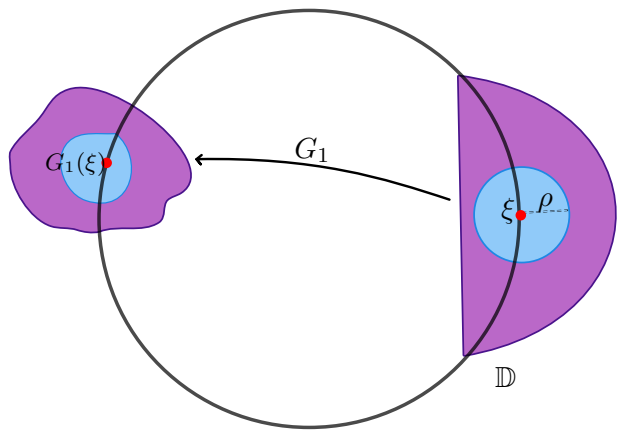
\begin{figure}[htb!]\centering
		\includegraphics[width=8cm]{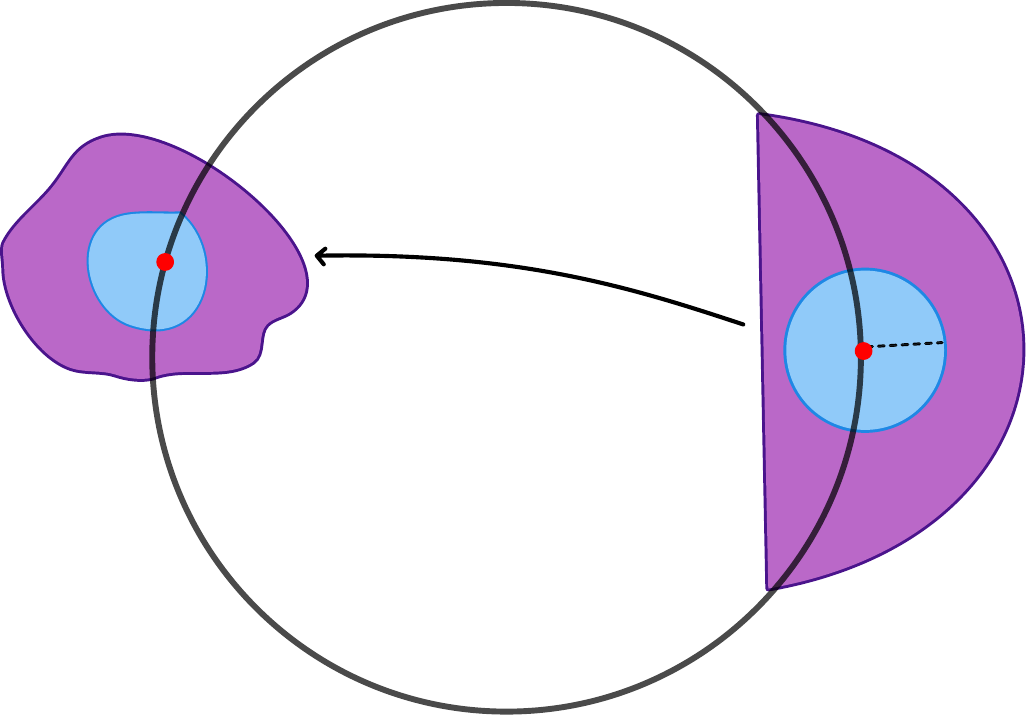}
			\setlength{\unitlength}{8cm}
		\put(-0.53, 0.45){$G_1$}
				\put(-0.13, 0.37){$\rho$}
				\put(-0.94, 0.43){\footnotesize$G_1(\xi)$}
						\put(-0.2, 0.07){$\mathbb{D}$}
		\put(-0.19,0.35){$\xi$}
		\caption{\footnotesize Whenever an inverse branch $ G_1 $ is well-defined in a crosscut neighbourhood, it can be extended across the unit circle by Schwarz reflection. }\label{fig-SV}
	\end{figure}
\end{proof}
It follows from Proposition \ref{prop-inverse-branches} that a value $ v\in\partial\mathbb{D} $ is singular for $ g $ if and only if it is accumulated by singular values in $ \mathbb{D} $, i.e. $ v\in \overline{SV(g)\cap\mathbb{D}} $. Clearly, for finite Blaschke products, all values $ v\in\partial\mathbb{D} $ are regular, and the same is true if $ SV(g)\cap\mathbb{D} $ is compactly contained in $ \mathbb{D} $. Moreover,
\begin{corol}{\bf (Non-singular Denjoy-Wolff point)}\label{coro-nonsingDW}
	Let $ g\colon\mathbb{D}\to\mathbb{D} $ be an inner function with Denjoy-Wolff point $ p\in\partial\mathbb{D} $. If $ p\notin \overline{SV(g)} $, then $ p $ is not a singularity for $ g $.
\end{corol}
\begin{proof}
	If $ p\notin \overline{SV(g)} $, then there exists a crosscut neighbourhood $ N_p $ such that $ p\in \partial N_p $ and $ SV(g)\cap N_p=\emptyset $. Since the Denjoy-Wolff point is radially fixed (Thm. \ref{wolfflemma}), there exists a curve $ \gamma\subset N_p $ landing at $ p $, such that $ g(\gamma)\subset N_p $ also lands at $ p $. Consider $ G_1 $ the inverse branch of $ g^{-1} $ defined in $ N_p $ such that $ G_1(g(\gamma))=\gamma $. By Proposition \ref{prop-inverse-branches}, $ G_1 $ extends conformally to $ D(p,\rho) $ for some $ \rho >0 $, and $ G_1(p)=p $. Then, $ D_1\coloneqq G_1(D(p, \rho)) $ is a negihbourhood of $ p $, and \[g\colon D_1\to D(p, \rho)\] conformally. Therefore, by Lemma \ref{lemma-sing}, $ p $ is not a singularity for $ g $, as desired.
\end{proof}

\begin{remark}{\bf (At most one asymptotic value per singularity)} 
	By 	the Lehto-Virtanen Theorem \ref{thm-lehto-virtanen}, given a singularity $ \xi\in E(g) $, there exists at most one asymptotic value $ v\in\overline{\mathbb{D}} $ corresponding to $ \xi $. Indeed, if $ v $ is an asymptotic value corresponding to the singularity $ \xi $, there exists a curve landing at $ \xi $ whose image lands at $ v $. By Lehto-Virtanen Theorem, $ g^*(\xi)=v $. Since radial limits, if they exist, are unique, there cannot be more asymptotic values corresponding to $ \xi $. Hence,
	\[\# \textrm{Sing} (g)\geq \# AV(g)\cap\overline{\mathbb{D}}.\]
	
	This differs from when a meromorphic function $ f\colon\mathbb{C}\to\widehat{\mathbb{C}} $ is considered, where the essential singularity (infinity) can have infinitely many asymptotic values corresponding to it.
\end{remark}

Next, we analyse the distortion induced by inverse branches near $ \partial\mathbb{D} $, and how we can control the preimages of radial limits in terms of Stolz angles.
\begin{prop}{\bf (Control of radial limits in terms of Stolz angles)}\label{prop-radial-limits}
Let $ g\colon\mathbb{D}\to\mathbb{D} $ be an inner function with Denjoy-Wolff point $ p\in\overline{\mathbb{D}} $. Let $ \xi\in\partial\mathbb{D} $, $ \xi\neq p $.
Assume there exists $ \rho_0>0$ such that $ D(\xi, \rho_0)\cap SV(g)\neq 0 $.  Then, for all $0< \alpha<\frac{\pi}{2} $, there exists $ \rho_1\coloneqq \rho_1(\alpha, \rho_0)<\rho_0 $ such that all branches $ G_1 $ of $ g^{-1} $ are well-defined in $ D(\xi, \rho_1) $ and, for all $ \rho<\rho_1 $,\[G_1(R_\rho(\xi, p))\subset \Delta_{\alpha,\rho}(G_1(\xi), p),\] where $ R_\rho(\cdot, p) $ and $ \Delta_{\alpha,\rho}(\cdot, p) $ stand for the generalized radial segment and Stolz angle with respect to $ p $ {\em (Def. \ref{defi-radi-stolz-angle})}. 
\end{prop}
		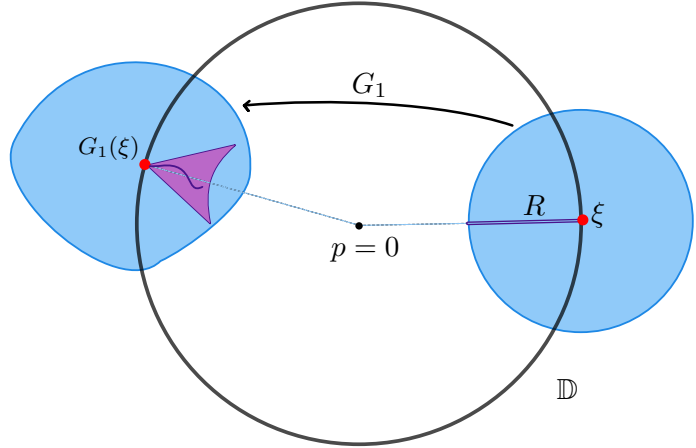
\begin{figure}[htb!]\centering
		\includegraphics[width=9cm]{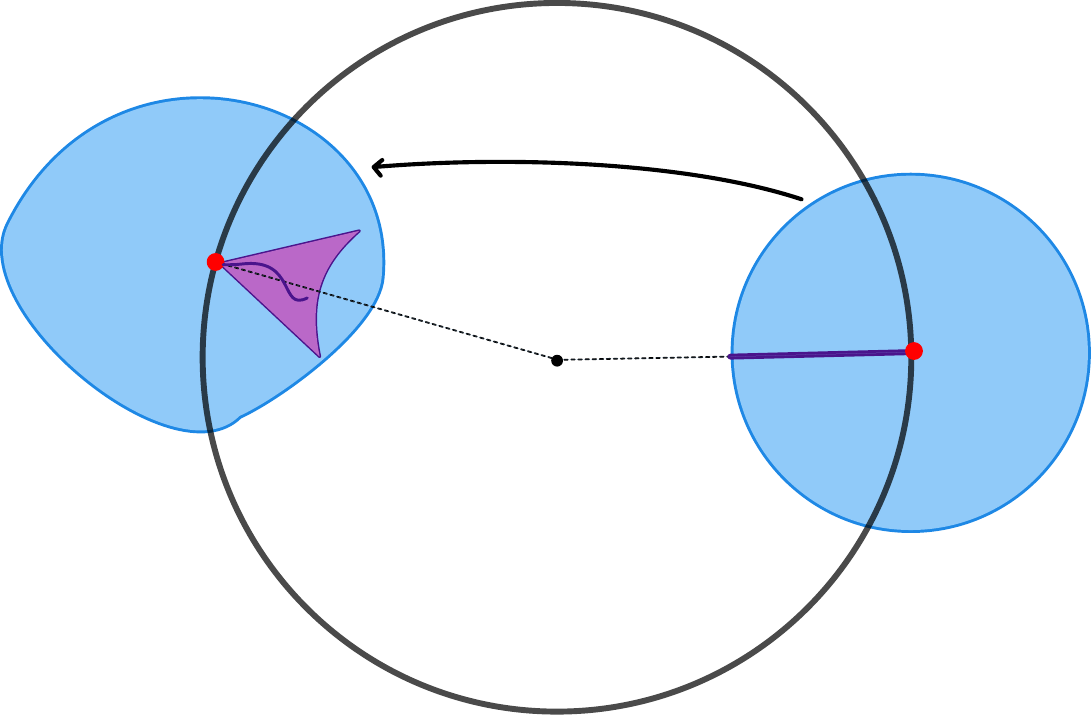}
	\setlength{\unitlength}{9cm}
	\put(-0.5, 0.52){$G_1$}
	\put(-0.25, 0.34){$R$}
	\put(-0.9, 0.43){\footnotesize$G_1(\xi)$}
	\put(-0.2, 0.07){$\mathbb{D}$}
		\put(-0.53, 0.28){$p=0$}
	\put(-0.15,0.33){$\xi$}
	\caption{\footnotesize Whenever an inverse branch $ G_1 $ is well-defined at a boundary point $ \xi\in\partial\mathbb{D} $, $ G_1 $ sends radial segments into angular neighbourhoods of a given openning. In the figure, $ p=0\in\mathbb{D} $.}\label{fig-radiallim}
\end{figure}

Note that $ \rho_1 $ depends only on $ \rho_0 $ and $ \alpha $, but not on the point $ \xi\in\partial \mathbb{D} $, nor on the inverse branch $ G_1 $.

\begin{proof}
	Note that, since $ D(\xi, \rho_0)\cap SV(g)\neq \emptyset $, all branches $ G_1 $ of $ g^{-1}$ are well-defined in $ D(\xi, \rho_0) $. We shall distinguish two cases. \begin{itemize}
		\item Assume first that $ g $ is elliptic, so $ p\in\mathbb{D} $. According to Definition \ref{defi-radi-stolz-angle}, it is enough to consider $ g\colon\mathbb{D}\to\mathbb{D} $ with $ g(0)=0 $ and prove \[G_1(R_\rho(\xi, 0))\subset \Delta_{\alpha,\rho}(G_1(\xi), 0).\]
		
	By Schwarz lemma \ref{scharzlemma}, $ \left| G_1(z)\right| \geq\left| z\right|  $ for $ z\in D(\xi, \rho_0) \cap\mathbb{D}$. It is left to see that, for $ z\in  R_\rho(\xi, 0)$, \[\left| \textrm{Arg }G_1(\xi)- \textrm{Arg }(G_1(\xi)-G_1(z))\right| <\alpha.\] To do so, consider the linear map \[L_1(z)\coloneqq G_1(\xi)+ G'_1(\xi)(z-\xi). \] Note that $ \left| L_1(z)-G_1(\xi)\right| =\left| G'_1(\xi)\right| \left| z-\xi\right|  $. Moreover, by Corollary \ref{corol-distotion}, there exists $ \rho_1<\rho_0 $ and a constant $ C(\rho_1)>0 $ such that \[\left| G_1(z)-L_1(z)\right| \leq C(\rho_1)\left| z-\xi\right| \left| G'_1(\xi)\right| ,\]for all $ z\in D(\xi, \rho_1) $. That is, the point $ G_1(z) $ belongs to the disk of center $ L_1(z)$ and radius $C(\rho_1)\left| z-\xi\right| \left| G'_1(\xi)\right|  $ (see Fig. \ref{fig-stolzangle1}).
	
	 Since $ C(\rho_1) \to 0$ as $ \rho_1 \to 0$, we have $ \frac{C(\rho_1)}{1-C(\rho_1)} \to 0$ as $ \rho_1 \to 0$. Without loss of generality, we assume $ C(\rho_1) $ satisfies \[\frac{C(\rho_1)}{1-C(\rho_1)} <\tan\alpha . \]  
	Let \[	\beta \coloneqq \left| \textrm{Arg }G_1(\xi)- \textrm{Arg }(G_1(\xi)-G_1(z))\right|. \]
We claim that, if $ z\in R_\rho(\xi, 0) $, then $ \textrm{Arg }L_1(z)= \textrm{Arg }(G_1(\xi))$. Indeed, $ L_1 $ is the affine map associated to $ G_1 $, which is an inverse branch of $ g $. The map $ G_1 $ is conformal in $ D(\xi, \rho_0) $, and hence angle preserving, and $ G_1(D(\xi, \rho_0)\cap\partial\mathbb{D})\subset \partial \mathbb{D} $. From this, it follows that, if $ \textrm{Arg }z= \textrm{Arg }\xi$, then $ \textrm{Arg }L_1(z)= \textrm{Arg }(G_1(\xi))$, i.e. if $ z $ lies on the radial segment at $\xi $, then $ L_1(z) $ lies on the radial segment at $ G_1(\xi) $.
	
		Then, noting that $ G_1(z) $ belongs to the disk of center $ L_1(z)$ and radius $C(\rho_1)\left| z-\xi\right| \left| G'_1(\xi)\right|  $, it follows
		\[\tan\beta\leq \dfrac{C(\rho_1)\left| G'_1(\xi)\right| \left| z-\xi\right| }{(1-C(\rho_1))\left| G'_1(\xi)\right| \left| z-\xi\right| }= \frac{C(\rho_1)}{1-C(\rho_1)} \leq \tan\alpha,\]as desired. See also Figure \ref{fig-stolzangle1}. 
		\begin{figure}[h]
			\centering

			\begin{subfigure}{0.3\textwidth}
					\centering
				\includegraphics[width=\textwidth]{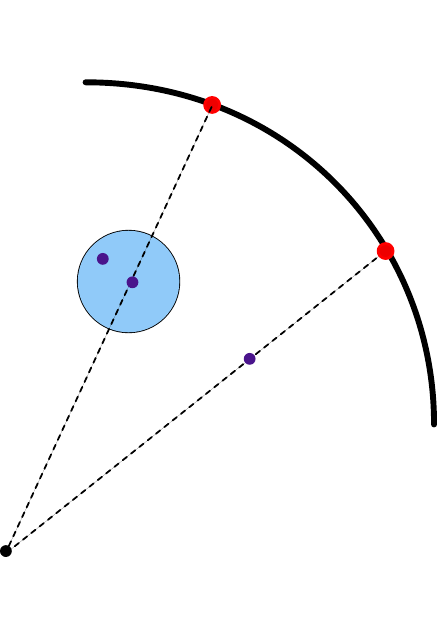}
				\setlength{\unitlength}{\textwidth}
				\put(-0.08, 0.87){$\xi$}
				\put(-0.5, 1.22){$G_1(\xi)$}
					\put(-1, 0.1){$0$}
								\put(-0.4, 0.6){$z$}
				\put(-0.67, 0.75){\footnotesize$L_1(z)$}
							\put(-1, 0.85){\footnotesize$G_1(z)$}
						\put(-0.05, 0.4){$\partial\mathbb{D}$}
				\caption{\footnotesize }
			\end{subfigure}
			\hspace{4cm}
			\begin{subfigure}{0.3\textwidth}
				\centering
				\includegraphics[width=0.8\textwidth]{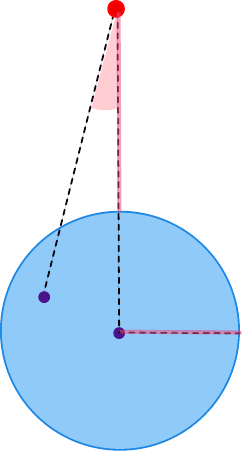}
				\setlength{\unitlength}{\textwidth}
				\put(-0.37, 1.43){$G_1(\xi)$}
				\put(-0.49, 1.05){$\beta$}
				\put(-0.5, 0.3){$L_1(z)$}
					\put(-0.4, 0.43){\tiny$C(\rho_1)\left|G_1'(\xi) \right| \left| z-\xi\right| $}
								\put(-0.4, 1){\tiny$(1-C(\rho_1))\left|G_1'(\xi) \right| \left| z-\xi\right| $}
				\put(-0.7, 0.4){$G_1(z)$}
				\caption{\footnotesize }
			\end{subfigure}
					\caption{\footnotesize Image (a) shows how the inverse branch $ G_1 $ acts near $ \xi $, and (b) provides more detail on it. Indeed, the point $ G_1 (z)$ lies on the disk $ D(L_1(z), C(\rho_1)\left| G_1'(\xi)\right| \left| z-\xi\right| ) $. The angle $ \beta $ captures the opening of the vector $ G_1(z)-G_1(\xi) $ with respect to the vector $ L_1(z)-G_1 (\xi)$. Hence, $ \tan\beta $ is bounded above by the maximal distance between $ G_1(z) $ and $ L_1(z) $ divided by the minimal distance between $ G_1(\xi) $ and $ G_1(z) $. }\label{fig-stolzangle1}
		\end{figure}
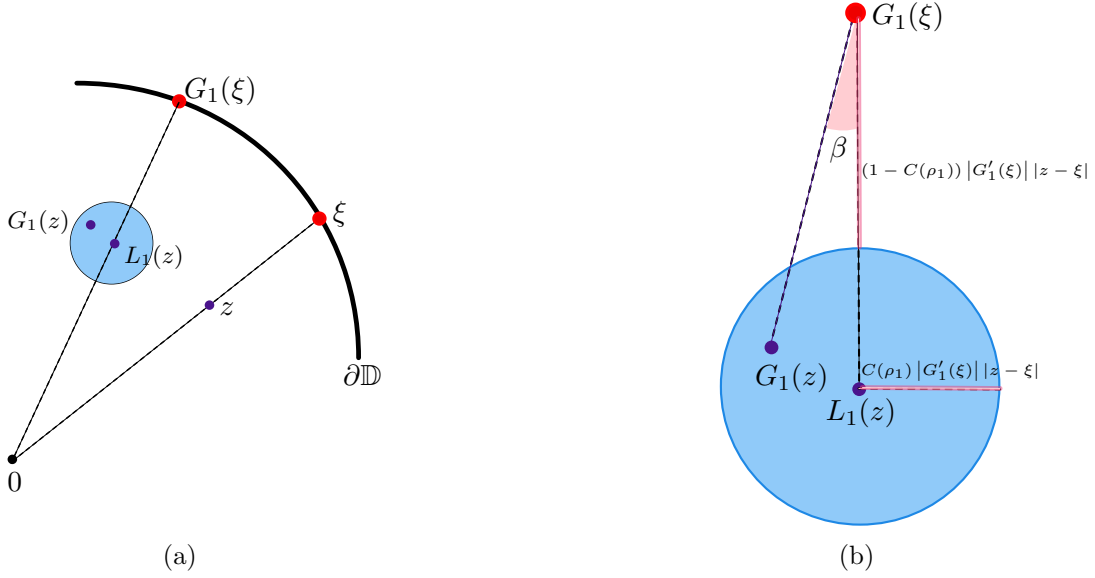
		\item Assume $ g $ is non-elliptic, so $ p\in\partial\mathbb{D} $. Note that $ G_1(\xi)\neq p $, since $ p $ is the Denjoy-Wolff point and, hence, it is radially fixed (Thm. \ref{wolfflemma}).
		
		Now, consider $ h\colon\mathbb{H}\to\mathbb{H} $, $ h\coloneqq M\circ g\circ M^{-1} $, where $ M\colon\mathbb{D}\to\mathbb{H} $, $ M(z)=i\dfrac{p+z}{p-z} $. Then, there exists $ \widetilde{\rho}_0 $ such that $ D(M(\xi), \widetilde{\rho}_0 )\cap SV(h)=\emptyset $, and consider $ H_1 $ the branch of $ h^{-1} $ corresponding to $ G_1 $, well-defined in $ D(M(\xi), \widetilde{\rho}_0 ) $. It is enough to prove that there exists $ \widetilde{\rho}_1 <\widetilde{\rho}_0  $ such that, for all $ \rho<\widetilde{\rho}_1 $,
		\[H_1(R^\mathbb{H}_\rho(M(\xi)))\subset \Delta^\mathbb{H}_{\alpha,\rho}(H_1(M(\xi))).\] First note that, by Wolff lemma \ref{wolfflemma}, if $ \textrm{Im }w<\rho $, then $ \textrm{Im }H_1(w)<\rho $.
		Now, consider the linear map \[ L_1(w)\coloneqq H_1(M(\xi))+ H'_1(M(\xi))(w-M(\xi)).\]
		Note that $ \left| L_1(w)-H_1(\xi)\right| =\left| H'_1(M(\xi))\right| \left| w-M(\xi)\right|  $. Moreover, by Corollary \ref{corol-distotion}, there exists $ \widetilde{\rho}_1<\widetilde{\rho}_0 $ and a constant $ C(\widetilde{\rho}_1)>0 $ such that \[\left| H_1(w)-L_1(w)\right| \leq C(\widetilde{\rho}_1)\left| w-M(\xi)\right| \left| H'_1(M(\xi))\right| ,\]for all $ w\in D(\xi, \widetilde{\rho}_1) $. We assume, without loss of generality, \[\dfrac{C(\widetilde{\rho}_1)}{1-C(\widetilde{\rho}_1)}<\tan\alpha .\]
		
		Since $ h(\mathbb{H})\subset \mathbb{H} $, and $ H_1 $ is a branch of $ h^{-1} $, conformal where defined, then if $ w\in  R_{\rho}^\mathbb{H}(M(\xi))$, then $ \textrm{Re }L_1(w)=H_1(M(\xi)) $, i.e. if $ w $ lies on the radial segment at $ M(\xi) $, then $L_1 (w) $ lies on the radial segment at $ L_1(M(\xi)) $. 
		
		We claim that, for $ w\in R_{\rho}^\mathbb{H}(M(\xi)) $, it holds \[ \dfrac{\left|\textrm{Re }H_1(w)-H_1(M(\xi))\right| }{\textrm{Im }H_1(w)}<\tan\alpha.\] Indeed, 
		\[\left| \textrm{Re }H_1(w)-H_1(M(\xi))\right| =\left| \textrm{Re }H_1(w)- \textrm{Re }L_1(w)\right|\leq C(\widetilde{\rho}_1)\left| w-M(\xi)\right| \left| H'_1(M(\xi))\right| ,\]\[\textrm{Im }H_1(w)=\textrm{Im }(H_1(w)-H_1(M(\xi)))\geq (1-C(\widetilde{\rho}_1)) \left| H_1'(M(\xi))\right| \left| w-M(\xi) \right| ,\]as desired.
		See also Figure \ref{fig-stolzangle2}. 
		
		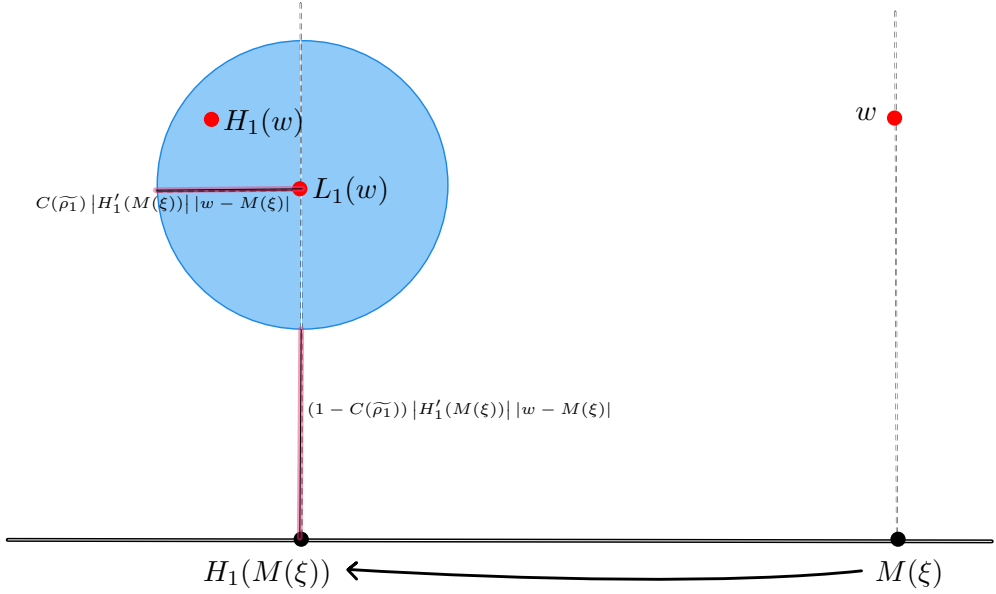
\begin{figure}[htb!]\centering
			\includegraphics[width=13cm]{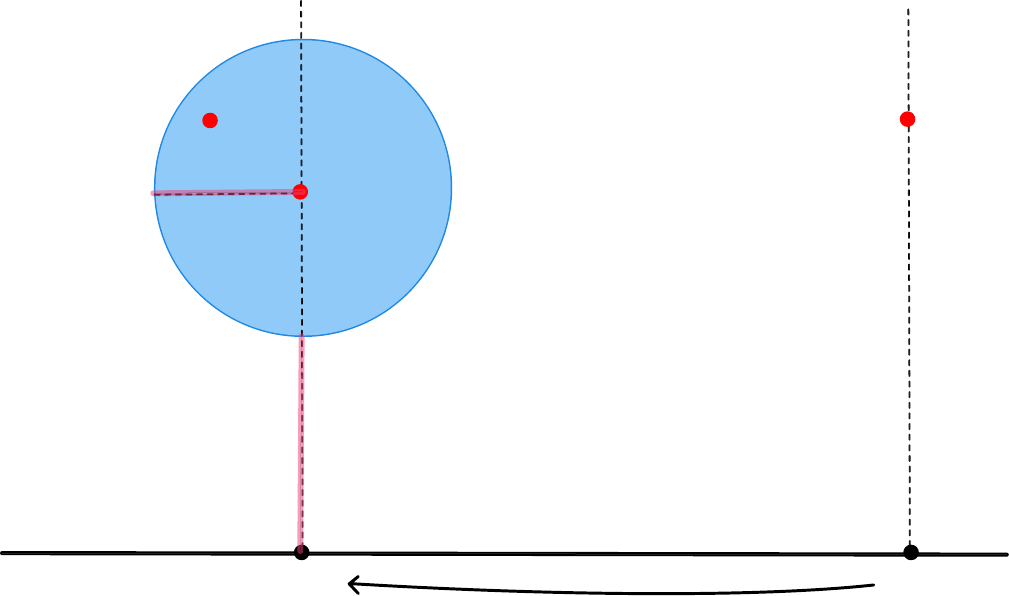}
				\setlength{\unitlength}{13cm}
			\put(-0.14, 0.47){$w$}
				\put(-0.12, 0.0){$M(\xi)$}
					\put(-0.8, 0.0){$H_1(M(\xi))$}
			\put(-0.69, 0.39){$L_1(w)$}
				\put(-0.78, 0.46){$H_1(w)$}
					\put(-0.97, 0.38){\tiny$C(\widetilde{\rho_1})\left|H_1'(M(\xi)) \right| \left| w-M(\xi)\right| $}
				\put(-0.695, 0.17){\tiny$(1-C(\widetilde{\rho_1}))\left|H_1'(M(\xi)) \right| \left| w-M(\xi)\right| $}
			\caption{\footnotesize The point $ H_1 (w)$ lies on the disk $ D(L_1(w), C(\widetilde{\rho}_1)\left| H_1'(\xi)\right| \left| z-\xi\right| ) $, and this gives the estimates on its real and imaginary part. }\label{fig-stolzangle2}
		\end{figure}
	\end{itemize}
\end{proof}

We are now ready to prove \ref{thm-inverse-inner}. In fact, we show a stronger statement (Thm. \ref{thm-postsingular}), from which \ref{thm-inverse-inner} is an immediate corollary.

Consider the \textit{postsingular set} \[ P(g)\coloneqq\overline{\bigcup\limits_{v\in SV(g)}\bigcup\limits_{n\geq 0} g^n(v)}.\] With our accurate study of the singular values, the following theorem is quite straight forward.

\begin{thm}{\bf (Inverse branches at boundary points)}\label{thm-postsingular}
	Let $ g\colon\mathbb{D}\to\mathbb{D} $ be an inner function, such that $ g^*|_{\partial\mathbb{D}} $ is recurrent. Assume there exists $ \zeta \in\partial\mathbb{D}$ and  a  crosscut neighbourhood $ N_\zeta$  of $ \zeta $ such that $ P(g)\cap N_\zeta =\emptyset$. Then,
	for $ \lambda $-almost every $ \xi\in\partial\mathbb{D} $, there exists $ \rho_0\coloneqq \rho_0(\xi)>0 $ such that all  branches $ G_n$ of $ g^{-n }$ are well-defined in $ D(\xi, \rho_0) $. In particular, the set $ E(g) $ of singularities of $ g $ has $ \lambda $-measure zero.
	
	\noindent In addition, for all $0< \alpha<\frac{\pi}{2} $, there exists $ \rho_1<\rho_0 $ such that, for all $ n\geq0 $, all branches $ G_n $ of $ g^{-n} $ are well-defined in $ D(\xi, \rho_1) $ and, for all $ \rho<\rho_1 $,\[G_n(R_\rho(\xi), p)\subset \Delta_{\alpha,\rho}(G_n(\xi),p),\] where $ R_\rho(\cdot, p) $ and $ \Delta_{\alpha,\rho}(\cdot, p) $ stand for the radial segment and the Stolz angle with respect to $ p $ {\em (Def. \ref{defi-radi-stolz-angle})}. 
\end{thm}
\begin{proof}
	By Proposition \ref{prop-inverse-branches}, the existence of  a  crosscut neighbourhood $ N_\zeta $  of $ \zeta\in\partial\mathbb{D} $ such that $ P(g)\cap N_\zeta =\emptyset$, implies the existence of $ \rho_\zeta>0 $ such that all branches $ G_n$ of $ g^{-n} $ are well-defined in $ D(\zeta, \rho_\zeta) $.
	
	We have to  see that, for  $ \lambda $-almost every $ \xi \in\partial\mathbb{D}$, there exists  $
	\rho_\xi >0$  such that all branches $ G_n$ of $ g^{-n} $ are well-defined in $ D(\xi, \rho_\xi) $.
	Since we are assuming $ g^* $ to be recurrent, for $ \lambda $-almost every $ \xi\in\partial\mathbb{D} $, $ \left\lbrace (g^n)^*(\xi) \right\rbrace_n  $ is dense in $ \partial\mathbb{D} $ (Thm. \ref{thm-ergodic-properties}{\em\ref{ergodic-properties-b}}). Therefore, there exists $ n_0\coloneqq n_0(\xi) $ such that $ (g^{n_0})^*(\xi) \in D(\zeta, \rho_\zeta)$. This already implies the existence of $ \rho_\xi>0 $ such that all inverse branches $ G_n $ of $ g^{-n} $ are well-defined in $ D(\xi, \rho_\xi) $.
	
	Next we prove that $ \lambda (E(g))=0 $. To do so, let 
	\[ K\coloneqq \left\lbrace \xi\in\partial\mathbb{D}\colon \exists\rho>0 \textrm{ such that all branches } G_n \textrm{ of }g^{-n} \textrm{ are well-defined in } D(\xi, \rho)\right\rbrace .\] Note that points in $ g^{-1}(K) $ do not belong to $ E(g) $. Indeed, if $ \zeta\in g^{-1}(K) $, then there exists a neighbourhood of $ \zeta $ which is mapped conformally to $ D(g(\zeta), \rho) $, and hence $ \zeta $ cannot be a singularity. Therefore, it is enough to prove that $ \lambda(g^{-1}(K))=1 $. This follows from the fact that $ \lambda (K)=1 $ and that $ g^* $ is non-singular (Thm. \ref{thm-ergodic-properties}{\em\ref{ergodic-properties-nonsing}}).
	
Finally,	the control of the image of radial segments by inverse branches in terms of angular neighbourhoods follows from Proposition \ref{prop-radial-limits}. Indeed, note that the estimates obtained therein do not depend on the inverse branches considered, but only on the radius of the disk where they are defined.
\end{proof}

\begin{proof}[Proof of \ref{thm-inverse-inner}]
It follows straightforward from applying Theorem \ref{thm-postsingular}.
\end{proof}

\begin{remark}
	Observe that a sufficient condition so that hypotheses of Theorem \ref{thm-postsingular} are satisfied is that singular values are compactly contained in $ \mathbb{D} $. Indeed, it is enough to show that, if singular values of $ g|_{\mathbb{D}} $ are compactly contained in $ \mathbb{D} $, then there exists $ \zeta \in\partial\mathbb{D}$ and  a  crosscut neighbourhood $ N_\zeta$  of $ \zeta $ such that $ P(g)\cap N_\zeta =\emptyset$. 
	
		Assume first that $ g $ is elliptic, with Denjoy-Wolff point $ p\in\mathbb{D} $. After conjugating by a Möbius transformation, we assume $ p=0 $. Now, consider an Euclidean disk $ D(0,r) $, with $ r\in(0,1) $ big enough so that $ SV(g)\subset D(0,r) $. By Schwarz Lemma \ref{scharzlemma}, $ D(0,r) $ is forward invariant under $ g $, so $ P(g) \subset D(0,r)$. This implies that, for all $ \zeta\in  \partial\mathbb{D}$, we can find a crosscut neighbourhood of $ \zeta$ disjoint from the postsingular set.
	
	If $ g $ is doubly parabolic, the procedure is analogous, with the difference that we should work with a holodisk thangent to the Denjoy-Wolff point instead of an Euclidean disk, and we apply Wolff Lemma \ref{wolfflemma} instead of Schwarz Lemma. Note that we can find a crosscut neighbourhood disjoint from the postsingular set for any $ \zeta \in \partial \mathbb{D}$, except for the Denjoy-Wolff point.
\end{remark}
\section{Invariant Fatou components of functions in class $ \mathbb{K} $}\label{sect-classeK}

As mentioned in the introduction, consider $ f\in\mathbb{K} $, i.e. \[f\colon \widehat{\mathbb{C}}\smallsetminus E(f)\to \widehat{\mathbb{C}},\] where $ \Omega(f)\coloneqq\widehat{\mathbb{C}}\smallsetminus E(f)  $ is the largest set where $ f $ is meromorphic and $ E(f) $ is the set of singularities of $ f $, which is assumed to be closed and countable. Note that $ \Omega(f)$ is open.

\vspace{0.2cm}
{\bf Notation.} Having fixed a function $ f\in\mathbb{K} $, we denote $ \Omega(f) $ and $ E(f) $ simply by $ \Omega $ and $ E $, respectively. Given a domain $ U\subset\Omega $, we denote by $ \partial U $ the boundary of $ U $ in $ \Omega $, and we keep the notation $ \widehat{\partial} U $ for the boundary with respect to $ \widehat{\mathbb{C}} $.

The dynamics of such functions was studied in \cite{Bolsch_repulsivepp,Bolsch-thesis,Bolsch-Fatoucomponents, BakerDominguezHerring,BakerDominguezHerring2,  Dominguez4,Dominguez3,}. The Fatou set $ \mathcal{F}(f) $ is defined as the largest open set in which $ \left\lbrace f^n\right\rbrace _n $ is well defined and normal, and the Julia set $ \mathcal{J}(f) $, as its complement in $ \widehat{\mathbb{C}} $. The standard theory of Fatou and Julia for rational or entire functions extends successfully to this more general setting. We need the following properties.

\begin{thm}{\bf (Properties of Fatou and Julia sets, {\normalfont \cite[Thm. A]{BakerDominguezHerring}})}
	Let $ f\in\mathbb{K} $. Then,
	\begin{enumerate}[label={\em (\alph*)}]
		\item\label{Fatou1} $ \mathcal{F}(f) $ is completely invariant in the sense that $ z\in\mathcal{F}(f) $ if and only if $ f(z)\in \mathcal{F}(f) $;
		\item\label{Fatou2} for every positive integer $ k $, $ f^k\in\mathbb{K} $, $ \mathcal{F}(f^k)=\mathcal{F}(f) $ and $ \mathcal{J}(f^k)=\mathcal{J}(f) $;
		\item $ \mathcal{J}(f) $ is perfect;
		\item repelling periodic points are dense in $ \mathcal{J}(f)  $.
	\end{enumerate}
\end{thm}

By {\em\ref{Fatou1}}, Fatou components (i.e. connected components of $ \mathcal{F}(f) $) are mapped among themselves, and hence classified into periodic, preperiodic or wandering. By {\em\ref{Fatou2}}, the study of periodic Fatou components reduces to the invariant ones, i.e. those for which $ f(U)\subset U $. Those Fatou components are classified into attracting basins, parabolic basins, Siegel disks, Herman rings and Baker domains \cite[Thm. C]{BakerDominguezHerring}. A \textit{Baker domain} is, by definition, a periodic Fatou component $ U $ of period $ k\geq 1 $ for which there exists $ z_0\in\widehat{\partial} U$ such that $ f^{nk}(z)\to z_0 $, for all $ z\in U $ as $ n\to\infty $, but $ f^k $ is not meromorphic at $ z_0 $. In such case, $ z_0 $ is accessible from $ U $  \cite[658]{BakerDominguezHerring}.

\begin{thm}{\bf (Connectivity of Fatou components, {\normalfont \cite{Bolsch-Fatoucomponents}})}	Let $ f\in\mathbb{K} $, and let $ U $ be a periodic Fatou component of $ f $. Then, the connectivity of $ U $ is 1, 2, or $ \infty $.\end{thm}

In this paper, we focus on simply connected periodic Fatou components, which we assume to be invariant. We analyse them by means of the Riemann map $ \varphi\colon\mathbb{D}\to U $.

\subsection{Inner function associated to an invariant Fatou component}\label{subsect-associated-inner-function}
Let $ f\in\mathbb{K} $, and let $ U $ be an invariant Fatou component for $ f $, which we assume to be simply connected. Consider $ \varphi\colon\mathbb{D}\to U $ to be a Riemann map. Then,  $ f\colon U\to U $ is conjugate by $ \varphi $ to a holomorphic map $ g\colon\mathbb{D}\to\mathbb{D} $, such that the diagram
\[\begin{tikzcd}
U \arrow{r}{f} & U \\	
	\mathbb{D} \arrow{r}{g} \arrow{u}{\varphi} & \mathbb{D}\arrow[swap]{u}{\varphi}
\end{tikzcd}
\]
commutes.  

For entire functions, where the unique essential singularity lies at $ \infty $, it is well-known that  $ g $ is an inner function (see e.g. \cite[Sect. 2.3]{efjs}, or \cite[Prop. 1.1]{FatousAssociates}). Actually, the same holds for functions in class $ \mathbb{K} $ (see Prop. \ref{prop-g-inner}), and we say that $ g $ is an {\em inner function associated to $ (f,U) $}. Note that two  inner functions associated to the same $ (f,U) $ are conformally conjugate and hence have the same dynamical behaviour.

As usual, singular values (defined in Sect. \ref{subsect-regular-singular}) play a distinguished role in the dynamics. It follows from \cite[Thm. 1.2]{Bolsch-thesis} that $ Cl(f, z)=\widehat{\mathbb{C}} $, for every $ z\in E(f) $ (see also \cite[p. 651]{BakerDominguezHerring}). Thus, $ \mathbb{K}\subset\mathbb{M} $, and the discussion in Section \ref{subsect-regular-singular} applies.

Note that, as in the case of inner functions, asymptotic paths land. This comes from the fact that $ \partial\Omega=E(f) $ is countable, and hence if the  landing set $ L(\gamma) $ of a curve $ \gamma \subset \Omega $ consists of more than two points $ \partial \Omega $, then $ L(\gamma)\cap\Omega\neq\emptyset $.

The dynamics of $ f|_U $ is controlled by the following subset of singular values.

\begin{defi}{\bf (Relevant singular values)}
Let $ f\in\mathbb{K} $, and let $ U $ be an invariant Fatou component for $ f $. We define:\begin{itemize}
	\item Relevant critical values for $ f|_U  $,\[ CV(f, U)\coloneqq \left\lbrace v\in  U\colon v=f(z), f'(z)=0\right\rbrace. \]
	\item Relevant asymptotic values for $ f|_U  $,\[ AV(f,{U})\coloneqq\left\lbrace v\in{U}\cap AV(f)\colon \textrm{ there exists an asymptotic path  } \gamma\subset U \textrm{ for }v \right\rbrace.\]
	\item Relevant singular values for $  f|_U $, $SV(f,{U})\coloneqq\overline{CV(f,{U})\cup AV(f,{U})}\cap U$.
	\item Relevant postsingular set for $ f|_U $, $ P(f, {U})\coloneqq \overline {\bigcup\limits_{v\in SV(f, {U})}\textrm{ }\bigcup\limits_{n\geq 0} f^n(v)}\cap U$.
\end{itemize}
\end{defi}
Note that there may be other singular values in $ U $, if another Fatou component is mapped into $ U $ non-conformally. However, these singular values do not play a role in the internal dynamics of $ (f, U) $, and the inner function does not see them.
Note that \[SV(g,\mathbb{D})=\varphi^{-1}(SV(f, U)),\] and  \[P(g,\mathbb{D})=\varphi^{-1}(P(f, U)).\]

It is well-known that $ SV(f, U) \neq \emptyset$ whenever $ U $ is an attracting or parabolic basin, or a doubly parabolic Baker domain. If $  SV(f, U) $ is compactly contained in $ U $, the behaviour of the associated inner function $ g $ is well understood, in the following sense.

\begin{prop}{\bf (Singular values compactly contained)}\label{prop-SVcompact}
Let $ f\in\mathbb{K} $, let $ U $ be an invariant simply connected Fatou component for $ f $, and $ g $ its associated inner function. Assume $ SV(f,U) $ is compactly contained in $ U $. Then, the following hold.
\begin{enumerate}[label={\em (\alph*)}]
	\item If $ U $ is an attracting basin, then, for all $ \xi\in\partial\mathbb{D} $, there exists a crosscut neighbourhood $ N_\xi $ of $ \xi$, such that $ N_\xi\cap P(g,\mathbb{D})=\emptyset $. 
	\item If $ U$ is either a parabolic basin or a Baker domain, then 
	the Denjoy-Wolff point $ p\in\partial\mathbb{D} $ of $ g $ is not a singularity for $ g $. Moreover, for all $ \xi\in\partial\mathbb{D} $, $ \xi\neq p $, there exists a crosscut neighbourhood $ N_\xi $ of $ \xi\in\partial\mathbb{D} $, such that $ N_\xi\cap P(g,\mathbb{D})=\emptyset $. 
\end{enumerate}
\end{prop}
\begin{proof}\begin{enumerate}[label={(\alph*)}]

		\item  Let $ z_0\in U $ be the attracting fixed point, and consider $ g $ to be the inner function associated by a Riemann map $ \varphi$, with $ \varphi(0)=z_0 $. Then, there exists $ r\in (0,1) $ big enough so that $SV(g,\mathbb{D})\subset D(0,r)$. 
		
		By Schwarz lemma \ref{scharzlemma}, $ D(0,r) $ is forward invariant, so $P(g,\mathbb{D})\subset D(0,r)$, and (a) follows trivially. 
		
				\item 	
		By Corollary \ref{coro-nonsingDW}, it is enough to find a crosscut neighbourhood $ N_p $ of $ p $ such that $ N_p\cap SV(g,\mathbb{D}) =\emptyset$, and this is immediate from the hypothesis.
		
		The second statement follows for applying the same argument as in (a), using a tangent disk at the Denjoy-Wolff point and Wolff Lemma \ref{wolfflemma}.
	\end{enumerate}
\end{proof}

\subsection{Ergodic properties of the boundary map $ f\colon\partial U\to  \partial U$}
Let $ \varphi\colon\mathbb{D}\to U $ be the Riemann map, as in Section \ref{subsect-associated-inner-function}.
Then, the radial extension \[\varphi^*\colon\partial\mathbb{D}\to \widehat{\partial}U\]is well-defined $ \lambda $-almost everywhere, and   $ \widehat{\partial}  U$ admits a harmonic measure $ \omega_U $, which stands for the push-forward of the normalized Lebesgue measure $ \lambda $ of $ \partial\mathbb{D} $ (see Sect. \ref{subsection-harmonic-measure}). The ergodic properties of $ f|_{\partial U} $ will be derived from the ergodic properties of $ g^* $, where $ g $ is the  inner function associated to $ (f, U) $. 

We start by proving that $ g $ is actually an inner function. To that end, consider the following subsets of $ \partial \mathbb{D} $.
\[\Theta_E\coloneqq \left\lbrace \xi\in\partial\mathbb{D}\colon \varphi^*(\xi)\in E(f) \right\rbrace  \] 
\[\Theta_\Omega\coloneqq \left\lbrace \xi\in\partial\mathbb{D}\colon \varphi^*(\xi)\in\Omega(f) \right\rbrace \] 
Note that, since $ E(f) $ is countable, $ \lambda(\Theta_E)=0 $, so $ \lambda(\Theta_\Omega)=1 $. Moreover, the conjugacy $ f\circ \varphi=\varphi\circ g $ extends for the radial extensions wherever it makes sense, as it is shown in the following lemma.
\begin{lemma}{\bf(Radial limits commute)}\label{lemma-radial-limit}
	Let $ \xi\in\Theta_\Omega $, then $ g^*(\xi) $ and $ \varphi^*(g^*(\xi)) $ are well-defined, and \[f(\varphi^*(\xi))=\varphi^*(g^*(\xi)).\]
	\end{lemma}
\begin{proof}
	Let $ r_\xi(t)=t\xi $, with $ t\in\left[ 0,1\right)  $. By assumption, $ \varphi(r_\xi(t))\to \varphi^*(\xi)\eqqcolon w\in\partial U $, as $ t\to1^- $ (recall that $ \partial U $ denotes the boundary of $ U $ taken in $ \Omega $). Since $ f $ is continuous at $ w $ and $ f\circ\varphi=\varphi\circ g $, for all $ 0<t<1 $,\[\varphi(g(r_\xi(t)))=f(\varphi(r_\xi(t)))\to f(w)\in\widehat{\partial}U,\] as $ t\to 1^{-} $. We claim that this already implies that $ \gamma(t)\coloneqq g(r_\xi(t)) $ lands at some $ \zeta\in\partial\mathbb{D} $. Indeed, consider \[L_{\gamma, 1}\coloneqq\left\lbrace z\in\overline{\mathbb{D}}\colon\textrm{ there exists }t_n\to1^-\textrm{ such that }\gamma(t_n)\to z\right\rbrace ,\] which is a non-empty, compact, connected set contained in $ \partial\mathbb{D} $, since points in $ \mathbb{D} $ are mapped to $ U $ by $ \varphi $. If $ L_{\gamma, 1} $ is a non-degenerate arc $ I $, with $ \lambda(I)>0 $, for $ \lambda $-almost every $ \zeta\in I $, $ \varphi^*(\zeta)=f(w) $, which is a contradiction with  Fatou's Theorem \ref{thm-FatouRiez}. Hence, $  L_{\gamma, 1}=\zeta\in\partial \mathbb{D}$.
	
	Finally, the Lehto-Virtanen Theorem \ref{thm-lehto-virtanen} implies that $  g^*(\xi)=\zeta$ and $ \varphi^*(g^*(\xi))=f(w) $, as desired. 
\end{proof}

\begin{prop}{\bf ($ g $ is inner)}\label{prop-g-inner}
	Let $ f\in\mathbb{K} $, and let $ U $ be an invariant simply connected Fatou component of $ f $. Then, the associated map $ g\colon\mathbb{D} \to\mathbb{D}$ is an inner function. 
\end{prop}
\begin{proof}
	Since $ g $ is a self-map of the unit disk, its radial extension $ g^* $ is well-defined $ \lambda $-almost everywhere. We have to see that $ \left| g^*(\xi)\right| =1 $, for $ \lambda $-almost every $ \xi\in\partial\mathbb{D} $. 
	
	Assume, on the contrary, that there exists $ A\subset\partial\mathbb{D} $ with $ \lambda (A)>0 $ and $ \left| g^*(\xi)\right| <1 $, for  all $ \xi\in A $. We can assume, without loss of generality, that $ \varphi^*(\xi)\in\partial U \subset \mathcal{J}(f)$, for all $ \xi\in A $. By Lemma \ref{lemma-radial-limit}, for all $ \xi\in A $, $ \varphi^*(\xi)\in\mathcal{J}(f)\smallsetminus E(f) $, and
	\[f(\varphi^*(\xi))=\varphi^*(g^*(\xi))\in U\subset\mathcal{F}(f)\] which is a contradiction with the total invariance of the Fatou set.
\end{proof}

We note that, since $ \omega_U(E(f))=0 $, for every Borel set $ B\subset\widehat{\mathbb{C}} $ and $ z\in U $, we have \[\omega_U(z,B)=\omega_U(z, B\cap\Omega(f)).\]

With these tools at hand, and those developed in the previous sections, we can now prove ergodic properties like ergodicity and recurrence for the boundary map of Fatou components of maps in class $ \mathbb{K} $, generalizing the results of Doering and Mañé \cite{DoeringMané1991} for rational maps.

\begin{thm}{\bf (Ergodic properties of the boundary map)}\label{thm-ergodic-boundary-map}
	Let $ f\in\mathbb{K}$, and let $ U $ be an invariant simply connected Fatou component for $ f $. Let  $ g $ be an inner function associated to $ (f, U) $. Then, the following are satisfied. \begin{enumerate}[label={\em (\roman*)}]
		\item If $ U $ is either an attracting basin, a parabolic basin, or a Siegel disk, then $ g^*|_{\partial\mathbb{D}} $ is ergodic and recurrent with respect to the Lebesgue measure $ \lambda$. 
		\item If $ U $ is a doubly parabolic Baker domain, $ g^*|_{\partial \mathbb{D}} $ is ergodic with respect to $ \lambda$. In addition,  assume one of the following conditions is satisfied. \begin{enumerate}[label={\em (\alph*)}]
			\item $f|_ U $ has finite degree. 
			\item Relevant singular values $ SV(f, U) $ are compactly contained in $ U $. 
			\item The Denjoy-Wolff point of $ g $ is not a singularity for $ g $.
		\item There exists $ z\in U $ and $ r>1 $ such that \[\textrm{\em dist}_U(f^{n+1}(z), f^n(z))\leq \frac{1}{n}+O\left( \frac{1}{n^r}\right), \] as $ n\to\infty $, where $ \textrm{\em dist}_U $ denotes the hyperbolic distance in $ U $.
		\end{enumerate}
		Then, $ g^*|_{\partial \mathbb{D}}  $ is  recurrent with respect to $ \lambda $. 

\item If $ g^*|_{\partial \mathbb{D}} $ is ergodic (resp. recurrent) with respect to $ \lambda $,  so is $ f|_{{\partial} U} $ with respect to $ \omega_U $. 

\noindent If $ g^*|_{\partial \mathbb{D}} $ is recurrent with respect to $ \lambda $, then for $ \omega_U $-almost every point $ x\in\partial U $, $ \left\lbrace f^n(x)\right\rbrace _n $ is dense in $ \partial U$.

		\item  Let $ k $ be a positive integer. Then, the inner function associated to $ (f,U) $ has the same ergodic properties than the inner function associated to $ (f^k,U) $.
	\end{enumerate}
\end{thm}

\begin{proof}
	\begin{enumerate}[label={ (\roman*)}]
		\item The associated inner function to these Fatou components is either elliptic or doubly parabolic, so $ g^* $ is ergodic (Thm. \ref{thm-ergodic-properties}{\em\ref{ergodic-properties-a}}).
		
	Recurrence for the inner function associated to a Siegel disk or an attracting basin follows from the fact that these inner functions are always elliptic (Thm. \ref{thm-ergodic-properties}{\em\ref{ergodic-properties-c}}).
		Recurrence for parabolic basins follows from \cite[Thm. 6.1]{DoeringMané1991} (although it is stated for rational maps, the proof only uses the local behaviour around the parabolic fixed point, so it is valid for $ f\in\mathbb{K} $).

		\item Ergodicity of doubly parabolic Fatou components follows from Theorem \ref{thm-ergodic-properties}{\em\ref{ergodic-properties-a}}.  Recurrence (under the stated conditions) follows from Theorem \ref{thm-ergodic-properties}{\em\ref{ergodic-properties-d}}.  Indeed, note that, if $ f|_U $ has finite degree, then relevant singular values $ SV(f, U) $ are compactly contained in $ U $. By Proposition \ref{prop-SVcompact}, in this case, the Denjoy-Wolff point $ p\in\partial\mathbb{D} $ is not a singularity for $ g $. Hence, by Theorem \ref{thm-ergodic-properties}{\em\ref{ergodic-properties-d}}, the boundary map $ g^* |_{\partial\mathbb{D}}$ is recurrent with respect to $ \lambda $. 
		
		To see that condition (d) implies recurrence, note that \[\textrm{dist}_U(f^{n+1}(z), f^n(z))= \textrm{ dist}_\mathbb{D}(g^{n+1}(\varphi^{-1}(z)), g^n(\varphi^{-1}(z))), \] and apply again Theorem \ref{thm-ergodic-properties}{\em\ref{ergodic-properties-d}}.

\item 
We start with ergodicity. Let $ A\subset\partial U $ be measurable. If $ A= f^{-1}(A)$, then \[(\varphi^*)^{-1}(A)= (\varphi^*)^{-1}(f^{-1}(A))=( g^*)^{-1}((\varphi^*)^{-1}(A)) .\] If $ g^* $ is ergodic, then $ \lambda((\varphi^*)^{-1}(A))=0 $ or $ \lambda((\varphi^*)^{-1}(A))=1 $. Then, $ \omega_U(A)=0 $ or $ \omega_U(A)=1 $, and $ f|_{\partial U} $ is ergodic. 

For the recurrence, assume $ A\subset\partial U $ is measurable, and consider $ (\varphi^*)^{-1}(A) $. Then, for $ \lambda $-almost every $ \xi\in (\varphi^*)^{-1}(A)  $, there exists $ n_k\to\infty $ such that $ (g^{n_k})^*(\xi)\in (\varphi^*)^{-1}(A)  $. Since $ A\subset\partial U\subset \Omega $, Lemma \ref{lemma-radial-limit} applies, and\[\varphi^*((g^{n_k})^*(\xi))=f^{n_k}(\varphi^*(\xi))\in A,\] proving recurrence for $ f|_{\partial U} $.
		Since $ \textrm{supp }\omega_U =\widehat{\partial} U$, it follows from 
		Theorem \ref{thm-almost-every-orbit-dense} that $ \omega_U $-almost every orbit is dense.
		\item It follows from Theorem \ref{thm-ergodic-properties}{\em\ref{ergodic-properties-e}}.
	\end{enumerate}
\end{proof}

\section{Density of periodic boundary points. Proof of \ref{teo:A}}\label{sect-demostracio}

With the tools developed in the previous sections we are now able to 
 prove a more general version of  \ref{teo:A}, namely Theorem \ref{thm-periodic-points-are-dense}.

\begin{thm}{\bf (Periodic boundary points are dense)}\label{thm-periodic-points-are-dense}
	Let $ f\in \mathbb{K}$, and let $U$ be an invariant simply connected Fatou component for $ f $. Let $ \varphi\colon\mathbb{D}\to U $ be a Riemann map, and let $ g\colon\mathbb{D}\to\mathbb{D} $ be the  inner function  associated to $ (f, U) $ by $ \varphi $. 
	\\ \noindent Assume the following conditions are satisfied.
	\begin{enumerate}[label={\em (\alph*)}]
		\item\label{hypothesisc} $ g^*|_{\partial \mathbb{D}}$ is recurrent with respect to $ \lambda$.
		\item \label{hypothesisa} There exists 
		 $ x_0\in\partial U $ and $ r_0\coloneqq r_0(x_0)>0 $ such that, for all $ n\geq 0 $, if $ D_n $ is a connected component of $ f^{-n}(D(x_0, r_0)) $ such that $ D_n\cap U\neq\emptyset $, then $ f^n|_{D_n} $ is conformal.
		\item\label{hypothesisb} There exists a crosscut $ C\subset\mathbb{D} $ and a crosscut neighbourhood $ N_C $ with $ N_C\cap P(g)=\emptyset $. 
	\end{enumerate}
Then, accessible periodic points are dense in $ \partial U $.
\end{thm}
Note that in Theorem \ref{thm-periodic-points-are-dense} we assume  $ U $ to be invariant, i.e. $ f(U)\subset U $. This is not restrictive:  indeed, if $ U $ was a $ p $-periodic Fatou component with $ p\geq 2 $, i.e. $ f^p(U)\subset U $, we shall replace $ f $ by $ f^p $. Note that $ f^p\in\mathbb{K}$, since class $  \mathbb{K}$ is closed under composition.

The proof of Theorem \ref{thm-periodic-points-are-dense} is postponed until the end of the section but we show now how to deduce \ref{teo:A} from it.

\begin{proof}[Proof of \ref{teo:A}]
	It is enough to show that, if the hypotesis {\em \ref{maini}}-{\em \ref{mainiii}} in \ref{teo:A} are satisfied, then {\em \ref{hypothesisc}}-{\em \ref{hypothesisb}} in Theorem \ref{thm-periodic-points-are-dense} also hold. First, if hypotesis {\em \ref{maini}} and {\em \ref{mainiii}} in \ref{teo:A} are satisfied, then  {\em \ref{hypothesisc}} holds, applying Theorem \ref{thm-ergodic-boundary-map}. Second, it is clear that {\em \ref{mainii}} implies {\em \ref{hypothesisa}}. Finally, {\em \ref{mainiii}} is equivalent to  {\em \ref{hypothesisb}}, by means of the Riemann map.
\end{proof}

Finally, we shall state some other conditions under which the hypothesis of Theorem \ref{thm-periodic-points-are-dense} are satisfied, and hence periodic points are dense in the boundary of the Fatou component. Note that the following conditions allow for the possibility of having infinitely many singular values accumulating at infinity.

First let $ U $ be an attracting or a parabolic basin. In this case, the boundary map is always recurrent. Hence, it is enough to ask that there exists a domain $ V \subset U$, such that  $$ P(f)\cap \overline{U}\subset \overline{V}$$ and $ \omega_U(\overline{V}) =0$ (see Fig. \ref{fig-starlike}).

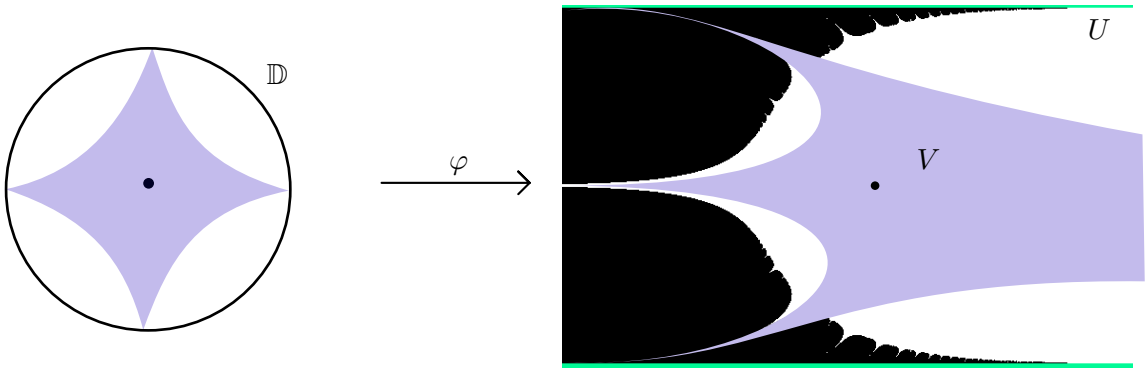
\begin{figure}[htb!]\centering
	\includegraphics[width=15cm]{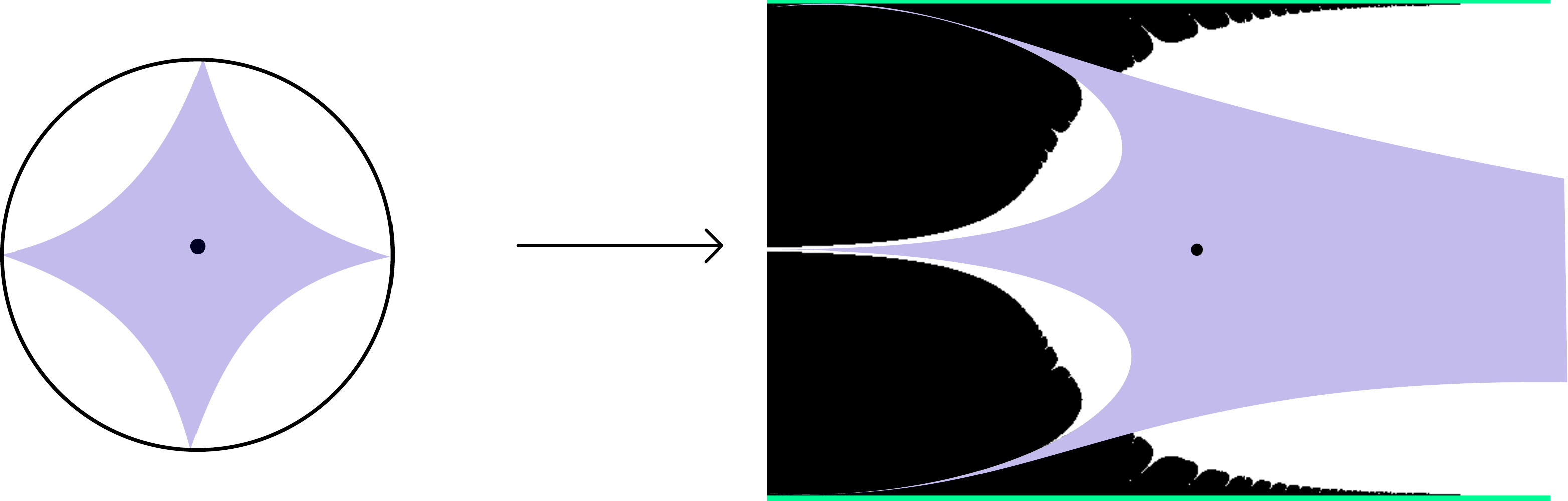}
	\setlength{\unitlength}{15cm}
	\put(-0.05, 0.29){$U$}
	\put(-0.2, 0.175){$V$}
	\put(-0.61, 0.175){$\varphi$}
	\put(-0.77, 0.25){$\mathbb{D}$}
	\caption{\footnotesize Schematic representation of a Fatou component satisfying the previous condition of the postsingular set. The Fatou component of the right is a Baker domain of $ z+e^{-z} $. For this particular example, the domain $ V $ could have been taken simpler. However, we wanted to illustrate how $ V $ looks like in general. It follows that, for the associated inner function, there exist crosscut neighbourhoods without postsingular values, and hence hypothesis {\em \ref{hypothesisb}} is satisfied.
	}\label{fig-starlike}
\end{figure}

Let us check that such a Fatou component satisfies the  hypothesis of Theorem \ref{thm-periodic-points-are-dense}. It is clear that 
this already implies that for $ \omega_U $-almost every $ x\in\partial U $, there exists $ r\coloneqq r(x)>0 $ such that $ D(x,r)\cap P(f)=\emptyset $, and hence {\em\ref{hypothesisa}} is satisfied. To see {\em\ref{hypothesisb}}, let $ \varphi\colon\mathbb{D}\to U $ be a Riemann map, and note that $ \varphi^{-1}(\overline{V}\cap U) $ is a closed set in $ \mathbb{D} $. By hypothesis, $ \lambda(\overline  {\varphi^{-1}({V}\cap U)} )=0 $, implying {\em\ref{hypothesisb}}.
Compare with the Fatou components considered in \cite{JoveFagella2}.

In the case of a doubly parabolic Baker domain, we need to ask additionally that $ f|_{\partial U} $ is recurrent with respect to $ \omega_U $. This can be ensured with any of the conditions given in Theorem \ref{thm-ergodic-boundary-map}{\em \ref{ergodic-properties-d}}.

\subsection{Proof of Theorem \ref{thm-periodic-points-are-dense}}

		We shall split the proof into several steps.
	
		\begin{enumerate}
		\item {\em Inverse branches well-defined $ \omega_U $-almost everywhere.}
		
		First note that hypothesis {\em \ref{hypothesisa}} implies that inverse branches of $ f^n $ interplaying with $ \partial U $ are well-defined in $ D(x_0,r_0) $, being $ r_0 $ uniform for all $ n\geq 0 $ and all inverse branches. Let us start by proving that this actually holds for $ \omega_U $-almost every $ x\in \partial U $, as an easy consequence of {\em \ref{hypothesisa}} together with the ergodic properties of $ f|_{\partial U} $.
		
		\begin{claim*}
For $ \omega_U $-almost every $ x\in\partial U $, there exists $ r\coloneqq r(x)>0 $ such that, for all $ n\geq 0 $, $ f^{-n}(D(x, r)) \cap U $ is non-empty, and,
if $ D_n $ is a connected component of $ f^{-n}(D(x, r)) $ such that $ D_n\cap U\neq\emptyset $, then $ f^n|_{D_n} $ is conformal.
		\end{claim*}
	 We shall refer to the inverse branches considered above as \textit{relevant inverse branches} of $ f^n $ at $ x\in\partial U $.  When we refer to a particular inverse branch, we write $ F_{n,y,x} $ meaning that $ F_n $ is an inverse branch   of $ f^{n} $ sending $ y$ to $ x $. When the points $ x $ and $ y $ are clear from the context, we shall write only $ F_n $ to lighten the notation. Since we are interested in the study of $ f|_{\partial U} $,  relevant inverse branches are the only ones that play a role in our construction. 
	\begin{proof}[Proof of the claim]
		The first statement of the claim is deduced from 
		the conjugacy between $ f|_U $ and the inner function $ g $, and the fact that inner functions associated to Fatou components of functions in class  omit at most two values (see e.g. 	\cite[Thm. 1]{Bolsch-Fatoucomponents}).
		
		For the second assertion, 
	let $ x_0 $ and $ r_0 $ be the ones given by hypothesis {\em \ref{hypothesisa}}.	Since $ \omega_U(D(x_0, r_0))>0 $, by hypothesis {\em \ref{hypothesisc}} and Theorem \ref{thm-ergodic-boundary-map}, the orbit of $ \omega_U $-almost every point visits infinitely many times $ D(x_0, r{(x_0)}) $. Hence, $ \omega_U $-almost every $ x\in\widehat{\partial} U$ there exists $ n_0\coloneqq n_0(x) $ such that $ f^{n_0}(x)\in D(x_0,r{(x_0)}) $.
		
		Fix $ x\in\widehat{\partial} U $ with this property.  Then, there exists $ r\coloneqq r(x) >0$ such that $ f^{n_0}(D(x,r))\subset D(x_0, r_0) $. 
		By hypothesis {\em \ref{hypothesisa}}, $ f^{n_0}|_{D(x,r) }$ is conformal.
		
Let $ D_{m} $ be a connected component of $ f^{-m}(D(x, r)) $ such that $ D_{m} \cap U\neq\emptyset $.  Then, $ D_{m} $ is contained in a connected component of $ f^{-m-n_0}(D(x_0, r_0)) $, so there exists $ x_m\in D_m$  and a relevant inverse branch of $ f^{m+n_0} $ 
\[F_{m+n_0, x_0, x_m}\colon f^{n_0} (D(x,r))\subset D(x_0, r_0)  \longrightarrow D_m \]
which is well-defined, and hence conformal. Then, \[F_{m, x, x_m}=F_{m+n_0, x_m, x_0}\circ f^{m}\colon D(x,r) \longrightarrow D_m \] is a well-defined inverse branch of $ f^m $. In particular, $ f^m|_{D_m} $ is conformal, proving the claim.
	\end{proof}

		\item {\em Construction of an expansive metric around $ \partial U $.}
		
		\begin{lemma*}
			There exists a hyperbolic open  set $ W\subset \mathbb{C} $ and a measurable set $ X\subset W$ for which the following are satisfied. \begin{enumerate}[label={\em (2.\arabic*)}]
				\item\label{IL1-a} The set  $ X $ is contained in $ {\partial} U $, and it has full $ \omega_U$-measure.
				\item\label{IL1-b} For all  $ x\in X $ and  $ n\geq 0 $, there exists $ r_W\coloneqq r_W(x) >0$ such that the hyperbolic disk $ D_W(x,r_W) $ is simply connected and compactly contained in $ W $,  with all relevant branches $ F_n $ of $ f^{-n} $ at $ x $ well-defined in $ D_W(x,r_W) $. Moreover, $ F_n(D_W(x,r_W))\subset W $.
				\item\label{IL1-c} Relevant inverse branches do not increase the hyperbolic distance $\textrm{\em dist} _W $ between points, i.e. for any $ x\in W $ and $ F_1 $ relevant branch of $ f^{-1} $ at $ x $, if $ z,w\in  D_W(x,r_W)$, \[\textrm{\em dist}_W(F_1(z), F_1(w))\leq  \textrm{\em dist}_W(z, w).\]
			\end{enumerate}
		\end{lemma*}
		\begin{proof}
		For all $ x\in\partial U $, let $ r_n(x)\in \left[ 0, +\infty\right)  $ be the radius of the maximal Euclidean disk $ D(x, r_n(x)) $ for which all relevant branches of $ f^{-n} $ at $ x $ are well-defined. We assume at least one such inverse branch exists, otherwise set $ r_n(x) =0$ (by the previous claim, this situation only happens on a set of zero $ \omega_U $-measure). Clearly, $ r_n(x)\geq r_{n+1}(x) $, so $ \left\lbrace r_n(x)\right\rbrace _n $ is a convergent sequence for all $ x\in\partial U $. Consider \[X\coloneqq \left\lbrace x\in\partial U\colon r_n(x)\to r(x)>0 \right\rbrace. \]
		By the claim in the first step, $ \omega_U(X) =1$. Let \[ W\coloneqq \bigcup\limits_{x\in X}\bigcup\limits_{n\geq 0} \left\lbrace F_n(D(x, r(x))) \colon F_n\textrm{ is relevant at } x\right\rbrace .\]
			Note that $ W $ is open, and $ X\subset W $. Hence, {\em \ref{IL1-a}} holds.
			
			Taking $ r(x) >0$ smaller if needed, we can assume $ W $ omits at least two points, so it is hyperbolic and admits a hyperbolic metric on it. 
			
			Note that $ W $ may be disconnected. In this case, the hyperbolic density is defined on each connected component separatedly. Indeed, each connected component $ W_1 $ of $ W $ is a hyperbolic domain, and hence admits a hyperbolic density $ \rho_{W_1} $. Given $ z\in W $, we define \[\rho_W(z)\coloneqq \rho_{W_1}(z),\] where $ W_1 $ stands for the connected component of $ W $ with $ z\in W_1 $. Given $ z,w\in W $, the hyperbolic distance is defined as $ \textrm{dist}_W(z, w)= \textrm{dist}_{W_1}(z, w) $, if $ z $ and $ w $ lie in the same connected component $ W_1 $ of $ W $; and $ \textrm{dist}_W (z,w)=\infty $, otherwise.
			
			By construction, it holds that, for every $ x\in X $, all relevant branches of $ f^{-n} $ are well-defined in $ D(x,r(x))\subset W $. Since the Euclidean and the hyperbolic metrics are locally equivalent, there exists $ r_W\coloneqq r_W(x) $ such that $ \overline{D_W(x,r_W)} \subset {D(x,r)}$. 
			Hence, $ D_W(x,r_W) $ is simply connected and compactly contained in $ W $, and all relevant inverse branches $ F_n $ are well-defined in $ D_W(x,r_W)$, and $ F_n(D_W(x,r_W))\subset F_n (D(x,r))\subset W $. Thus, {\em \ref{IL1-b}} holds.
			
			Finally, we claim that $ f $ does not increase the hyperbolic distance between points. 
			Consider \[ W'\coloneqq \bigcup\limits_{x\in X}\bigcup\limits_{n\geq 1} \left\lbrace F_n(D(x, r(x))) \colon F_n\textrm{ is relevant at } x\right\rbrace \subset W.\] Consider the hyperbolic density $ \rho_{W'} $ in $ W' $, defined component by component if needed. Note that each connected component $ W_1 $ of $ W' $ is mapped onto a connected component $ W_2 $ of $ W $ as a holomorphic covering. Hence, if $ x\in W_1\subset W'$ and $ f(x)\in W_2 \subset W$, then \[ \rho_{W'}(x)=\rho_{W_1}(x)=\rho_{W_2}(f(x))\cdot \left| f'(x) \right|=\rho_{W}(f(x))\cdot \left| f'(x) \right|. \]
			
			 Since each connected component of $ W' $ is contained in a connected component of $ W $,  we have $ \rho_W\leq\rho_{W'} $ and, in particular, if $ x\in W' $, 
			 \[ \rho_{W}(x)\leq \rho_{W}(f(x))\cdot \left| f'(x) \right|. \]
			
			Now, let $ x\in W $ and let  $ F_1 $ be a relevant branch of $ f^{-1} $ at $ x $.  Since $ F_1 $ is well-defined in $ D_W(x,r_W)$, it holds \[\rho_W(F_1(z))\left| F'_1(z)\right| \leq \rho_W(z),\] for all $ z\in  D_W(x,r_W)$. Next, take $ z,w\in  D_W(x,r_W)$. Note that $ z,w \in  W_1 $, and $ F_1(z), F_1(w) \in W_2$, for two connected components $ W_1 , W_2$ of $ W $.
			Moreover,  since hyperbolic disks are hyperbolically convex (i.e. a geodesic joining two points in the disk is contained in the disk), we can  take $ \gamma\subset D_W(x,r_W)$ geodesic between $ z $ and $ w $. Then,
			\[\textrm{dist}_W(F_1(z),F_1(w))= \textrm{dist}_{W_2}(F_1(z),F_1(w))\leq \int_{F_1(\gamma)}\rho_W(t)dt=\]\[=\int_\gamma \rho_W(F_1(t))\left| F_1'(t) \right| dt\leq \int_\gamma \rho_W(t)dt= \textrm{dist}_{W_1}(z,w)= \textrm{dist}_{W}(z,w). \]
			proving the claim.
		\end{proof}

		\item{\em Control of radial limits in terms of Stolz angles.}

		Let us fix $ \alpha\in (0, \frac{\pi}{2}) $, and let $ p\in\overline{\mathbb{D}} $ be the Denjoy-Wolff point of the associated inner function $ g $.
		It follows from Theorem \ref{thm-postsingular} that, $ \lambda $-almost every $ \xi\in\partial\mathbb{D} $, there exists $ \rho\coloneqq\rho(\xi)>0 $ such that:\begin{enumerate}[label={ (3.\arabic*)}]
			\item \label{IL2a} for all $ n\geq 0 $, every branch $ G_n $ of $ g^{-n}$ is well-defined in $ D(\xi, \rho) $;
			\item\label{IL2b} there exists $ \rho_1 \coloneqq\rho_1 (\xi)$ such that, for all $ n\geq 0 $,  \[  G_n(R_{\rho_1}, p)\subset \Delta_{\alpha, \rho_1}(G_n(\xi), p),\]
			 where $ R_\rho(\cdot, p) $ and $ \Delta_{\alpha,\rho}(\cdot, p) $ stand for the radial segment and the Stolz angle with respect to $ p $ {(Def. \ref{defi-radi-stolz-angle})}. 
		\end{enumerate}
		In the sequel, to lighten the notation we denote the radial segments and the Stolz angles just by  $ R_{\rho} $ and $ \Delta_{\rho} $, respectively. However, one should keep in mind that they are radial segments and Stolz angles with respect to the Denjoy-Wolff point, and that the opening $ \alpha $ of the Stolz angles is fixed throughtout the proof.
	
		\item {\em  Choice of a set $  K_\varepsilon \subset \partial\mathbb{D}$, where bounds on $ \varphi $ and $ G_n $ are uniform.}
		\begin{lemma*}\label{intermediate lemma}
					Fix $ \varepsilon>0 $. There exists a measurable set $ K_\varepsilon \subset\partial\mathbb{D} $ with $ \lambda(K_\varepsilon) \geq 1-\varepsilon$ such that the following holds.
			\begin{enumerate}[label={\em (4.\arabic*)}]
			 \item \label{IL3a} For all $ \xi\in K_\varepsilon $, $ \varphi^*(\xi) $ exists and $ \varphi^*(\xi) \in X$. Moreover, $ (g^n)^*(\xi) \in \Theta_\Omega$, for all $ n\geq0 $.
				\item \label{IL3b} There exists $ r_\varepsilon>0$ such that for all $ \xi\in K_\varepsilon $ and $ n\geq 0 $,  all relevant inverse branches of $ f^n $ are well-defined in $ D_W(\varphi^*(\xi), r_\varepsilon) $. 
				\item \label{IL3c}  There exist $ \rho_\varepsilon>0 $ such that\begin{enumerate}[label={\em \roman*.}]
					\item For every $ \xi\in K_\varepsilon $ and $ n\geq 0 $, every branch $ G_n $ of $ g^{-n} $ is well-defined in $ D(\xi, \rho_\epsilon) $.
					\item For every $ \xi\in K_\varepsilon $, \[  G_n(R_{\rho_\varepsilon}(\xi))\subset \Delta_{ \rho_\varepsilon}(G_n(\xi)).\] 
						\item For every $ \xi\in K_\varepsilon $, if $ z\in\Delta_{ \rho_\varepsilon} (\xi)$,  then $ \varphi(z)\in W $ and \[	\textrm{\em dist}_W(\varphi(z), \varphi^*(\xi))<\frac{r_\varepsilon}{3} .\]
				\end{enumerate}
				\item\label{IL3 (d)} There are no isolated points in $ K_\varepsilon  $. In fact, for every $ \xi\in K_\varepsilon $, there exists a subsequence $ \left\lbrace n_k\coloneqq n_k(\xi)\right\rbrace _k $, $ n_k \to\infty$,  such that $ (g^{n_k})^*(\xi)\in K_\varepsilon $ and $ (g^{n_k})^*(\xi) \to \xi$.
				
				\item\label{IL3 (e)} For every $ \xi\in K_\varepsilon $, the orbit of $ \varphi^*(\xi) $ under $ f $ is dense in $ \partial U $.
			\end{enumerate}
		\end{lemma*}
	
	\begin{proof}
	Consider $ K\coloneqq (\varphi^*)^{-1}(X)\subset \partial \mathbb{D}$. Observe that $ \lambda(K)=1 $.
	There is no loss of generality on assuming that $ (g^n)^*(\xi) \in \Theta_\Omega$, for all $ n\geq0 $, since this holds $ \lambda $-almost everywhere. Hence, all points in $ K $ satisfy {\em \ref{IL3a}}.
	
	Next, by {\em \ref{IL1-b}}, for all $ \xi \in K $, there exists $ r\coloneqq r(\xi)>0 $, such that all relevant inverse branches $ F_n $ are well-defined in $ D_W(\varphi^*(\xi), r(\xi)) $. Hence, we can write $ K $ as the countable union of the nested measurable sets \[ K_m\coloneqq \left\lbrace \xi\in K\colon\textrm{ all relevant }  F_n\textrm{ are  well-defined in }D_W(\varphi^*(\xi), 1/m)\right\rbrace\subset K_{m+1}. \] Choose $ m_0 $ such that $ \lambda(K_{m_0}) \geq1-\varepsilon/3$, which satisfies {\em \ref{IL3b}} with $ r_\varepsilon=1/m_0 $. 
	
	Now, by \ref{IL2a} and \ref{IL2b}, we can assume that, for all $ \xi\in K_{m_0} $,  there exists $ \rho \coloneqq\rho (\xi)>0$ such that $ G_n $ is well-defined in $ D(\xi, \rho) $ and  $G_n(R_{\rho})\subset \Delta_{ \rho}(G_n(\xi))$.
	Hence, $ K_{m_0} $ can be written as the countable union of the nested measurable sets\[K_{m_0}^k\coloneqq\left\lbrace \xi\in K_{m_0}\colon  G_n \textrm{ well-defined in }D(\xi, 1/k) \textrm{ and } G_n(R_{1/k})\subset \Delta_{ \frac{1}{k}}(G_n(\xi))\right\rbrace.\]
	Choose $ k_0 $ such that $ \lambda(K^{k_0}_{m_0}) \geq1-\varepsilon/2$. Finally, note that  the angular limit of $ \varphi $ exists at every $ \xi\in K^{k_0}_{m_0} $, that is, for all $ \xi\in K^{k_0}_{m_0} $ there exists $ \rho_1(\xi)<\rho(\xi) $ such that, for all $ z\in \Delta_{\rho_1} $, \[\textrm{dist}_W(\varphi(z), \varphi^*(\xi))<\frac{r_\varepsilon}{3} .\]Hence, proceeding as before, we find $ K_\varepsilon\subset  K^{k_0}_{m_0}$, with $ \lambda(K_\varepsilon)\geq 1-\varepsilon $ and satisfying {\em \ref{IL3c}} for some $ \rho_\varepsilon >0 $, uniform for $ \xi\in K_\varepsilon $.
	
	Since $ \lambda $-almost every point in $ K_\varepsilon $ is a Lebesgue density point (Def. \ref{def-lebesgue-density}), there is no loss of generality on assuming that every $\xi\in K_\varepsilon $ is a Lebesgue density point. In particular, there are no isolated points in $ K_\varepsilon $. 
	
	Moreover, since $ g^*|_{\partial\mathbb{D}} $ is recurrent, every measurable set $ E\subset X $ with $ \lambda(E)>0 $, we have that for $ \lambda $-almost every point $ \xi\in K_\varepsilon $, there exists a subsequence $ \left\lbrace n_k\right\rbrace _k $, $ n_k\to\infty $, with $ (g^{n_k})^*(\xi)\in E $.
	
	Now, take a countable sequence $ \left\lbrace \xi_n\right\rbrace _n\subset K_\varepsilon $, such that $ \left\lbrace \xi_n\right\rbrace _n $ is dense in $ K_\varepsilon  $ and each $ \xi_n $ is a Lebesgue density point for $ K_\varepsilon $. Consider $ E_{j, n}\coloneqq D(\xi_n, 1/j)\cap K_\varepsilon $, for $ j, n\geq 1 $. Since each $ \xi_n$ is a Lebesgue density point for $ K_\varepsilon $, $ \lambda (E_{j,n}) >0$, for all $ j, n\geq 1 $. 
	
	Applying the previous property to the sequence $ \left\lbrace E_{j, n}\right\rbrace _{j,n }$, we have that, for $ \lambda $-almost every $ \xi\in K_\varepsilon  $, there exists a subsequence $ \left\lbrace n_{k}\right\rbrace _k $, $ n_k\to\infty $, with $ (g^{n_k})^*(\xi)\in E_{k,k}$. Hence, there exists a subsequence $ \left\lbrace n_k\right\rbrace _k $, $ n_k\to\infty $, with $ (g^{n_k})^*(\xi)\in  K_\varepsilon$, $ (g^{n_k})^*(\xi)\to \xi$, proving {\em \ref{IL3 (d)}}.
	
	Finally, since points in $ \partial U $ with dense orbit have full harmonic measure, we can assume that $ K_\varepsilon $ is chosen so that {\em \ref{IL3 (e)}} holds.
	\end{proof}

		\item {\em Construction of a periodic point in $ D_W(\varphi^*(\xi), r) $, for all $ \xi\in K_\varepsilon $ and $ r\in (0,r_\varepsilon )$.}
		
		Set $\xi\in K_\varepsilon $ and $ r\in (0,r_\varepsilon )$. The goal in this section is to find a periodic point in  $ D_W(\varphi^*(\xi), r) \cap\partial U$.
		
		Write  $ \xi_n \coloneqq (g^n)^*(\xi)$, for all $ n\geq 0 $. By {\em\ref{IL3a}}, $ \varphi^*(\xi_n) $ exists and belongs to $ \Omega $, for all $ n\geq0 $. By {\em \ref{IL3 (d)}} and  {\em \ref{IL3b}}, $ \xi_n\in K_\varepsilon $ infinitely often, and all relevant inverse branches are well-defined in $ D_W(\varphi^*(\xi_n), r_\varepsilon) $. In particular, for all $ 0\leq m<n $, there is a relevant inverse branch $ F_{n-m} $ of $ f^{n-m} $ in $ D(\varphi^*(\xi_n), r_\varepsilon) $. 
		Hence, for all $ n\geq 0 $, $ f $ maps conformally a neighbourhood of $ \varphi^*(\xi_n) $ onto a neighbourhood of $ \varphi^*(\xi_{n+1}) $. This implies that, although we are considering $ f\in\mathbb{K} $, in practice, everywhere where we consider $ f $, it is holomorphic (and, in fact, conformal).

		Now, consider $ D_W(\varphi^*(\xi_0), r_\varepsilon)  $, and let $ W_1 $ be the connected component of $ W $ such that $\varphi^*(\xi_0)\in W_1  $. Note that, by {\em \ref{IL3b}}, all relevant branches of $ f^{-n} $ at $ \varphi^*(\xi_0) $ are well-defined in $ D_W(\varphi^*(\xi_0), r_\varepsilon)  $. In particular, there exists $ n_0\geq 1 $ and  a relevant inverse branch of $ f^{n_0} $, say $ F_{n_0} $, such that $ F_{n_0}(\varphi^*(\xi_0))\in W_1 $ (recall that preimages of any point are dense in the Julia set, with at most two exceptions). Consider $ D_{n_0}\coloneqq F_{n_0}(D_W(\varphi^*(\xi_0), r_\varepsilon)) $. Therefore, \[F_{n_0}\colon D_W(\varphi^*(\xi_0), r_\varepsilon)\longrightarrow D_{n_0}\] conformally.
		\begin{claim*}
			There exists $ k \in (0,1)$ such that, for all $ z,w\in  D_W(\varphi^*(\xi_0), r_\varepsilon)$, \[\textrm{\em dist}_W(F_{n_0}(z), F_{n_0}(w))\leq k\cdot \textrm{\em dist}_W(z, w).\] 
		\end{claim*}
		\begin{proof}
				Let $ W_1 $ be the connected component of $ W $ in which $ \varphi^*(\xi_0) $ lies. Consider, as in the proof of {\em \ref{IL1-c}},\[ W'\coloneqq \bigcup\limits_{x\in X\cap W_1}\bigcup\limits_{n\geq 1} \left\lbrace F_n(D(x, r(x))) \colon F_n\textrm{ is relevant at } x \right\rbrace \subset W.\] 
			
			Let $ W'_1 $ be the connected component of $ W' $ that contains $ F_{n_0}(\varphi^*(\xi_0)) $. Then, $ W'_1\subset W_1 $, and \[f^{n_0}\colon W'_1\to W_1\] is a holomorphic covering. Then,  for $ x\in W'_1\subset W_1$, it holds \[ \rho_{W_1'}(x)|=\rho_{W_1}(f^{n_0}(x))\cdot \left| (f^{n_0})'(x) \right| .\]
			Note that the inclusion $ W'_1\subset W_1 $ is strict (otherwise $ f^{n_0}(W_1)=W_1 $, and this is impossible since $ W_1 $ contains points of $ \mathcal{J}(f) $). Hence, $ \rho_{W_1}< \rho_{W'_1}$, so  \[ \rho_{W_1}(x)|<\rho_{W_1}(f^{n_0}(x))\cdot \left| (f^{n_0})'(x) \right| .\]
				
	Moreover,	since $ D_ W(\varphi^*(\xi_0), r_W)$ is compactly contained in $ W $, there exists $ k \in (0,1)$ such that, for all $ x\in D_ W(\varphi^*(\xi_0), r_W)$,  \[ \rho_{W_1}(F_{n_0}(x))\cdot \left| F_{n_0}'(x) \right|\leq k\cdot\rho_{W_1}(x). \]Finally, for  $ z,w\in D_W(\varphi^*(\xi_0),r_W) $, take $ \gamma\subset D_W(\varphi^*(\xi_0),r_W)$ geodesic between $ z $ and $ w $, and
	\[\textrm{dist}_W(F_{n_0}(z),F_{n_0}(w))= \textrm{dist}_{W_1}(F_{n_0}(z),F_{n_0}(w))\leq \int_{F_{n_0}(\gamma)}\rho_W(t)dt=\]\[=\int_\gamma \rho_W(F_{n_0}(t))\left| F_{n_0}'(t) \right| dt\leq k\int_\gamma \rho_W(t)dt= k\textrm{ dist}_{W_1}(z,w)= k\textrm{ dist}_{W}(z,w), \]
			proving the claim.
		\end{proof}
		Now, we claim that we can find $ N\geq 1 $ satisfying the following properties. \begin{enumerate}[label={ (5.\arabic*)}]
			\item\label{IL4b}  If $ N_0\coloneqq \#\left\lbrace n\leq N\colon \varphi^*(\xi_n)\in D_{n_0} \right\rbrace  $, then $ k^{N_0}<\dfrac{r}{3r_\varepsilon} $.
			\item\label{IL4a}  $ \xi_N\coloneqq (g^N)^*(\xi_0)\in K_\varepsilon $
			\item\label{IL4c} There exists $ t_N\in(0,1) $ such that $t_N\xi_N\in R_{\rho_\varepsilon}(\xi_N) \cap \Delta_{	\rho_\varepsilon}(\xi_0)$.
		\end{enumerate} 
	
	Indeed, by {\em \ref{IL3 (e)}}, the orbit of $ \varphi^*(\xi_0) $ is dense in $ \partial U $. In particular, it visits $ D_{n_0} $ infinitely many times. Hence, there exists $ N' $ so that 	{\ref{IL4b}} is satisfied for $ N' $. By {\em \ref{IL3 (d)}}, there exists a subsequence $ \left\lbrace n_k\right\rbrace _k $, $ n_k\to\infty $, such that $ \xi_{n_k} \in K_\varepsilon$ and $  \xi_{n_k}\to \xi_0$. Thus, we can find $ N\geq N' $ for which conditions  \ref{IL4a} and \ref{IL4c} are also satisfied.
	Note that the geometric condition in \ref{IL4c} is satisfied as long $ \xi_N $ is close enough to $ \xi_0 $, since the radius $ \rho_\varepsilon $ and the angle $  \alpha$ are uniform (see Fig. \ref{fig-tnxin} for a geometric intuition). 
			\begin{figure}[htb!]\centering
		\includegraphics[width=8cm]{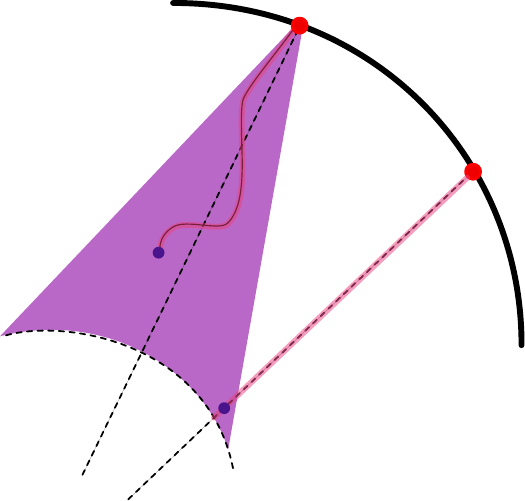}
			\setlength{\unitlength}{8cm}
		\put(-0.07, 0.62){$\xi_N$}
			\put(-0.55, 0.15){$t_N\xi_N$}
				\put(-0.82, 0.42){$G_N(t_N\xi_N)$}
						\put(-0.4, 0.3){$R(\xi_N)$}
							\put(-0.84, 0.65){$\Delta_{\rho_\varepsilon}(\xi_0)$}
			\put(-0.04, 0.25){$\partial\mathbb{D}$}
				\put(-0.42, 0.93){$\xi_0$}
		\caption{\footnotesize The choice of the point $t_N\xi_N\in R_{\rho_\varepsilon}(\xi_N) \cap \Delta_{	 \rho_\varepsilon}(\xi_0)$.}\label{fig-tnxin}
	\end{figure}
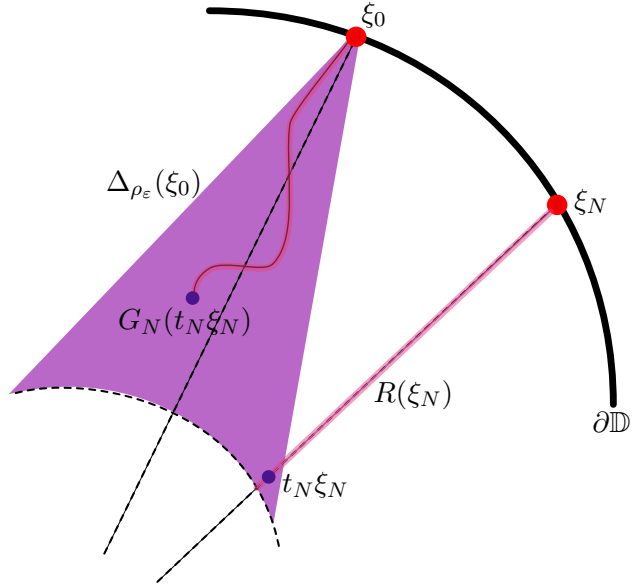

\begin{claim*}
	There exists a relevant inverse branch $ F_N $ of $ f^N $ at $ \varphi^*(\xi_N) $ defined in $ D_W(\varphi^*(\xi_N), r_\varepsilon) $, which satisfies $ F_N(\varphi^*(\xi_N) )=\varphi^*(\xi_0) $
	and \[ F_N( D_W(\varphi^*(\xi_N), r_\varepsilon))\subset  D_W\left( \varphi^*(\xi_0), \frac{r}{3}\right) \subset  D_W(\varphi^*(\xi_N), r_\varepsilon).\]
\end{claim*}
We note that, in particular, $ F_N( D_W(\varphi^*(\xi_N), r_\varepsilon))\subset  D_W\left( \varphi^*(\xi_0), r\right)  $.
		\begin{proof}
			First note that $ \xi_N\in K_\varepsilon $ \ref{IL4a}, so all relevant inverse branches are well-defined in $ D_W(\varphi^*(\xi_N), r_\varepsilon) $ {\em \ref{IL3b}}. Since $ \xi_N= (g^N)^*(\xi_0)$, by Lemma \ref{lemma-radial-limit}, we have $ f^N(\varphi^*(\xi_0))= \varphi^*(\xi_N) $. Hence, there exists a relevant inverse branch $ F_N $ of $ f^N $ at $ \varphi^*(\xi_N) $ defined in $ D_W(\varphi^*(\xi_N), r_\varepsilon) $, which satisfies $ F_N(\varphi^*(\xi_N) )=\varphi^*(\xi_0) $.
			
			Note that $ F_N $ is the composition of different inverse branches of $ f $, and each of them does not increase the hyperbolic distance $ \textrm{dist}_W $ between points {\em \ref{IL1-c}}. Moreover, $ \left\lbrace f^n(\xi_0) \right\rbrace_{n=0}^N  $ visits $ D_{n_0} $ at least $ N_0 $ times \ref{IL4b}. This means that applying $ F_N $ corresponds to applying the inverse $ F_{n_0} $, which acts as a contraction by $ k$,  at least $ N_0 $ times. Thus, we have \[F_N(D_W(\varphi^*(\xi_N), r_\varepsilon))\subset D_W\left( \varphi^*(\xi_0),k^{N_0} r_\varepsilon\right) \subset D_W\left( \varphi^*(\xi_0), \frac{r}{3}\right).\]
			
			To see the remaining inclusion, note that, by  {\em \ref{IL3c}} and \ref{IL4b}, we have \[\varphi(t_N\xi_N)\in D_W\left( \varphi^*(\xi_0), \frac{r_\varepsilon}{3}\right) \cap D_W\left( \varphi^*(\xi_N), \frac{r_\varepsilon}{3}\right).\] Hence, applying the triangle inequality,
			\[\textrm{dist}_W(\varphi^*(\xi_0), \varphi^*(\xi_N))\leq \textrm{dist}_W(\varphi^*(\xi_0), \varphi(t_N\xi_N))+\textrm{dist}_W(\varphi(t_N\xi_N), \varphi^*(\xi_N))<\frac{2r_\varepsilon}{3},\] implying the desired inclusion. 
	\end{proof}
	
	\begin{claim*}
		The map \[F_N\colon D_W(\varphi^*(\xi_N), r_\varepsilon)\longrightarrow D_W(\varphi^*(\xi_N), r_\varepsilon)\] has an attracting fixed point in $ D_W( \varphi^*(\xi_0), r)  $, which is accessible from $ U $. Hence, $ f $ has a repelling $ N $-periodic point in $ D_W\left( \varphi^*(\xi_0), r\right) \cap \partial U $.
	\end{claim*}
	\begin{proof}
  Since $ F_N(D_W(\varphi^*(\xi_N), r_\varepsilon))\subset  D_W\left( \varphi^*(\xi_0), r\right) $, by the Denjoy-Wolff Theorem, $ F_n $ as a fixed point $p\in D_W(\varphi^*(\xi_0), r)$, which attracts all points in $ D_W(\varphi^*(\xi_N),r_\varepsilon)$ under the iteration of $ F_n $. Hence it is repelling under $ f^n $ and thus belongs to $ \mathcal{J}(f) $.
  
  It is left to see that $ p $ is accessible from $ U $. To do so, first note that, by {\em \ref{IL3c}}, there exists a branch $ G_N $ of $ g^{-N} $ such that $ G_N $ is well-defined in $ D(\xi_N, \rho_\varepsilon ) $ and $ G_N(\xi_N)=\xi_0 $. It holds that $ \varphi\circ G_N= F_N\circ \varphi $ in $ \Delta_{ \rho_\varepsilon}(\xi_N) $. Moreover, $ G_N(R_{\rho_\varepsilon}(\xi_N))\subset  \Delta_{	\rho_\varepsilon}(\xi_0) $. In particular, $ G_N(t_N\xi_N)\in  \Delta_{ \rho_\varepsilon}(\xi_0)$. 
  
  Since $ t_N\xi_N\in  \Delta_{	 \rho_\varepsilon}(\xi_0)$, we can find a curve $ \gamma\subset  \Delta_{	 \rho_\varepsilon}(\xi_0)$ joining $ t_N\xi_N $ and $ G_N(t_N\xi_N) $. Then, $ \varphi(\gamma)\subset D_W(\varphi^*(\xi_N), r_\varepsilon) $ joins $ \varphi(t_N\xi_N) $ with $ F_N(\varphi(t_N\xi_N)) $. Define \[\Gamma\coloneqq \bigcup\limits_{m\geq 0}F^m_N(\gamma).\]Then, $ \Gamma \subset\partial U$ lands at $ p $, proving the claim.
	\end{proof}
  	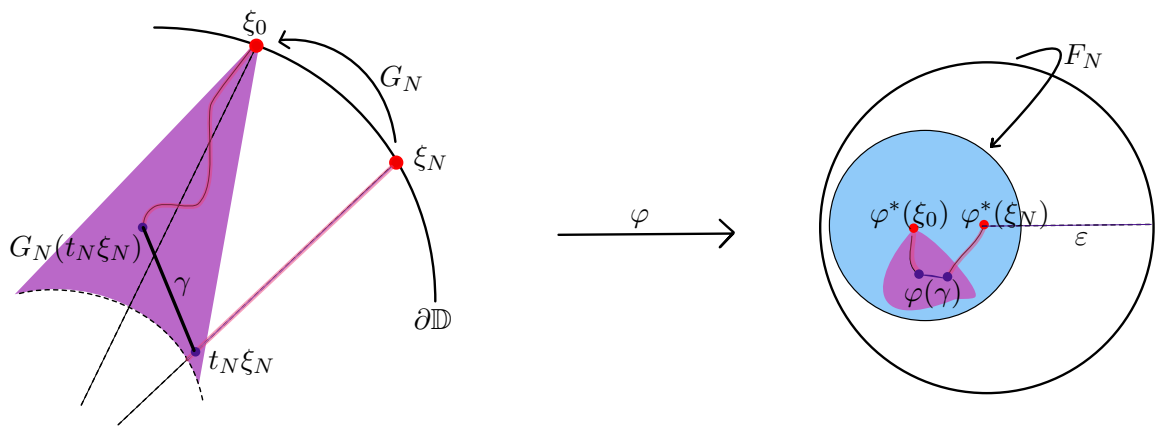
\begin{figure}[htb!]\centering
	\includegraphics[width=15cm]{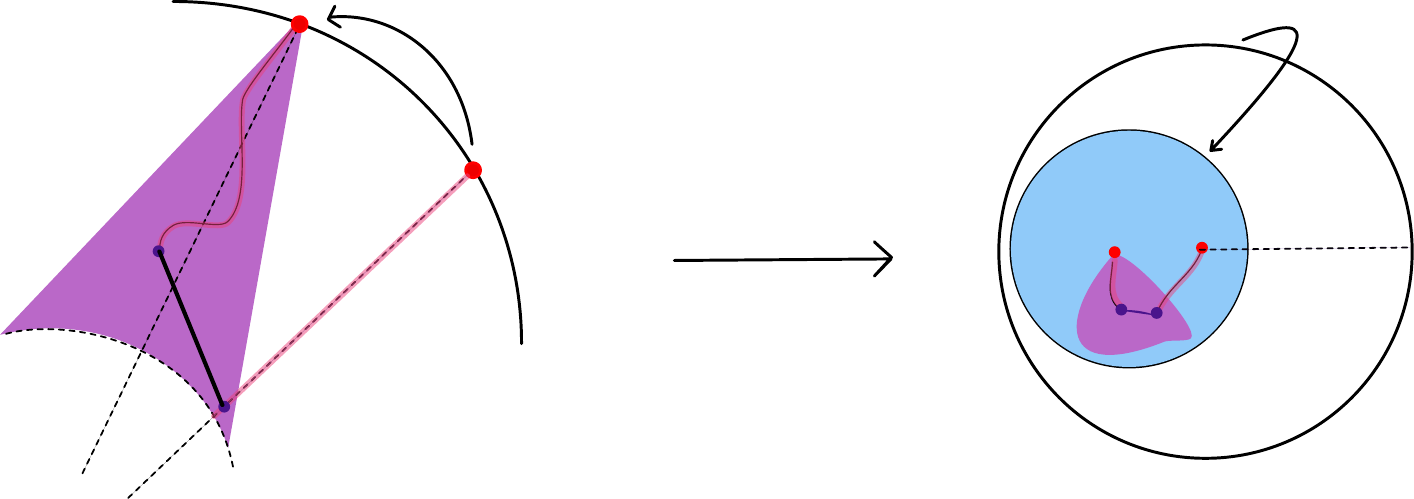}
		\setlength{\unitlength}{15cm}
	\put(-0.65, 0.23){$\xi_N$}
	\put(-0.83, 0.05){$t_N\xi_N$}
	\put(-1.005, 0.15){$G_N(t_N\xi_N)$}
		\put(-0.68, 0.3){$G_N$}
			\put(-0.08, 0.32){$F_N$}
	\put(-0.86, 0.12){$\gamma$}
	\put(-0.46, 0.18){$\varphi$}
	\put(-0.65, 0.085){$\partial\mathbb{D}$}
	\put(-0.8, 0.35){$\xi_0$}
		\put(-0.22, 0.11){$\varphi(\gamma)$}
			\put(-0.17, 0.18){$\varphi^*(\xi_N)$}
				\put(-0.25, 0.18){$\varphi^*(\xi_0)$}
					\put(-0.07, 0.16){$\varepsilon$}
	\caption{\footnotesize The construction of the curve $ \gamma $ in $ \mathbb{D} $, and its image $ \varphi(\gamma) $ in the dynamical plane.}\label{fig-puntsperiodics}
\end{figure}
		\item {\em Periodic points are dense in $ \partial U $.}
		
		Finally, to see that the previous construction leads to density of periodic points in $ \partial U $, one should take into account that $ \textrm{supp }\omega_U=\widehat{\partial} U $ (Lemma \ref{lemma_support_harmonic_measure}). Hence, for all $ x\in\partial U $ and $ \delta>0 $, it holds $ \omega_U (D(x,\delta))>0 $. Take $ \varepsilon\coloneqq  \omega_U (D(x,\delta))/2$, and consider $ K_\varepsilon $ as before. Note that, by the choice of $ \varepsilon $, we have $ \omega_U (D(x,\delta)\cap \varphi^*(K_\varepsilon))>0$.
		
		 Let $ \xi\in K_\varepsilon $ be such that $ \varphi^*(\xi)\in  D(x,\delta)$, and let  $ r\leq r_\varepsilon $.
		 In the previous step, we proved the existence of a periodic point in $ D_W(\varphi^*(\xi), r) $. 
		Taking $ r $ small enough, since $ \varphi^*(\xi)\in  D(x,\delta)$, we can ensure that the periodic point is in $ D(x,\delta) $. 
	\end{enumerate}
This ends the proof of Theorem \ref{thm-periodic-points-are-dense}.

\vspace{1cm}
{\bf Conflicts of interest.} The author states that there is no conflict of interest.

\vspace{0.2cm}
{\bf Data availability.} Data sharing not applicable to this article as no datasets were generated or analysed during the current study.
\printbibliography

@article {bieberbach,
	AUTHOR = {de Branges, L.},
	TITLE = {A proof of the {B}ieberbach conjecture},
	JOURNAL = {Acta Math.},
	FJOURNAL = {Acta Mathematica},
	VOLUME = {154},
	YEAR = {1985},
	NUMBER = {1-2},
	PAGES = {137--152},
}

@article {przytycki_zdunik,
	AUTHOR = {Przytycki, F. and Zdunik, A.},
	TITLE = {Density of periodic sources in the boundary of a basin of
	attraction for iteration of holomorphic maps: geometric coding
	trees technique},
	JOURNAL = {Fund. Math.},
	FJOURNAL = {Fundamenta Mathematicae},
	VOLUME = {145},
	YEAR = {1994},
	NUMBER = {1},
	PAGES = {65--77},
}

@book {milnor,
	AUTHOR = {Milnor, J.},
	TITLE = {Dynamics in one complex variable},
	SERIES = {Annals of Mathematics Studies},
	VOLUME = {160},
	EDITION = {Third},
	PUBLISHER = {Princeton University Press, Princeton, NJ},
	YEAR = {2006},
}

@book {Pommerenke,
	AUTHOR = {Pommerenke, Ch.},
	TITLE = {Boundary behaviour of conformal maps},
	SERIES = {Grundlehren der mathematischen Wissenschaften },
	VOLUME = {299},
	PUBLISHER = {Springer-Verlag, Berlin},
	YEAR = {1992},
}

@book {Hawkins,
	AUTHOR = {Hawkins, J.},
	TITLE = {Ergodic dynamics-from basic theory to applications},
	SERIES = {Graduate Texts in Mathematics},
	VOLUME = {289},
	PUBLISHER = {Springer},
	YEAR = {2021},
}

@book {PrzytyckiUrbanski,
	AUTHOR = {Przytycki, F. and Urba\'{n}ski, M.},
	TITLE = {Conformal fractals: ergodic theory methods},
	VOLUME = {371},
	PUBLISHER = {Cambridge University Press},
	YEAR = {2010},
}

@book {Rudin,
	AUTHOR = {Rudin, W.},
	TITLE = {Real and complex analysis},
	EDITION = {Third},
	PUBLISHER = {McGraw-Hill Book Co., New York},
	YEAR = {1987},
}

@book {garnett,
	AUTHOR = {Garnett, J. B.},
	TITLE = {Bounded analytic functions},
	SERIES = {Graduate Texts in Mathematics},
	VOLUME = {236},
	PUBLISHER = {Springer, New York},
	YEAR = {2007},
}

@article {wolff,
	AUTHOR = {Wolff, J.},
	TITLE = {Sur une généralisation d'un théorème de Schwarz},
	JOURNAL = {Comptes Rendus Acad. Sci.},
	VOLUME = {183},
	YEAR = {1926},
	NUMBER = {},
	PAGES = {918-920},
}

@article {cowen,
	AUTHOR = {Cowen, C. C.},
	TITLE = {Iteration and the solution of functional equations for
	functions analytic in the unit disk},
	JOURNAL = {Trans. Amer. Math. Soc.},
	FJOURNAL = {Transactions of the American Mathematical Society},
	VOLUME = {265},
	YEAR = {1981},
	NUMBER = {1},
	PAGES = {69--95},
}

@article {BakerDomínguez,
	AUTHOR = {Baker, I. N. and Dom\'{\i}nguez, P.},
	TITLE = {Boundaries of unbounded {F}atou components of entire
	functions},
	JOURNAL = {Ann. Acad. Sci. Fenn. Math.},
	FJOURNAL = {Annales Academi\ae  Scientiarum Fennic\ae . Mathematica},
	VOLUME = {24},
	YEAR = {1999},
	NUMBER = {2},
	PAGES = {437--464},
}

@article{DoeringMané1991,
	author = {Doering, C. and Mañé, R.},
	journal = {Ensaios Matemáticos (SBM)},
	pages = {1-79},
	title = {The Dynamics of Inner Functions},
	volume = {3},
	year = {1991},
}

@incollection {Bargmann,
	AUTHOR = {Bargmann, D.},
	TITLE = {Iteration of inner functions and boundaries of components of
	the {F}atou set},
	BOOKTITLE = {Transcendental dynamics and complex analysis},
	SERIES = {London Math. Soc. Lecture Note Ser.},
	VOLUME = {348},
	PAGES = {1--36},
	PUBLISHER = {Cambridge Univ. Press, Cambridge},
	YEAR = {2008},
}

@article {Bonfert,
	AUTHOR = {Bonfert, P.},
	TITLE = {On iteration in planar domains},
	JOURNAL = {Michigan Math. J.},
	FJOURNAL = {Michigan Mathematical Journal},
	VOLUME = {44},
	YEAR = {1997},
	NUMBER = {1},
	PAGES = {47--68},
}

@article {bfjk-escaping,
	AUTHOR = {Bara\'{n}ski, K. and Fagella, N. and Jarque, X. and
	Karpi\'{n}ska, B.},
	TITLE = {Escaping points in the boundaries of {B}aker domains},
	JOURNAL = {J. Anal. Math.},
	FJOURNAL = {Journal d'Analyse Math\'{e}matique},
	VOLUME = {137},
	YEAR = {2019},
	NUMBER = {2},
	PAGES = {679--706},
	shorthand={BFJK19}
}

@article {BakerDominguezHerring,
	AUTHOR = {Baker, I. N. and Dom\'{\i}nguez, P. and Herring, M. E.},
	TITLE = {Dynamics of functions meromorphic outside a small set},
	JOURNAL = {Ergodic Theory Dynam. Systems},
	FJOURNAL = {Ergodic Theory and Dynamical Systems},
	VOLUME = {21},
	YEAR = {2001},
	NUMBER = {3},
	PAGES = {647--672},
}

@article {Bolsch_repulsivepp,
	AUTHOR = {Bolsch, A.},
	TITLE = {Repulsive periodic points of meromorphic functions},
	JOURNAL = {Complex Variables Theory Appl.},
	FJOURNAL = {Complex Variables. Theory and Application. An International
	Journal},
	VOLUME = {31},
	YEAR = {1996},
	NUMBER = {1},
	PAGES = {75--79},
}

@article {bfjk_Newton,
	AUTHOR = {Bara\'{n}ski, K. and Fagella, N. and Jarque,
	X. and Karpi\'{n}ska, B.},
	TITLE = {Connectivity of {J}ulia sets of {N}ewton maps: a unified
	approach},
	JOURNAL = {Rev. Mat. Iberoam.},
	FJOURNAL = {Revista Matem\'{a}tica Iberoamericana},
	VOLUME = {34},
	YEAR = {2018},
	NUMBER = {3},
	PAGES = {1211--1228},
		shorthand={BFJK18}
}

@article {bfjk_connectivity,
	AUTHOR = {Bara\'{n}ski, K. and Fagella, N. and J.,
	Xavier and Karpi\'{n}ska, B.},
	TITLE = {On the connectivity of the {J}ulia sets of meromorphic
	functions},
	JOURNAL = {Invent. Math.},
	FJOURNAL = {Inventiones Mathematicae},
	VOLUME = {198},
	YEAR = {2014},
	NUMBER = {3},
	PAGES = {591--636},
		shorthand={BFJK14}
}

@book {Conway,
	AUTHOR = {Conway, J. B.},
	TITLE = {Functions of one complex variable. {II}},
	SERIES = {Graduate Texts in Mathematics},
	VOLUME = {159},
	PUBLISHER = {Springer-Verlag, New York},
	YEAR = {1995},
}

@article {Bolsch-Fatoucomponents,
	AUTHOR = {Bolsch, A.},
	TITLE = {Periodic {F}atou components of meromorphic functions},
	JOURNAL = {Bull. London Math. Soc.},
	FJOURNAL = {The Bulletin of the London Mathematical Society},
	VOLUME = {31},
	YEAR = {1999},
	NUMBER = {5},
	PAGES = {543--555},
}

@book {Aaronson97,
	AUTHOR = {Aaronson, J.},
	TITLE = {An introduction to infinite ergodic theory},
	SERIES = {Mathematical Surveys and Monographs},
	VOLUME = {50},
	PUBLISHER = {American Mathematical Society, Providence, RI},
	YEAR = {1997},
}

@misc{befrs-dw,
	doi = {10.48550/ARXIV.2203.06235},
	
	
	author = {Benini, A. M. and Evdoridou, V. and Fagella, N. and Rippon, P. J. and Stallard, G. M.},
	
	
	title = {The Denjoy--Wolff set for holomorphic sequences, non-autonomous dynamical systems and wandering domains},
	
	publisher = {arXiv},
	
	year = {2022},
	
	shorthand={BEFRS22},
	
}

@misc{JoveFagella2,
	title={Boundary dynamics in unbounded Fatou components}, 
	author={A. Jové and N. Fagella},
	year={2023},
	eprint={2307.11384},
	archivePrefix={arXiv},
	primaryClass={math.DS}
}

@book {harmonicmeasure2,
	AUTHOR = {Garnett, J. B. and Marshall, D. E.},
	TITLE = {Harmonic measure},
	SERIES = {New Mathematical Monographs},
	VOLUME = {2},
	PUBLISHER = {Cambridge University Press, Cambridge},
	YEAR = {2005},
	ISBN = {9780521470186},
}

@article {FatousAssociates,
	AUTHOR = {Evdoridou, V. and Rempe, L. and Sixsmith, D. J.},
	TITLE = {Fatou's associates},
	JOURNAL = {Arnold Math. J.},
	FJOURNAL = {Arnold Mathematical Journal},
	VOLUME = {6},
	YEAR = {2020},
	NUMBER = {3-4},
	PAGES = {459--493},
}

@article {efjs,
	AUTHOR = {Evdoridou, V. and Fagella, N. and Jarque, X.
	and Sixsmith, D. J.},
	TITLE = {Singularities of inner functions associated with hyperbolic
	maps},
	JOURNAL = {J. Math. Anal. Appl.},
	FJOURNAL = {Journal of Mathematical Analysis and Applications},
	VOLUME = {477},
	YEAR = {2019},
	NUMBER = {1},
	PAGES = {536--550},
			shorthand={EFJS19},
}

@article {bergweiler,
	AUTHOR = {Bergweiler, W.},
	TITLE = {Iteration of meromorphic functions},
	JOURNAL = {Bull. Amer. Math. Soc. (N.S.)},
	FJOURNAL = {American Mathematical Society. Bulletin. New Series},
	VOLUME = {29},
	YEAR = {1993},
	NUMBER = {2},
	PAGES = {151--188},
}

@article {przytycki_attracting,
	AUTHOR = {Przytycki, F.},
	TITLE = {Hausdorff dimension of harmonic measure on the boundary of an
	attractive basin for a holomorphic map},
	JOURNAL = {Invent. Math.},
	FJOURNAL = {Inventiones Mathematicae},
	VOLUME = {80},
	YEAR = {1985},
	NUMBER = {1},
	PAGES = {161--179},
}

@article {przytycki_attracting2,
	AUTHOR = {Przytycki, F.},
	TITLE = {Riemann map and holomorphic dynamics},
	JOURNAL = {Invent. Math.},
	FJOURNAL = {Inventiones Mathematicae},
	VOLUME = {85},
	YEAR = {1986},
	NUMBER = {3},
	PAGES = {439--455},
}

@article {BaranskiKarpinska_GeometricCodingTree,
	AUTHOR = {Bara\'{n}ski, K. and Karpi\'{n}ska, B.},
	TITLE = {Coding trees and boundaries of attracting basins for some
	entire maps},
	JOURNAL = {Nonlinearity},
	FJOURNAL = {Nonlinearity},
	VOLUME = {20},
	YEAR = {2007},
	NUMBER = {2},
	PAGES = {391--415},
}

@article {Baker_SimplyConnected,
	AUTHOR = {Baker, I. N.},
	TITLE = {Wandering domains in the iteration of entire functions},
	JOURNAL = {Proc. London Math. Soc. (3)},
	FJOURNAL = {Proceedings of the London Mathematical Society. Third Series},
	VOLUME = {49},
	YEAR = {1984},
	NUMBER = {3},
	PAGES = {563--576},
}

@article {Baker_FixedPoints,
	AUTHOR = {Baker, I. N.},
	TITLE = {Repulsive fixpoints of entire functions},
	JOURNAL = {Math. Z.},
	FJOURNAL = {Mathematische Zeitschrift},
	VOLUME = {104},
	YEAR = {1968},
	PAGES = {252--256},
}

@article {BakerKotusLu_I,
	AUTHOR = {Baker, I. N. and Kotus, J. and  L\"{u}, Y.},
	TITLE = {Iterates of meromorphic functions. {I}},
	JOURNAL = {Ergodic Theory Dynam. Systems},
	FJOURNAL = {Ergodic Theory and Dynamical Systems},
	VOLUME = {11},
	YEAR = {1991},
	NUMBER = {2},
	PAGES = {241--248},
}

@article {PUZ91,
	AUTHOR = {Przytycki, F. and Urba\'{n}ski, M. and Zdunik, A.},
	TITLE = {Harmonic, {G}ibbs and {H}ausdorff measures on repellers for
	holomorphic maps. {II}},
	JOURNAL = {Studia Math.},
	FJOURNAL = {Polska Akademia Nauk. Instytut Matematyczny. Studia
	Mathematica},
	VOLUME = {97},
	YEAR = {1991},
	NUMBER = {3},
	PAGES = {189--225},

}

@article {Taixés1,
	AUTHOR = {Fagella, N. and Jarque, X. and Taix\'{e}s, J.},
	TITLE = {On connectivity of {J}ulia sets of transcendental meromorphic
	maps and weakly repelling fixed points. {I}},
	JOURNAL = {Proc. Lond. Math. Soc. (3)},
	FJOURNAL = {Proceedings of the London Mathematical Society. Third Series},
	VOLUME = {97},
	YEAR = {2008},
	NUMBER = {3},
	PAGES = {599--622},
}

@article {Taixés2,
	AUTHOR = {Fagella, N. and Jarque, X. and Taix\'{e}s, J.},
	TITLE = {On connectivity of {J}ulia sets of transcendental meromorphic
	maps and weakly repelling fixed points {II}},
	JOURNAL = {Fund. Math.},
	FJOURNAL = {Fundamenta Mathematicae},
	VOLUME = {215},
	YEAR = {2011},
	NUMBER = {2},
	PAGES = {177--202},
}

@article {BergweilerEremenko,
	AUTHOR = {Bergweiler, W. and Eremenko, A.},
	TITLE = {On the singularities of the inverse to a meromorphic function
	of finite order},
	JOURNAL = {Rev. Mat. Iberoamericana},
	FJOURNAL = {Revista Matem\'{a}tica Iberoamericana},
	VOLUME = {11},
	YEAR = {1995},
	NUMBER = {2},
	PAGES = {355--373},
}

@PHDTHESIS{Bolsch-thesis, 
	author =       {A. Bolsch}, 
	title =        {Iteration of meromorphic functions with countably many essential singularities}, 
	school =       {Technischen Universität Berlin}, 
	year =         {1997}, 
}

@PHDTHESIS{Iversen-thesis, 
	author =       {F. Iversen}, 
	title =        {Recherches sur les fonctions inverses des fonctions méromorphes}, 
	school =       {Helsingfors}, 
	year =         {1914}, 
}

@article {BakerDominguezHerring2,
	AUTHOR = {Baker, I. N. and Dom\'{\i}nguez, P. and Herring, M. E.},
	TITLE = {Functions meromorphic outside a small set: completely
	invariant domains},
	JOURNAL = {Complex Var. Theory Appl.},
	FJOURNAL = {Complex Variables. Theory and Application. An International
	Journal},
	VOLUME = {49},
	YEAR = {2004},
	NUMBER = {2},
	PAGES = {95--100},
}

@article {Dominguez3,
	AUTHOR = {Dom\'{\i}nguez, P. and Montes de Oca, M. A. and
	Sienra, G. J. F.},
	TITLE = {Extended escaping set for meromorphic functions outside a
	countable set of transcendental singularities},
	JOURNAL = {Ann. Polon. Math.},
	FJOURNAL = {Annales Polonici Mathematici},
	VOLUME = {129},
	YEAR = {2022},
	NUMBER = {1},
	PAGES = {25--41},
}

@article {Dominguez4,
	AUTHOR = {Dom\'{\i}nguez, P.},
	TITLE = {Residual {J}ulia sets for meromorphic functions with countably
	many essential singularities},
	JOURNAL = {J. Difference Equ. Appl.},
	FJOURNAL = {Journal of Difference Equations and Applications},
	VOLUME = {16},
	YEAR = {2010},
	NUMBER = {5-6},
	PAGES = {519--522},
}

@article {EvdoridouMartiPeteSixsmith,
	AUTHOR = {Evdoridou, V. and Mart\'{\i}-Pete, D. and Sixsmith,
	D. J.},
	TITLE = {On the connectivity of the escaping set in the punctured
	plane},
	JOURNAL = {Collect. Math.},
	FJOURNAL = {Collectanea Mathematica},
	VOLUME = {72},
	YEAR = {2021},
	NUMBER = {1},
	PAGES = {109--127},
		shorthand={EMS21},
}

@misc{huang2022connectivity,
	title={Connectivity of Fatou Components of Meromorphic Functions}, 
	author={J. Huang and C. Wu and J.-H. Zheng},
	year={2022},
	eprint={2211.14258},
	archivePrefix={arXiv},
	primaryClass={math.CV}
}

@misc{ivrii2023inner,
	title={Inner Functions, Composition Operators, Symbolic Dynamics and Thermodynamic Formalism}, 
	author={O. Ivrii and M. Urbański},
	year={2023},
	eprint={2308.16063},
	archivePrefix={arXiv},
	primaryClass={math.DS}
}

@InProceedings{Keen,
	author={Keen, L.},
	title={Dynamics of holomorphic self-maps of $ \mathbb{C}^* $},
	booktitle={Holomorphic Functions and Moduli I},
	year={1988},
	publisher={Springer US},
	address={New York, NY},
	pages={9--30},
}

\end{document}